\documentclass[leqno,10pt]{amsart}
\RequirePackage{fix-cm}

\usepackage{amssymb,amsmath,amsthm}
\usepackage{faktor} 
\usepackage[latin1]{inputenc}
\usepackage{graphicx,color}
\usepackage[dvipsnames]{xcolor}

\usepackage[toc,page]{appendix}
\usepackage{comment}
\usepackage[colorlinks,citecolor=blue]{hyperref}

\usepackage[left=2cm,top=2.5cm,bottom=2.5cm,right=2cm]{geometry}
\usepackage[english]{babel}
\usepackage[colorlinks]{hyperref}

\usepackage{graphicx}
\usepackage{subcaption}

\usepackage{tikz} 
\usetikzlibrary{quotes,angles}
\usepackage{pgfplots}
\usepgfplotslibrary{groupplots}
\usepgflibrary{arrows}
\usetikzlibrary{calc}
\usetikzlibrary{intersections}
\usepgfplotslibrary{fillbetween}
\usepackage{tikz-cd} 

\def \R {{\mathbb R}}
\def \N {{\mathbb N}}
\def \Z {{\mathbb Z}}

\def \al {\alpha}
\def \la {\lambda}
\def \ve {\varepsilon}
\def \vt {\vartheta}
\def \vp {\varphi}

\newcommand{\cerchio}{\mathbb{S}^{1}}
\newcommand{\sset}{\subseteq}
\newcommand{\uguale}{\stackrel{.}{=}}
\newcommand{\Wzero}{W^0} 
\newcommand{\Veps}{V^\ve} 
\newcommand{\homx}{\hat{x}_\xi} 
\newcommand{\homy}[2]{\hat{y}_{#2}(#1)}
\newcommand{\homz}{\hat{z}_\xi} 
\newcommand{\wconv}{\rightharpoonup} 
\newcommand{\Mod}[1]{\ (\mathrm{mod}\ #1)} 
\newcommand{\extarc}[4]{y_{ext}(#1;p_{#2},p_{#3};#4)} 
\newcommand{\intarc}[4]{y_{int}(#1;p_{#2},p_{#3};#4)} 
\newcommand{\dextarc}[4]{\dot{y}_{ext}(#1;p_{#2},p_{#3};#4)} 
\newcommand{\dintarc}[4]{\dot{y}_{int}(#1;p_{#2},p_{#3};#4)} 

\newtheorem{theorem}{Theorem}[section]
\newtheorem{lemma}[theorem]{Lemma}

\newtheorem{definition}[theorem]{Definition}
\newtheorem{proposition}[theorem]{Proposition}
\newtheorem{remark}[theorem]{Remark}
\newtheorem{corollary}[theorem]{Corollary}

\title{Symbolic dynamics for the anisotropic $N$-centre problem at negative energies}
\author{Vivina Barutello, Gian Marco Canneori and Susanna Terracini}

\address{Dipartimento di Matematica ``G. Peano''
	\newline\indent
	Universit\`a degli Studi di Torino
	\newline\indent
	 Via Carlo Alberto 10, 10123 Torino, Italy\\}
\email{vivina.barutello@unito.it}
\email{gianmarco.canneori@unito.it}
\email{susanna.terracini@unito.it}

\date{\today} 



\begin{document}

\begin{abstract}
The planar $N$-centre problem describes the motion of a  particle moving in the plane under the action of the force fields of $N$ fixed attractive centres:
\[
	\ddot{x}(t)=\sum_{j=1}^N\nabla V_j(x-c_j)
\]

 In this paper we prove symbolic dynamics at slightly negative energy for an $N$-centre problem where the potentials $V_j$ are positive, anisotropic and homogeneous of degree $-\alpha_j$:
\[
V_j(x)=|x|^{-\al_j}V_j\left(\frac{x}{|x|}\right).
\]
The proof is based on a broken geodesics argument and trajectories are extremals of the Maupertuis' functional.

Compared with the classical $N$-centre problem with Kepler potentials, a major difficulty arises from the lack of a regularization of the singularities.  We will consider both the collisional dynamics and the non collision one. Symbols describe geometric and topological features of the associated trajectory.
\end{abstract}
\maketitle

\section{Introduction and main results}

The aim of this paper is to describe the onset of chaos in the case of an  $N$-centre problem involving anisotropic attractive Kepler-like potentials. Anisotropic homogeneous singular potentials arise in the reduction by symmetry of the symmetric $N$-body problem of Celestial Mechanics, e.g., in the isosceles $3$-body problem.  Another relevant physical example occurs in the atomic theory of semiconductor crystals of silicon or germanium, due to the presence of impurities. An additional nuclear charge in the donor impurity causes a deformation that breaks the symmetry of the atoms lattice, ultimately resulting in an anisotropy of the mass tensor; this anisotropy can be referred to the potential as well  \cite[Chap. 11]{GutBook}. The case of one anisotropic attractive centre has been extensively explored in the Celestial Mechanics literature and in other physical systems, also in the search for connections between classical and quantum mechanics. As Gutzwiller highlighted in a series of pioneering papers  (\cite{Gutzwiller73,Gutzwiller77,Gutzwiller81}), compared to the isotropic Kepler problem, the anisotropic case may lose its integrability and  present chaotic trajectories. Moreover, collisions cease to be regularizable, as highlighted by Devaney in \cite{DevJDE}. Because of their homogeneity,  $-\alpha$-anisotropic potentials and their collision trajectories have been extensively studied  in the literature \cite{McG1974,DevInvMath1978,Deva1980, DevProgMath1981,BaSe,Saari84, BFT2,BTVplanar,BTV,BHPT,BCTmin} by analytical and geometrical methods.  While the problem of $N$-centres with isotropic Keplerian potentials is a great classic in the recent literature of Celestial Mechanics (cf. the end of this section), as far as we know this is the first work dealing with the case of anisotropic potentials. 

We consider an anisotropic planar $N$-centre problem, where we associate with each centre a positive, anisotropic potential $V_j\in\mathcal{C}^2(\R^2\setminus\{0\})$, homogeneous of degree $-\al_j$:
\[
V_j(x) =|x|^{-\al_j}V_j\left(\frac{x}{|x|}\right)=r^{-\al_j}U_j(\theta)\;,
\]
where $(r,\theta)$ are polar coordinates and $U_j\uguale V_j|_{\cerchio}$. Denoting by $c_1,\ldots,c_N\in\R^2$  the positions of the $N\geq 2$ centres, we introduce the total potential 
\[
V(x)=\sum_{j=1}^NV_j(x-c_j)=\sum_{j=1}^N|x-c_j|^{-\al_j}V_j\left(\frac{x-c_j}{|x-c_j|}\right).
\]
so that the equation of motion reads as
\begin{equation}\label{eq:moto}
	\ddot{x}(t)=\nabla V(x(t)),
\end{equation}
where $x=x(t)$ represents the position of the moving particle at time $t\in\R$.  The associated Hamiltonian being
\[
H(x,v)=\frac{1}{2}|v|^2-V(x),
\]
every solution of \eqref{eq:moto} verifies the energy conservation law
\begin{equation}\label{eq:energy}
	\frac{1}{2}|\dot{x}|^2-V(x)=-h.
\end{equation}

Given $h>0$, we are interested in those solutions of equation \eqref{eq:moto} which are confined in the 3-dimensional negative energy shell
\[
\mathcal{E}_h=\left\lbrace (x,v)\in(\R^2\setminus\{c_1,\ldots,c_N\})\times\R^2:\ \frac12|v|^2-V(x)=-h\right\rbrace,
\]
which projects on the configuration space onto the bounded \emph{Hill's region} 
\begin{equation}\label{def:hill_region}
	\mathcal{R}_h=\{x\in\R^2\setminus\{c_1,\ldots,c_N\}:\,V(x)\geq h\}.
\end{equation}

We are going to investigate the presence of chaotic behaviour at negative energies, through the detection of a subsystem displaying a symbolic dynamics.  In order to give rigour to these concepts, we need to recall some basic definitions. Consider a finite set $\mathcal{S}$, with at least two elements, endowed with the discrete metric ($\rho(s_j,s_k)=\delta_{jk}$, where $\delta_{jk}$ is the Kronecker delta and $s_j,s_k\in\mathcal{S}$). Consider the set of bi-infinite sequences of elements of $\mathcal{S}$
\[
\mathcal{S}^\Z\uguale\{(s_m)_{m\in\Z}:\,s_m\in \mathcal{S},\ \mbox{for all}\ m\in\Z\}
\]
and make it a metric space with the distance
\[
d((s_m),(t_m))\uguale\sum\limits_{m\in\Z}\frac{\rho(s_m,t_m)}{2^{|m|}},
\]
defined for every $(s_m),(t_m)\in \mathcal{S}^\Z$. Introduce also the \emph{Bernoulli right shift} as the discrete dynamical system $(\mathcal{S}^\Z,T_r)$, where
\[
\begin{aligned}
	T_r\colon &\mathcal{S}^\Z\longrightarrow \mathcal{S}^\Z \\
	&(s_m)\mapsto T_r((s_m))\uguale (s_{m+1}),
\end{aligned}
\]

The main features of the discrete dynamical system $(\mathcal{S}^\Z,T_r)$ are paradigmatic of a chaotic behaviour (see for instance \cite{Teschl_ode}): 
\begin{itemize}
	\item $(\mathcal{S}^\Z,T_r)$ has a dense countable set of periodic points (all the periodic sequences are periodic points);
	\item $(\mathcal{S}^\Z,T_r)$ displays high sensitivity with respect initial data, i.e., if we define as $T_r^k$ the $k$-th iteration of the Bernoulli shift, we have that for any $\varrho>0$ there exist two arbitrarily close sequences $(s_m),(t_m)\in\mathcal{S}^\Z$ such that
	\[
	\sup\limits_{k\in\Z}d(T_r^k((s_m)),T_r^k((t_m)))\geq\varrho;
	\]
	\item the previous property actually holds for a big set of initial data, therefore the dynamical system $(\mathcal{S}^\Z,T_r)$ has positive topological entropy.
\end{itemize}

The Bernoulli shift is our reference dynamical system in order to describe complex behaviour of solutions to the anisotropic $N$-centre problem. 

\begin{definition}
	Let $\mathcal{S}$ be a finite set, $\mathcal{E}$ be a metric space and $\mathfrak{R}\colon\mathcal{E}\to\mathcal{E}$ be a continuous map. Then, we say that the dynamical system $(\mathcal{E},\mathfrak{R})$ has a symbolic dynamics with set of symbols $\mathcal{S}$ if there exist a subset $\Pi\sset\mathcal{E}$ which is invariant through $\mathfrak{R}$ and a continuous and surjective map $\pi\colon\Pi\to\mathcal{S}^\Z$ such that the diagram 
	\[
	\begin{tikzcd}
		\Pi \arrow{r}{\mathfrak{R}} \arrow{d}{\pi} & \Pi \arrow{d}{\pi} \\
		\mathcal{S}^\Z \arrow{r}{T_r}	& \mathcal{S}^\Z
	\end{tikzcd}
	\]
	commutes. In other words, we are saying that the map $\mathfrak{R}|_{\Pi}$ is topologically semi-conjugate to the Bernoulli right shift $T_r$ in the metric space $(\mathcal{S}^\Z,d)$.	
\end{definition}

In addition to identifying the presence of a symbolic dynamic through semi-conjugation with the Bernoulli shift, we are very interested in giving a physical interpretation to the symbols, in terms of the geometric characteristics of the solutions. To both ends, we need to go further into the analysis of the elements of our system. At first, without loss of generality, we can assume 
\[
\al_1\leq\al_2\leq\ldots\leq\al_N.
\]

Clearly, the smallest degree of homogeneity $\al_1$ leads the overall potential at infinity. Hence, assuming that $\al_1=\al_2=\ldots=\al_k$ for some $1\leq k<N$, and denoting $\al\uguale\al_1$, it is convenient to gather all the $-\al$-homogeneous potentials in this way

\begin{equation}\label{def:potential}
	V(x)=W(x)+\sum\limits_{j=k+1}^N| x-c_j|^{-\al_j}V_j\left(\frac{x-c_j}{\lvert x-c_j\rvert}\right)\;,
\end{equation} 
where
\[
W(x)\uguale\sum\limits_{i=1}^kV_i(x-c_i)=\sum\limits_{i=1}^k| x-c_i|^{-\al}V_i\left(\frac{x-c_i}{\lvert x-c_i\rvert}\right),
\]
so that $W\in\mathcal{C}^2(\R^2\setminus\{c_1,\ldots,c_k\})$. We set:
\[
U(\vt)\uguale \sum\limits_{i=1}^kU_i(\vt),\quad\mbox{for}\ \vt\in\cerchio,
\]
so that $W(x)\simeq |x|^{-\alpha}U\left(\frac{x}{\lvert x\rvert}\right)$ when $|x|>>1$.

Any critical point of the potential $U$ will be termed a \emph{central configuration}. Our basic assumption on $V$ is about the number of its non-degenerate minimal central configurations.
\begin{equation}\label{hyp:V}
	\tag{$V$} 
	\fontsize{9pt}{\baselineskip}
	\begin{cases}
		\al<2; \\
		\exists\,(\vt_l^*)_{l=1}^{m}\sset\cerchio\,:\,\ \forall\ l=1,\ldots,m,\ m>0,\ U''(\vt_l^*)>0,\  U(\vt)\geq U(\vt_l^*)>0,\ \forall\ \vt\in\cerchio.\\		
	\end{cases}
\end{equation}


\begin{remark}\label{rem:en_bound}
	From now on, without loss of generality, we will assume that $\max\limits_j|c_j|\leq1$ and we define
	\[
	\mathfrak{m}=\min_{j=1,\ldots,N} \min_{\cerchio}U_j,
	\]
	so that, for every $x\in\R^2\setminus\{c_1,\ldots,c_N\}$, we have
	\[
	\begin{aligned}
		V(x)\geq \mathfrak{m}\sum\limits_{j=1}^N|x-c_j|^{-\al_j}\geq \frac{\mathfrak{m}}{(|x|+1)^{\al}}.
	\end{aligned}
	\]
	Define \(
	\tilde{h}\uguale\mathfrak{m}/{2^\al}
	\): then, for every $h\in(0,\tilde{h})$, the Hill's region $\mathcal{R}_h$ contains the unit ball where the centres lie (see Figure \ref{fig:hill}).
\end{remark}

\begin{figure}
	\centering
	\includegraphics{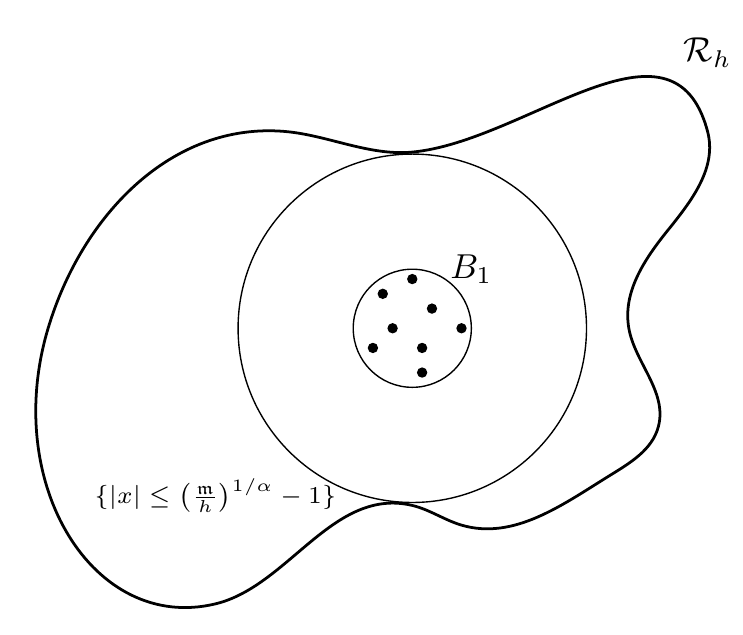}	
	\caption{An example of Hill's region for the anisotropic $N$-centre problem that includes a ball of radius greater than 1 (see Remark \ref{rem:en_bound}).}
	\label{fig:hill}
\end{figure}

For the purposes of this work, we need to take into account different definitions of solutions, allowing for collisions (cf. also \cite{BoOrZh,BoDaPa2}). 
\begin{definition}
We define a \emph{non-collision solution} of \eqref{eq:moto} as a $\mathcal{C}^2$-function $x\colon J\sset\R\to\R^2$ such that $x(t)\neq c_j$ for every $t\in J$ and for every $j=1,\ldots,N$ and that solves \eqref{eq:moto} and \eqref{eq:energy} in the classical sense. 
We say that $x$ defined on the same interval $J$ is a \emph{collision solution} of \eqref{eq:moto} if $x\in H^1(J)$ and there exists a collision instants set $T_c(x)\sset J$ such that:
\begin{itemize}
	\item the set $T_c(x)$ has null measure;
	\item for any $t\in T_c(x)$, it holds $x(t)=c_j$, for some $j=1,\ldots,N$;
	\item for any $(a,b)\sset J\setminus T_c(x)$, the restriction $x|_{(a,b)}$ is a non-collision solution of \eqref{eq:moto};
	\item for every $t\in J\setminus T_c(x)$, $x(t)$ verifies the energy equation \eqref{eq:energy}.
\end{itemize}
\end{definition}

In what follows, an infinite number of periodic solutions of \eqref{eq:moto} confined in the energy shell $\mathcal{E}_h$ will be provided through a variational method and we will relate the occurrence of collisions  to the homogeneity degrees of every potential $V_j$. 
%
%
%

Motivated by this, our first main result states the existence of a (possibly collisional) symbolic dynamics in the presence of at least two centres and two minimal non-degenerate central configurations for $W$.

\begin{theorem}\label{thm:symb_dyn}
	Assume that $N\geq 2$, $m\geq 2$ and consider a function $V\in\mathcal{C}^2(\R^2\setminus\{c_1,\ldots,c_N\})$ defined as in \eqref{def:potential} and satisfying \eqref{hyp:V}. There exists $h^*>0$ and a finite set of symbols $\mathcal{S}$ such that, for every $h\in(0,h^*)$, there exist a subset $\Pi_h$ of the energy shell $\mathcal{E}_h$, a (possibly collisional) first return map $\mathfrak{R}\colon\Pi_h\to\Pi_h$ and a continuous and surjective map $\pi\colon\Pi_h\to\mathcal{S}^\Z$ such that the diagram 
	\[
	\begin{tikzcd}
		\Pi_h \arrow{r}{\mathfrak{R}} \arrow{d}{\pi} & \Pi_h \arrow{d}{\pi} \\
		\mathcal{S}^\Z \arrow{r}{T_r}	 & \mathcal{S}^\Z
	\end{tikzcd}
	\]
	commutes. In other words, for $h$ sufficiently small, the anisotropic $N$-centre problem at energy $-h$ admits a symbolic dynamics.
\end{theorem}

It is worthwhile noticing that the symbols here represent outer arcs shadowing break homothetic trajectories of the potential $r^{-\alpha}U(\theta)$, indexed on the set
\begin{equation}\label{eq:Xi}
\Xi\uguale\{\vt^*\in\cerchio:\ U'(\vt^*)=0\ \mbox{and}\ U''(\vt^*)>0\}=\{\vt_1^*,\ldots,\vt_m^*\}.
\end{equation}
This explains why we need at least two central configurations in Theorem \ref{thm:symb_dyn}. Despite the mildness of its assumptions, this theorem does not take into account the problem of collisions with the centres. As we shall use a minimization argument with topological constraints, we need a suitable argument to rule out the occurrence of collisions for topologically constrained minimizers. Following the strategy already introduced in \cite{BTV,BTVplanar} for the anisotropic Kepler problem, this step requires some additional  assumption on the homogeneity degrees of every potential $V_j$ at the centres $c_j$. As explained later on in Lemma \ref{lem:BTV}, there are thresholds $\bar{\al}_j\in(0,2)$ which depend only on the restricted potentials $U_j$ and on its non-degenerate central configurations, over which collision-less trajectories can be provided. For this purpose we introduce a further hypothesis on the restrictions $U_j$ of the $V_j$'s to the unit sphere  as follows
\begin{equation}\label{hyp:V_al}
	\tag{$V_\al$}
	\fontsize{9pt}{\baselineskip}
	\begin{cases}
	\forall\,j=1,\ldots,N\ \exists\,\vt_j\in\cerchio:\,U_j(\vt)\geq U_j(\vt_j)>0,\ \forall\,\vt\in\cerchio,\ U_j''(\vt_j)>0\\
	\forall\,j=1,\ldots,N,\ \al_j>\bar{\al}_j=\bar{\al}_j(U_j,\vt_j,\vt_j+4\pi)
	\end{cases}.
\end{equation}
Moreover, in order to give a characterization of the symbols naturally related with the collision-less trajectories of \eqref{eq:moto}, we shall adopt the strategy introduced in \cite{ST2012}, joining inner and outer arcs through a finite dimensional reduction. To this aim,
we need to introduce some further notations. As before, symbols to label the outer arcs are chosen to be the non-degenerate minimal central configurations of the $-\al$-homogeneous component of $V$.
Next, in order to parametrise the inner arcs, we consider all the possible partitions of the $N$ centres in two disjoint non-empty sets, which are exactly \label{intro:partitions} $2^{N-1}-1$,
and we denote the set of such partitions as 
\[
\mathcal{P}=\{P_j:\ j=0,\ldots,2^{N-1}-2\}.
\]
Now, being $\Xi$ defined in \eqref{eq:Xi}, we collect all possible choices in the set
\[
\mathcal{Q}=\mathcal{P}\times\Xi=\{Q_j:\ j=0,\ldots,m(2^{N-1}-1)-1\}.
\]

\begin{remark}\label{rem:division}
	For $n\in\N_{\geq 1}$ and $(Q_{j_0},\ldots,Q_{j_{n-1}})\in\mathcal{Q}^n$, consider the element $Q_{j_k}$ for some $k\in\{0,\ldots,n-1\}$. It is useful to introduce the quotient and the remainder of the division of $j_k$ by $m$ in this way
	\begin{equation}\label{eq:division}
		j_k=l_km+r_k,
	\end{equation}
	so that we have $Q_{j_k}=(P_{l_k},\vt_{r_k}^*)$.
\end{remark}


Adopting these notations, we can now state the following result on the existence of collision-less periodic solutions of \eqref{eq:moto} in negative energy shells.
\begin{theorem}\label{th:main1}
	Assume that $N\geq 3$ and $m\geq 1$ or $N\geq 2$ and $m\geq 2$. Consider a potential $V\in\mathcal{C}^2(\R^2\setminus\{c_1,\ldots,c_N\})$ defined as in \eqref{def:potential} and satisfying \eqref{hyp:V}-\eqref{hyp:V_al}. There exists $\bar h>0$ such that, for every $h\in(0,\bar{h})$, $n\in\N_{\geq1}$ and $(Q_{j_0},\ldots,Q_{j_{n-1}})\in\mathcal{Q}^n$, there exists a periodic collision-less and self-intersections-free solution $x=x(Q_{j_0},\ldots,Q_{j_{n-1}};h)$ of \eqref{eq:moto} satisfying \eqref{eq:energy}, which depends on $(Q_{j_0},\ldots,Q_{j_{n-1}})$ in this way: there exists $\bar R=\bar{R}(h)>0$ such that the solution $x$ crosses $2n$ times the circle $\partial B_{\bar R}$ in one period, at times $(t_s)_{s=0}^{2n-1}$, in such a way that, according to $\eqref{eq:division}$, for any $k=0,\ldots,n-1$:
	\begin{itemize}
		\item in the interval $(t_{2k},t_{2k+1})$ the solution stays outside $B_{\bar R}$ and there exists a neighbourhood $\mathcal{U}_{r_k}=\mathcal{U}(\bar R e^{i\vt_{r_k}^*})$ on $\partial B_{\bar{R}}$ such that
		\[
		x(t_{2k}),x(t_{2k+1})\in\mathcal{U}_{r_k};
		\]
		\item in the interval $(t_{2k+1},t_{2k+2})$ the solution stays inside $B_{\bar R}$ and separates the centres according to the partition $P_{l_k}$.
	\end{itemize}
\end{theorem}

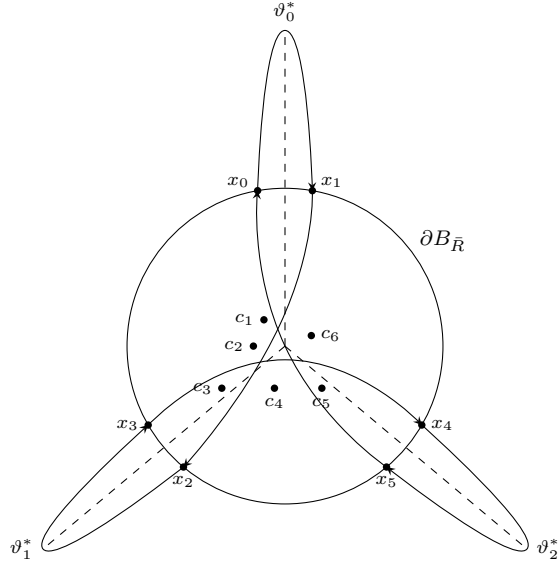
\begin{figure}
	\centering
	\begin{tikzpicture}[scale=1.4]
		\coordinate (0) at (0,0);
		\coordinate (p0) at (-0.26,1.477);
		\coordinate (p1) at (0.26,1.477);   		
		\coordinate (p2) at (-0.964,-1.149);   		
		\coordinate (p3) at (-1.299,-0.75); 
		\coordinate (p4) at (1.299,-0.75);
		\coordinate (p5) at (0.964,-1.149);
		\coordinate (c1) at (-0.2,0.25);
		\coordinate (c2) at (-0.3,0);   		
		\coordinate (c3) at (-0.6,-0.4);    		
		\coordinate (c4) at (-0.1,-0.4); 
		\coordinate (c5) at (0.35,-0.4);
		\coordinate (c6) at (0.25,0.1);
		
		\draw (0,0) circle (1.5cm);
		
		
		\draw[dashed] (0)--(0,3);
		\draw[dashed] (0)--(2.298,-1.928);
		\draw[dashed] (0)--(-2.298,-1.928);
		
		\fill (p0)
		circle[radius=1.0pt];  		
		\fill (p1)
		circle[radius=1.0pt];
		\fill (p2)
		circle[radius=1.0pt];
		\fill (p3)
		circle[radius=1.0pt];
		\fill (p4)
		circle[radius=1.0pt];
		\fill (p5)
		circle[radius=1.0pt];
		\fill (c1)
		circle[radius=1.0pt];
		\fill (c2)
		circle[radius=1.0pt];
		\fill (c3)
		circle[radius=1.0pt];
		\fill (c4)
		circle[radius=1.0pt];
		\fill (c5)
		circle[radius=1.0pt];
		\fill (c6)
		circle[radius=1.0pt];

		
		\draw[-stealth] plot [smooth, tension=1.5] coordinates {(p0) (0,3) (p1)}; 
		\draw[-stealth] plot [smooth, tension=1.5] coordinates {(p2) (-2.298,-1.928) (p3)};
		\draw[,-stealth] plot [smooth, tension=1.5] coordinates {(p4) (2.298,-1.928) (p5)};
		
		
		\draw[-stealth] plot [smooth, tension=1] coordinates {(p1) (-0.05,0.2)  (p2)};           
		\draw[-stealth] plot [smooth, tension=1] coordinates {(p3) (0,-0.13) (p4)};
		\draw[-stealth] plot [smooth, tension=1] coordinates {(p5)  (0,0) (p0)};
		
		\node[left] at (-0.25,1.56) {\scriptsize $x_0$};
		\node[right] at (0.25,1.56) {\scriptsize $x_1$};
		\node[below] at (p2) {\scriptsize $x_2$};
		\node[left] at (p3) {\scriptsize $x_3$};
		\node[right] at (p4) {\scriptsize $x_4$};   		
		\node[below] at (p5) {\scriptsize $x_5$};
		\node[above] at (0,3) {\scriptsize $\vt_0^*$};
		\node[left] at (-2.298,-1.928) {\scriptsize $\vt_1^*$};
		\node[right] at (2.298,-1.928) {\scriptsize $\vt_2^*$};
		\node[left] at (c1) {\scriptsize $c_1$};
		\node[left] at (c2) {\scriptsize $c_2$};
		\node[left] at (c3) {\scriptsize $c_3$};
		\node[below] at (c4) {\scriptsize $c_4$};
		\node[below] at (c5) {\scriptsize $c_5$};
		\node[right] at (c6) {\scriptsize $c_6$};
		\node at (1.5,1) {\footnotesize $\partial B_{\bar{R}}$};
	\end{tikzpicture}
	\caption{An example of classical periodic solution provided in Theorem \ref{th:main1}}
\end{figure}

As a consequence of the existence of collision-less periodic solutions, we are going to show that, assuming \eqref{hyp:V}-\eqref{hyp:V_al}, our system has a collision-free symbolic dynamics with set of symbols $\mathcal{Q}$. Differently from Theorem \ref{thm:symb_dyn}, in this case it is possible to include  also the case $m=1$, provided $N\geq 3$, so to have at least 2 elements in $\mathcal{Q}$. In facts, as highlighted in Theorem \ref{thm:symb_dyn2}, we can cover also the case $m=1$ and $N=2$.

\begin{theorem}\label{cor:symb} 
	In the same setting of Theorem \ref{th:main1}, take $h\in(0,\bar{h})$, with $\bar{h}>0$ therein defined. Then, there exists a subset $\Pi_h$ of the energy shell $\mathcal{E}_h$, a first return map $\mathfrak{R}\colon\Pi_h\to\Pi_h$ and a continuous and surjective map $\pi\colon\Pi_h\to\mathcal{Q}^\Z$, such that the diagram
	\[
	\begin{tikzcd}
		\Pi_h \arrow{r}{\mathfrak{R}} \arrow{d}{\pi} & \Pi_h \arrow{d}{\pi} \\
		\mathcal{Q}^\Z \arrow{r}{T_r}	 & \mathcal{Q}^\Z
	\end{tikzcd}
	\]
	commutes. In other words, for any $h$ sufficiently small, the anisotropic $N$-centre problem at energy $-h$ admits a collision-less symbolic dynamics, with sets of symbols $\mathcal{Q}$.
\end{theorem}

The last result of this work concerns a particular case of the $2$-centre problem, driven by a potential $V$ defined as in \eqref{def:potential}. We believe that this case deserves to be highly remarked since, for instance, due to the integrability, no results in this direction can be proven for Keplerian radial potential (see \cite{ST2012}), thus revealing a peculiar property of the anisotropic setting. In particular, in order to complete the treatment of all the possible cases not included in Theorems \ref{thm:symb_dyn}-\ref{cor:symb}, we will assume that $N=2$ and $m=1$. In this situation, the alphabet $\mathcal{Q}$ defined above would consist of a unique symbol and no symbolic dynamics could be proven to exist with such notation. For this reason, as a set of symbols, we will take the set $\mathcal{B}\uguale\{c_1,c_2\}$ of the 2 centres.  

\begin{theorem}\label{thm:symb_dyn2}
	Assume that $N=2$, $m=1$ and consider a function $V\in\mathcal{C}^2(\R^2\setminus\{{c_1,c_2}\})$, defined as in \eqref{def:potential} and satisfying \eqref{hyp:V}-\eqref{hyp:V_al}. There exists $\tilde{h}>0$, a set of two symbols $\mathcal{B}$ such that, for every $h\in(0,\tilde{h})$, there exist a subset $\Pi_h$ of the energy shell $\mathcal{E}_h$, a first return map $\mathfrak{R}\colon\Pi_h\to\Pi_h$ and a continuous and surjective map $\pi\colon\Pi_h\to\mathcal{B}^\Z$, such that the diagram 	
	\[
	\begin{tikzcd}
		\Pi_h \arrow{r}{\mathfrak{R}} \arrow{d}{\pi} & \Pi_h \arrow{d}{\pi} \\
		\mathcal{B}^\Z \arrow{r}{T_r}	 & \mathcal{B}^\Z
	\end{tikzcd}
	\]
	commutes. In other words, for any $h$ sufficiently small, the anisotropic 2-centre problem at energy $-h$ admits a collision-less symbolic dynamics, with set of symbols $\mathcal{B}$.
\end{theorem}

As a consequence of our main results, the anisotropic $N$-centre problem is not integrable in slightly negative energy shells. Let us put our results in a context and compare them with the known case of the $N$-centres problem  with Kepler potentials. Non integrability and chaotic behaviour at \emph{non negative} (possibly large positive) energies has been proved by variational methods starting from \cite{Bol,BolNeg,BoKo,KK1992, Kn2002,KnTai}. We also observe that in \cite{KnTai2} the authors proved that, over a high energy threshold, the $N$-centre problem is completely integrable through $\mathcal{C}^\infty$-integrals both in $\R^2$ and $\R^3$. Compared to ours, the case of non negative energy is considerably simple, since the Hill's regions have no boundary. The negative energy case of the $N$-centres with Keplerian potentials has been tackled in \cite{BolNeg2,Dim} as a perturbation of the $2$-centre problem, and in \cite{ST2012} in full generality in the planar case and displays symbolic dynamics as well. We conclude this discussion on the $N$-centre problem observing that it is also possible to find $T$-periodic, parabolic and heteroclinic solutions without any information on the energy of the system. Recent interesting contributions in this direction can be find in \cite{BoDaTe,BoDaPa,Cas,Yu,ChenYu20} and in \cite{CannPrep} for anisotropic potentials. It has to be noticed that an additional difficulty specific of the anisotropic case is that the singularities can not be regularized.

\subsection{Outline of the proof}

The key idea is to consider a different $N$-centre problem starting from the dynamical system \eqref{eq:moto} and the energy equation \eqref{eq:energy}. Defining a suitable rescaled version of potential $V$, we end up with the problem
\begin{equation}\label{pb:rescaled}
	\begin{cases}
		\ddot{y}(t)=\nabla\Veps(y(t)) \\
		\frac12|\dot{y}(t)|^2-\Veps(y(t))=-1,
	\end{cases}
\end{equation}
where $\ve=h^{1/\al}>0$ and $\Veps$ takes into account the rescaled centres $c_j'=\ve c_j$. In this way, all the new centres are confined in the ball $B_\ve(0)$ and collapse to the origin as the energy $h$ of the original problem becomes very small, since $\ve\to 0^+$ as $h\to 0^+$. It turns out that it is equivalent to look for periodic solutions of \eqref{eq:moto}-\eqref{eq:energy} and periodic solutions of \eqref{pb:rescaled}. Moreover, when $\ve$ becomes very small, outside a ball of radius $R\gg\ve$ and centred in the origin, the potential $V_\ve$ is a small perturbation of a suitable anisotropic Kepler-like potential. This fact, which, together with the previous discussion, is the content of Section \ref{sec:scaling}, allows us to split the proof of the main results in two steps: the construction of solution arcs outside and inside the ball $B_R(0)$. 

In Section \ref{sec:outer} we prove the existence of pieces of solutions of \eqref{pb:rescaled}, starting in $\partial B_R(0)$ and lying outside $B_R(0)$.

In Section \ref{sec:inner} we show how to build solution arcs which start in $\partial B_R$ and go through the centres without collisions. 

In Section \ref{sec:glueing} we glue together the pieces of solutions obtained in the previous sections, in order to obtain periodic solutions of \eqref{pb:rescaled} and thus of \eqref{eq:moto}-\eqref{eq:energy}.

In Section \ref{sec:symb_dyn} we use all the previous results to prove that, in all the situations described in Theorems \ref{thm:symb_dyn}-\ref{cor:symb}-\ref{thm:symb_dyn2}, a symbolic dynamics exists, using suitable alphabets.

\section{A useful rescaling}\label{sec:scaling}

Given $\ve>0$ and $y\in\R^2\setminus\{c_1,\ldots,c_N\}$, let us introduce the rescaled potential
\begin{equation}\label{def:scaling}
	\Veps(y)\uguale W^\ve(y)+\sum\limits_{j=k+1}^N\ve^{\al_j-\al}V_j(y-\ve c_j),
\end{equation}
where 
\[
W^\ve(y)\uguale\sum\limits_{i=1}^kV_i(y-\ve c_i).
\]
Notice that, with this notations and recalling that we have assumed that $\max|c_j|\leq 1$ (see Remark \ref{rem:en_bound}), the new centres $\ve c_j$ will be included inside the ball $B_\ve$.
\begin{proposition}\label{prop:scaling}
	Let $V\in\mathcal{C}^2(\R^2\setminus\{c_1,\ldots,c_N\})$ be defined as in \eqref{def:potential} and $x\in\mathcal{C}^2((a,b);\R^2)$ be a classical solution of
	\begin{equation}\label{pb:unpert}
		\begin{cases}
			\ddot{x}(t)=\nabla V(x(t)) \\
			\frac{1}{2}|\dot{x}(t)|^2-V(x(t))=-h,\quad h>0.
		\end{cases}
	\end{equation}
	Then, in the interval $(h^{\frac{\al+2}{2\al}}a,h^{\frac{\al+2}{2\al}}b)$, the function 
	\[
	y(t)\uguale h^{1/\al}x(h^{-\frac{\al+2}{2\al}}t)
	\]
	solves the problem
	\begin{equation}\label{pb:pert}
		\begin{cases}
			\ddot{y}(t)=\nabla \Veps(y(t)) \\
			\frac{1}{2}|\dot{y}(t)|^2-\Veps(y(t))=-1,
		\end{cases}
	\end{equation}
	where $\ve=h^{1/\al}$ and $\Veps$ is defined as in \eqref{def:scaling}.
	
	Conversely, if $y\in\mathcal{C}^2((c,d);\R^2)$ is a solution of \eqref{pb:pert} then, taking $h=\ve^\al$, the function
	\[
	x(t)\uguale h^{-1/\al}y(h^{\frac{\al+2}{2\al}}t)
	\]
	is a solution of \eqref{pb:unpert} in the interval $(h^{-\frac{\al+2}{2\al}}c,h^{-\frac{\al+2}{2\al}}d)$.
\end{proposition}
\begin{proof}
	It is an easy computation based on the homogeneity of the potentials involved.
\end{proof}

In the rest of this section we show that, outside a ball of radius $R>\ve>0$, if $\ve$ is sufficiently small problem \eqref{pb:pert} can be seen as a perturbation of a Kepler problem, driven by a sum of $-\al$-homogeneous potentials. We start by showing a limiting behaviour for $\Veps$ as $\ve\to 0^+$.
\begin{proposition}\label{prop:pert}
	Let $\delta>0$ and $\Veps$ be defined as in \eqref{def:scaling}. Then, there exists $\gamma>0$ such that, for every $y\in\R^2\setminus B_\delta$ 
	\[
	\Veps(y)=\Wzero(y)+O(\ve^\gamma)\quad\mbox{as}\ \ve\to 0^+,
	\]
	where $\Wzero$ is a $-\al$-homogeneous potential and, in particular, according to \eqref{def:scaling}
	\begin{equation}\label{def:wzero}
		\Wzero(y)=\sum\limits_{i=1}^kV_i(y)=|y|^{-\al}\sum\limits_{i=1}^kV_i\left(\frac{y}{|y|}\right).
	\end{equation}
	Moreover, the potential $\Veps$ is smooth with respect to $\ve$ and $\Veps\to\Wzero$ uniformly as $\ve\to 0^+$ on every compact subset of $\R^2\setminus\{0\}$.
\end{proposition}
\begin{proof} As a starting point, since $\ve\to 0^+$, we can assume $\delta>\ve$; moreover, if we fix $j\in\{1,\ldots,N\}$ and $|y|>\delta$, for every $\sigma\in\R$ we have
\[
	|y-\ve c_j|^{-\sigma}=|y|^{-\sigma}+O(\ve).
\]
In this way, for every $j\in\{1,\ldots,N\}$, as $\ve\to 0^+$ we can write
\[
	V_j\left(\frac{y-\ve c_j}{|y-\ve c_j|}\right)=V_j\left(\frac{y}{|y|}+O(\ve)\right)=V_j\left(\frac{y}{|y|}\right)+O(\ve),
\]
so that
	\[
	\begin{aligned}
		\Veps(y)&=\sum\limits_{i=1}^k| y-\ve c_i|^{-\al}V_i\left(\frac{y-\ve c_i}{| y-\ve c_i|}\right)+\sum\limits_{j=k+1}^N\ve^{\al_j-\al}|y-\ve c_j|^{-\al_j}V_j\left(\frac{y-\ve c_j}{|y-\ve c_j|}\right) \\
		&=|y|^{-\al}\sum\limits_{i=1}^k V_i\left(\frac{y}{|y|}\right)+\sum\limits_{j=k+1}^N\ve^{\al_j-\al}|y|^{-\al_j}V_j\left(\frac{y}{|y|}\right)+O(\ve) \\
		&=\Wzero(y)+O(\ve^\gamma),
	\end{aligned}	
	\]
	where $\gamma=\min\{1,\al_{k+1}-\al\}>0$.
	
	To conclude, the uniform convergence on compact subsets of $\R\setminus\{0\}$ is an easy consequence of the fact that the singularity set reduces to the origin as $\ve\to 0^+$.
\end{proof}

\begin{remark}\label{rem:wzero}
	Observe that the potential $\Wzero$ defined in \eqref{def:wzero} is singular in the origin, while the potential $W^\ve$ has multiple poles at $\ve c_1,\ldots,\ve c_k$. Thus, it turns out that assumption $\eqref{hyp:V}$ on page \pageref{hyp:V} requires that $\Wzero$ admits $m$ strictly minimal central configurations, for some $m>0$. Indeed, potential $\Wzero$ is exactly the profile limit of $W(x)$, introduced in \eqref{def:potential}, as $|x|>> 1$.
\end{remark}

To conclude this section, we notice that the energy bound found in Remark \ref{rem:en_bound} for problem \eqref{pb:unpert} corresponds to the following bound on the parameter $\ve$ for problem \eqref{pb:pert}
\begin{equation}\label{def:ve_tilde}
	\ve\in(0,\tilde{\ve}),\quad\mbox{where}\ \tilde{\ve}=\tilde{h}^{1/\al}=\frac{\mathfrak{m}^{1/\al}}{2},
\end{equation}
where we recall that $\mathfrak{m}=\min\limits_{j=1,\ldots,N}\min\limits_{\cerchio}V_j|_{\cerchio}$. Naturally, this bound guarantees that the ball $B_\ve$ containing the rescaled centres is completely included in the Hill's region of problem \eqref{pb:pert}
\[ \mathcal{R}_\ve\uguale\{y\in\R^2:\,\Veps(y)\geq 1\}.
\]
Indeed, following the same computations of Remark \ref{rem:en_bound}, if $\ve\in(0,\tilde{\ve})$ and $|y|\leq\ve$, then
\[
\Veps(y)\geq\mathfrak{m}|y-\ve c_1|^{-\al}\geq 1.
\]

\section{Outer dynamics}\label{sec:outer} 

At this point, the idea is to exploit a perturbation argument suggested by Proposition \ref{prop:pert} and to build pieces of solutions for \eqref{pb:pert}, which lie far from the centres and that will be referred as \emph{outer arcs}. Note that, if $y\colon J\to\R^2$, with $J\sset\R$ is a solution of \eqref{pb:pert}, then 
\[
\Veps(y(t))\geq 1,\quad\mbox{for every}\ t\in J;
\]
for this reason, we need to show that there exists an $R>0$ such that, for every $\ve\in(0,\tilde{\ve})$,
\begin{equation}\label{eq:hill_ve}
	B_{\ve} \subset B_R\subset\{y\in\R^2:\,\Veps(y)\geq1\}=\mathcal{R}_\ve.
\end{equation}
Following the same approach of the end of the previous section, we have that, for any $\ve\in(0,\tilde{\ve})$,
\[
B_{\mathfrak{m}^{1/\al}-\ve}\subset\{y\in\R^2:\ \Veps(y)\geq 1\}.
\]
Hence, choosing 
\begin{equation}\label{eq:R}
	R\in(\tilde \ve,\mathfrak{m}^{1/\al}-\tilde \ve)
\end{equation}
the inclusions \eqref{eq:hill_ve} hold for any $\ve\in(0,\tilde{\ve})$.

Inspired by Propositions \ref{prop:scaling}-\ref{prop:pert}, we are going to look for solutions of the $\ve$-problem \eqref{pb:pert} which start in $\partial B_R(0)$ and travel in $\R^2\setminus B_R(0)$; note that, in this setting, $R$ will satisfy \eqref{eq:R}. These solution arcs will be found as perturbed solutions of an anisotropic Kepler problem driven by $\Wzero$; given $p_0,p_1\in\partial B_R(0)$, we are going to look for solutions of the following problem
\begin{equation}\label{pb:outer}
	\begin{cases}
		\begin{aligned}
			
			&\ddot{y}(t)=\nabla \Veps(y(t))  &t\in[0,T]  \\
			&\frac{1}{2}|\dot{y}(t)|^2-\Veps(y(t))=-1 &t\in[0,T]  \\
			&|y(t)|>R  &t\in(0,T) \\
			&y(0)=p_0,\quad y(T)=p_1, &	
			
		\end{aligned}
	\end{cases}
\end{equation}
for some $T>0$ possibly depending on $\ve$.

\subsection{Homothetic solutions for the anisotropic Kepler problem}\label{par:homothetic}

The core of our perturbation argument consists in focusing on some special trajectories of an anisotropic Kepler problem driven by $\Wzero$, in order to study the behaviour of the close-by orbits. For this reason, we take $\ve=0$ and we consider the problem
\begin{equation}\label{pb:alfa_Kepler}
	\begin{cases}    
		\ddot{x}=\nabla\Wzero(x) \\
		
		\frac{1}{2}|\dot{x}|^2-\Wzero(x)=-1
	\end{cases},
\end{equation}
recalling that $\Wzero\in\mathcal{C}^2(\R^2\setminus\{0\})$ is the $-\al$-homogeneous potential introduced in \eqref{def:wzero}. Note that, if we introduce polar coordinates $x=(r\cos\vt,r\sin\vt)$, the potential $\Wzero$ can be written as 
\[
\Wzero(x)=r^{-\al}\sum_{i=1}^kV_i(\cos\vt,\sin\vt)=r^{-\al}\sum_{i=1}^kU_i(\vt)=r^{-\al}U(\vt),
\]
where $U_i=V_i|_{\cerchio}$ and $U=\sum\limits_{i=1}^kU_i$. From the energy equation in \eqref{pb:alfa_Kepler}, the boundary of the Hill's region for this problem is the closed curve parametrized by polar coordinates in this way
\[
\partial\mathcal{R}_0\uguale\{(r\cos\vt,r\sin\vt):\,r=U(\vt)^{1/\al},\ \vt\in[0,2\pi)\}.
\]    
Let us now take $\xi=Re^{i\vt_\xi}$ in the interior of $\mathcal{R}_0$, i.e., such that
\[
	0<R< U(\vt_\xi)^{\frac{1}{\al}}.
\]
We aim to understand when a homothetic solution for \eqref{pb:alfa_Kepler} starting at $\xi$ exists, i.e., a solution of \eqref{pb:alfa_Kepler} of the form 
\[
\homx(t)=\la(t)\xi, 
\]
where $\la\colon[0,T_\xi]\to\R^+$ and
\[
	\begin{cases}
		\la(t)> 1\ \mbox{for every}\ t\in(0,T_\xi) \\ \la(0)=1=\la(T_\xi)
	\end{cases},
\]
for some $T_\xi>0$. To proceed, we need the following classical definition. 

\begin{definition}
	In general, a \emph{central configuration} for $\Wzero$ is a critical point of $\Wzero$ constrained to a level surface of the inertial moment $I(x)=\frac12|x|^2$. In other words, a central configuration is a vector $\xi=Re^{i\vt_\xi}$ that verifies
	\begin{equation}\label{eq:constraint}
		\nabla \Wzero(\xi)+\mu_\xi\nabla I(\xi)=0,\ \mbox{i.e.,}\ U'(\vt_\xi)=0,
	\end{equation}
	where 
	\[
	\mu_\xi=\al R^{-\al-2} U(\vt_\xi)>0.
	\]
    In this paper, following a common habit, we will both refer to $\vt_\xi\in\cerchio$ and $\xi\in\partial B_R$ as a central configuration for $\Wzero$.
\end{definition}

Classical computations then show that a homothetic solution $\homx=\la\xi$ for \eqref{pb:alfa_Kepler} exists if only if $\xi\in\partial B_R$ is a central configuration for $\Wzero$ and $\la$ solves on $[0,T_\xi]$ the 1-dimensional $\al$-Kepler problem:
\begin{equation}\label{pb:kep_la}
    \begin{cases}
    	\ddot{\la}(t)=-\mu_\xi\la(t)^{-\al-1} \\
    	\la(0)=1,\ \dot{\la}(0)=\frac{1}{R}\sqrt{2(\Wzero(\xi)-1)}
    \end{cases} 
\end{equation}
In conclusion, given a central configuration $\vt_\xi\in\cerchio$ for $\Wzero$, we can consider the following Cauchy problem
\begin{equation}\label{pb:cauchy_x}
	\begin{cases}
		\ddot{x}(t)=\nabla \Wzero(x(t)) \\
		x(0)=\xi,\quad\dot{x}(0)=v_\xi=\frac{1}{R}\sqrt{2(\Wzero(\xi)-1))}\xi,
	\end{cases}
\end{equation}
which admits as unique solution the homothetic trajectory $\homx$, that reaches again the position $\xi$ after a time $T_\xi>0$, with opposite velocity $-v_\xi$. We conclude this short section with a characterization of the eigenvalues and eigenvectors of the Hessian matrix of $\Wzero$ at a central configuration.

\begin{remark}\label{rem:hess}
	For a central configuration $\xi=Re^{i\vt_\xi}$, we define the unit vectors
	\[
	s_\xi\uguale R^{-1}\xi\in\cerchio,\quad s_\tau\uguale s_\xi^\perp\in\cerchio.
	\]
	Then, it turns out that $s_\xi$ and $s_\tau$ are the eigenvectors of the Hessian matrix $\nabla^2\Wzero(\xi)$
	\begin{equation}\label{eq:tang_hessian}
		\begin{aligned}
			\nabla^2\Wzero(\xi)s_\xi&=\al(\al+1)R^{-\al-2}U(\vt_\xi)s_\xi \\
			\nabla^2\Wzero(\xi)s_\tau&=R^{-\al-2}\left(-\al U(\vt_\xi)+U''(\vt_\xi)\right)s_\tau
		\end{aligned}
	\end{equation}
	which correspond to the eigenvalues
	\begin{align}
		\la_\xi&=\langle\nabla^2\Wzero(\xi)s_\xi,s_\xi\rangle=\al(\al+1)R^{-\al-2}U(\vt_\xi) \label{eq:tang_hessian_xi}\\
		\la_\tau&=\langle\nabla^2\Wzero(\xi)s_\tau,s_\tau\rangle=R^{-\al-2}\left(-\al U(\vt_\xi)+U''(\vt_\xi)\right). \label{eq:tang_hessian_tau}
	\end{align}
	Further details of these computations can be found in \cite{CannThesis}.
\end{remark}

\subsection{Shadowing homothetic solutions in the anisotropic Kepler problem}

In Proposition \ref{prop:pert} we have seen that $\Veps$ reduces to $\Wzero$ as $\ve\to 0^+$, together with all the $\ve$-centres collapsing to the origin. For this reason, the aim of this paragraph is to provide an intermediate result, i.e., to prove the existence of trajectories for problem \eqref{pb:alfa_Kepler} which start very close to a given homothetic trajectory $\hat x_{\xi}$. In other words, we investigate the existence of a solution for 
\[
\begin{cases}
	\begin{aligned}
		&\ddot{x}(t)=\Wzero(x(t)), &t\in[0,\overline T] \\
		&\frac{1}{2}|\dot{x}(t)|^2-\Wzero(x(t))=-1, &t\in[0,\overline T]\\
		&|x(t)|>R, &t\in (0,\overline{T}) \\
		&x(0)=p_0,\ x(\overline{T})=p_1,   &
	\end{aligned}
\end{cases}
\]
where $p_0$ and $p_1$ are chosen sufficiently close to a central configuration $\xi=Re^{i\vt_\xi}$ for $\Wzero$. We will follow the same strategy proposed in \cite{ST2012}, with the difference that in our context the argument does not work on the whole sphere $\partial B_R$ because of the anisotropy of $\Wzero$. For our convenience, we need a characterization of the homothetic trajectory $\homx$, the unique solution of the Cauchy problem \eqref{pb:cauchy_x}, in Hamiltonian formalism; hence, we introduce the Hamiltonian function 
\[
H(x,v)\uguale\frac12|v|^2-\Wzero(x)
\] 
and we reword \eqref{pb:cauchy_x} as
\begin{equation}\label{pb:homotetic}
	\begin{cases}
		\dot{z}=F(z) \\
		z(0)=z_\xi,
	\end{cases}
\end{equation}
where
\[
z=\begin{pmatrix}
	x \\
	v
\end{pmatrix},\ F(z)=\begin{pmatrix}
	v \\
	\nabla \Wzero(x)
\end{pmatrix}\ \mbox{and}\ z_\xi=\begin{pmatrix}
	\xi \\
	v_\xi
\end{pmatrix}.
\]
According to this, we restrict the domain of the vector field $F$ to the 3-dimensional energy shell
\[
\mathcal{E}=\left\lbrace(x,v)\in(\R^2\setminus\{0\})\times\R^2:\,H(x,v)=-1\right\rbrace=H^{-1}(-1)
\]   
and we term $\homz$ the homothetic solution in the Hamiltonian formalism, i.e., the unique solution of \eqref{pb:homotetic} defined on $[0,T_\xi]$. Introducing the flow associated to the differential equation in \eqref{pb:homotetic}
\[
\begin{aligned}
	\Phi\colon\Omega&\subset\R\times\mathcal{E}\to\mathcal{E} \\	&(t,z)\mapsto\Phi(t,z)=\Phi^t(z)
\end{aligned}
\] 
we notice that 
\[
\homz(t)=\Phi^t(z_\xi)\quad\mbox{and}\quad \homx(t)=\pi_x(\homz(t))=\pi_x\Phi^t(z_\xi),
\]
where $\pi_x(z)$ and $\pi_v(z)$ represent the two canonic projections of $z$. 
\begin{remark}\label{rem:var_eq}
	In the following we will work with the differential of the flow $\Phi$ with respect to the spatial variable $z$. In general it is not possible to give an explicit expression of the Jacobian matrix of $\Phi$ with respect to $z$, but it is well known (see for instance \cite{HirSmaDev,Teschl_ode}) that such Jacobian matrix satisfies the so called Variational Equation. In particular, if $\gamma_{z_0}(t)\uguale\Phi^t(z_0)$ is a solution curve from $z_0$, then
	\[
	\begin{cases}
		\displaystyle
		\frac{d}{dt}\left(\frac{d}{dz}\Phi^t(z)\Big\rvert_{z=z_0}\right)=JF(\gamma_{z_0}(t))\frac{d}{dz}\Phi^t(z)\Big\rvert_{z=z_0} \\
		\displaystyle
		\frac{d}{dz}\Phi^{t_0}(z)\Big\rvert_{z=z_0}=I_n.
	\end{cases}
	\]
\end{remark}

Now, if we introduce the 2-dimensional inertial surface
\[
\Sigma=\{(x,v)\in\mathcal{E}:\,2I(x)=R^2\}\sset\mathcal{E}
\]
it turns out that both the starting and ending points of the homothetic motion lie on $\Sigma$, i.e., $\hat{z}_\xi(0),\hat{z}_\xi(T_\xi)\in\Sigma$. For our purposes, it is useful to provide a characterization of the elements of the tangent bundle of the surface $\Sigma$. Since $\Sigma=H^{-1}\cap I^{-1}(R^2/2)$, for a point $(x,v)\in\Sigma$ we deduce that
\[
(q,w)\in\mathcal{T}_{(x,v)}\Sigma\Longleftrightarrow\begin{cases}
	\langle v,w\rangle-\langle\nabla \Wzero(x),q\rangle=0 \\
	\langle x,q\rangle=0
\end{cases},
\]
where $\mathcal{T}_{(x,v)}\Sigma$ is the fiber of the tangent bundle $\mathcal{T}\Sigma$ at $(x,v)$. In particular, if $\xi\in\partial B_R$ is a central configuration for $\Wzero$, from the constraint equation \eqref{eq:constraint} we get that 
\begin{equation}\label{eq:tangent}
	(q,w)\in\mathcal{T}_{(\xi,-v_\xi)}\Sigma\Longleftrightarrow\begin{cases}
		\langle v_\xi,w\rangle=0 \\
		\langle\xi,q\rangle=0
	\end{cases}.
\end{equation}  
Moreover, for the same reason, it is not difficult to observe that the field $F$ is \emph{transversal} to $\Sigma$ in $(\xi,v_\xi)$.

Inspired by this, it is easy to prove the following proposition and thus to define a first return map on $\Sigma$.
\begin{proposition}\label{prop:implicit}
	Given $\xi\in\partial B_R$ central configuration for $\Wzero$, there exists a neighbourhood $\mathcal{U}\times\mathcal{V}\sset\Sigma$ of $(\xi,v_\xi)$ and a function $T\in\mathcal{C}^1(\mathcal{U}\times\mathcal{V};\R^+)$ such that
	\begin{itemize}
		\item $T(\xi,v_\xi)=T_\xi$;
		\item for every $(x,v)\in\mathcal{U}\times\mathcal{V}$, for $t>0$ holds 
		\[
		\Phi^t(x,v)\in\Sigma\ \mbox{if and only if}\ t=T(x,v).
		\]    		
	\end{itemize}   	
\end{proposition}
\begin{proof}
	Let us define the $\mathcal{C}^1$ map
	\[
	\begin{aligned}
		f\colon&\R\times\mathcal{E}\to\R \\
		&(t,x,v)\mapsto f(t,x,v)\uguale |\Phi^t(x,v)|^2-R^2.
	\end{aligned}
	\]
	Since $f(T_\xi,\xi,v_\xi)=0$ and 
	\[
		\dfrac{\partial }{\partial t}f(t,x,v)\Big\rvert_{(T_\xi,\xi,v_\xi)}=\left\langle (2\xi,0),(-v_\xi,\nabla \Wzero(\xi))\right\rangle\neq 0,
	\]
	the result in the statement easily follows from the Implicit Function Theorem. 
\end{proof}

From the previous proposition, given $\xi=Re^{i\vt_\xi}\in\partial B_R$ central configuration for $\Wzero$ and $\mathcal{U},\mathcal{V}$ as in Proposition \ref{prop:implicit}, the first return map
\[
\begin{aligned}
	g\colon&\mathcal{U}\times\mathcal{V}\to\Sigma \\
	&(x,v)\mapsto g(x,v)\uguale\Phi^{T(x,v)}(x,v).
\end{aligned}
\]
is well defined. In this way, if we fix $x_0\in\mathcal{U}$, we can define the arriving point $x_1$ as a function of $v\in\mathcal{V}$ as follows
\begin{equation}\label{map:shooting}
	x_1(v)\uguale\pi_x(g(x_0,v)).
\end{equation}
Our aim is to prove that the previous map is invertible, so that we would be able to build solution arcs starting in a point $x_0\in\partial B_R$ and arriving in another point $x_1\in\partial B_R$, with $x_0,x_1$ sufficiently close to $\xi$.

\begin{theorem}\label{thm:invert}
	Given $\xi=Re^{i\vt_\xi}$ central configuration for $\Wzero$ such that $U''(\vt_\xi)\geq 0$, the map $x_1$ defined in \eqref{map:shooting} is invertible in a neighbourhood of $v_\xi$.
\end{theorem}

The proof of Theorem \ref{thm:invert} is rather technical and relies on a series of lemmata which we state and prove below.

\begin{lemma}\label{lem:diff_g}
	In the same setting of Proposition \ref{prop:implicit}, the first return map $g$ is $\mathcal{C}^1$-differentiable over $\mathcal{U}\times\mathcal{V}$ and 
	\[
	dg(z_\xi)\zeta=\frac{d}{dz}\Phi^{T_\xi}(z)\Big\rvert_{z=z_\xi}\zeta+ F(\Phi^{T_\xi}(z_\xi))\langle\nabla T(z_\xi),\zeta\rangle,
	\]
	for every $\zeta\in\mathcal{T}_{z_\xi}\left(\mathcal{U}\times\mathcal{V}\right)$, where $\mathcal{T}_{z_\xi}\left(\mathcal{U}\times\mathcal{V}\right)$ is the fibre at $z_\xi$ of the tangent bundle $\mathcal{T}(\mathcal{U}\times\mathcal{V})$.
\end{lemma}
\begin{proof}
	First of all, observe that for the $\mathcal{C}^1$-dependence on initial data of the flow $\Phi^t$ and for Proposition \ref{prop:implicit}, the map $g$ is $\mathcal{C}^1$-differentiable over $\mathcal{U}\times\mathcal{V}$. Since $g(z_\xi)=(\xi,-v_\xi)$, then the differential of $g$ in the point $z_\xi\in\mathcal{U}\times\mathcal{V}$ is the linear map
	\begin{equation}\label{defn:diff_g}
		\begin{aligned}
			dg(z_\xi)\colon&\mathcal{T}_{z_\xi}(\mathcal{U}\times\mathcal{V})\to\mathcal{T}_{(\xi,-v_\xi)}\Sigma \\
			&\zeta\mapsto dg(z_\xi)\zeta=\frac{d}{dz}\left[\Phi^{T(z)}(z)\right]_{z=z_\xi}\zeta,
		\end{aligned}
	\end{equation}
	and 
	\[
	\begin{aligned}
		\frac{d}{dz}&\left[\Phi^{T(z)}(z)\right]_{z=z_\xi}=\lim\limits_{\|\eta\|\to0}\frac{\Phi^{T(z_\xi+\eta)}(z_\xi+\eta)-\Phi^{T(z_\xi)}(z_\xi)}{\eta}  \\
		&=\frac{d}{dz}\Phi^{T(z_\xi)}(z)\Big\rvert_{z=z_\xi}+F(\Phi^{T(z_\xi)}(z_\xi))\nabla T(z_\xi).
	\end{aligned}
	\]
	Then, the conclusion follows recalling that $T(z_\xi)=T_\xi$.
\end{proof}

\begin{lemma}\label{lem:qtau}
	In the same setting of Proposition \ref{prop:implicit}, given $\zeta\in\mathcal{T}_{z_\xi}(\mathcal{U}\times\mathcal{V})$ and $t\in(0,T_\xi)$, define
	\[
	q(t)\uguale\pi_x\frac{d}{dz}\Phi^t(z)\Big\rvert_{z=z_\xi}\zeta
	\]
	and consider $s_\xi,s_\tau\in\cerchio$ as in Remark \ref{rem:hess}. Then, the projection of $q$ over the direction $s_\tau$
	\[
	q_\tau(t)=\langle q(t),s_\tau\rangle s_\tau
	\]
	solves the linearised problem 
	\[
	\begin{cases}
		\ddot{q}_\tau=\langle\nabla^2\Wzero(\homx(t))s_\tau,s_\tau\rangle q_\tau \\
		q_\tau(0)=\langle\pi_x\zeta,s_\tau\rangle s_\tau,
	\end{cases}
	\]
	recalling that $\homx$ is the unique (homothetic) solution of \eqref{pb:cauchy_x}.
\end{lemma}
\begin{proof}
	Following Remark \ref{rem:var_eq}, we know that the partial derivative of $\Phi$ with respect to $z$ satisfies the variational equation along the homothetic solution $\homx$, which gives us information about how the flow is sensible under variations made on the initial condition $z(0)=(x(0),\dot{x}(0))$. Since the Jacobian matrix of the vector field $F$ in $z$ reads
	\[
	JF(z)=\begin{pmatrix}
		0_2 & I_2 \\
		\nabla^2 \Wzero(x) & 0_2
	\end{pmatrix},
	\]
	again by Remark \ref{rem:var_eq}, the variational equation reads
	\[
	\begin{cases}
		\displaystyle
		\frac{d}{dt}\left(\frac{d}{dz}\Phi^t(z)\Big\rvert_{z=z_\xi}\zeta\right)=\begin{pmatrix}
			0_2 & I_2 \\
			\nabla^2 \Wzero(\homx(t)) & 0_2
		\end{pmatrix}\frac{d}{dz}\Phi^t(z)\Big\rvert_{z=z_\xi}\zeta, \\
		\displaystyle
		\frac{d}{dz}\Phi^0(z)\Big\rvert_{z=z_\xi}\zeta=\zeta,
	\end{cases}
	\]
	for every $\zeta\in\mathcal{T}_{z_\xi}(\mathcal{U}\times\mathcal{V})$. In this way, writing 
	\[\begin{pmatrix}
		q(t) \\
		w(t)\end{pmatrix}=\frac{d}{dz}\Phi^t(z)\Big\rvert_{z=z_\xi}\zeta,
	\]
	we see that $q(t)$ must satisfy the problem
	\begin{equation}\label{eq:lin_var}
		\begin{cases}
			\ddot{q}=\nabla^2 \Wzero(\homx(t))q \\
			q(0)=\pi_x\zeta.
		\end{cases}
	\end{equation}
	Now, we can decompose $q$ in the orthogonal components 
	\[
	q=q_\xi+q_\tau=\langle q,s_\xi\rangle s_\xi+\langle q,s_\tau\rangle s_\tau
	\]
	and so, by the first equation in \eqref{eq:lin_var}, we get
	\[
	\begin{aligned}
		\ddot{q}_\xi+\ddot{q}_\tau&=\nabla^2\Wzero(\hat{x}_\xi(t))q_\xi+\nabla^2\Wzero(\hat{x}_\xi(t))q_\tau \\
		&=\langle q,s_\xi\rangle\nabla^2 \Wzero(\hat{x}_\xi(t))s_\xi+\langle q,s_\tau\rangle\nabla^2 \Wzero(\hat{x}_\xi(t))s_\tau.
	\end{aligned}
	\]	
	Now, since $s_\tau$ is one of the eigenvectors of the matrix $\nabla^2\Wzero(\homx(t))$ (see Remark \ref{rem:hess}) with eigenvalue
	\[
	\la_\tau=\langle\Wzero(\homx(t))s_\tau,s_\tau\rangle,
	\]
	the orthogonality of $s_\xi$ and $s_\tau$ proves that $q_\tau$ verifies
	\[
	\ddot{q}_\tau=\langle\Wzero(\homx(t))s_\tau,s_\tau\rangle q_\tau
	\]
	and thus problem \eqref{eq:lin_var} can be projected along the direction $s_\tau$ to finally obtain the proof.
\end{proof}

\begin{lemma}\label{lem:invert_no_pert}
	Given $\xi=Re^{i\vt_\xi}$ central configuration for $\Wzero$ such that $U''(\vt_\xi)\geq 0$, the Jacobian matrix
	\begin{equation}\label{eq:jacobian}
		\pi_x\frac{\partial}{\partial v}g(x,v)
	\end{equation}
	is invertible in $(\xi,v_\xi)$.
\end{lemma}
\begin{proof}
	Recalling the definition of the first return map $g$ and its differential in Lemma \ref{lem:diff_g}, assume by contradiction that there exists $\overline{\zeta}=(0,\overline{w})\in\mathcal{T}_{(\xi,v_\xi)}(\mathcal{U}\times\mathcal{V})$, with $\overline{w}\neq 0$ such that
	\[
	\pi_xdg(z_\xi)\overline{\zeta}=0.
	\]
	In this way, by Lemma \ref{lem:diff_g}, we have that
	\begin{equation}\label{eq:tang}
	\pi_x\frac{d}{dz}\Phi^{T_\xi}(z)\Big\rvert_{z=z_\xi}\overline{\zeta}=-\pi_x\left(F(\Phi^{T_\xi}(z_\xi))\langle\nabla T(z_\xi),\overline{\zeta}\rangle\right)=\langle\nabla T(z_\xi),\overline{\zeta}\rangle v_\xi\in\pi_x\mathcal{T}_{(\xi,-v_\xi)}\Sigma.
	\end{equation}
	At this point, since $\xi$ and $v_\xi$ are parallel, by \eqref{eq:tangent} and \eqref{eq:tang} we deduce that necessarily
	\[\langle\nabla T(z_\xi),\overline{\zeta}\rangle=0\quad\mbox{and thus}\quad\pi_x\frac{d}{dz}\Phi^{T_\xi}(z)\Big\rvert_{z=z_\xi}\overline{\zeta}=0.
	\]
	This means that, if we define as in Lemma \ref{lem:qtau}
	\[
	q(t)=\pi_x\frac{d}{dz}\Phi^t(z)\Big\rvert_{z=z_\xi}\overline{\zeta}
	\]
	then $q(T_\xi)=0$ and $q_\tau(T_\xi)=\langle q(T_\xi),s_\tau\rangle s_\tau=0$. Now, from Lemma \ref{lem:qtau}, we know that the projection $q_\tau(t)$ solves the Sturm-Liouville problem
	\begin{equation}\label{pb:slp}
		\begin{cases}
			\ddot{q}_\tau+c(t)q_\tau=0 \\
			q_\tau(0)=0=q_\tau(T_\xi),
		\end{cases}
	\end{equation}
	where by \eqref{eq:tang_hessian_tau} in Remark \ref{rem:hess}
	\[
	c(t)\uguale-\langle\nabla^2\Wzero(\homx(t))s_\tau,s_\tau\rangle=|\homx(t)|^{-\al-2}(\al U(\vt_\xi)-U''(\vt_\xi)).
	\]
	Let $u(t)\uguale|\homx(t)|=\la(t)R$, where $\la(t)$ solves the 1-dimensional $\al$-Kepler problem \eqref{pb:kep_la}; then $u$ solves 
	\begin{equation}\label{pb:u}
		\begin{cases}
			\ddot{u}+\al u^{-\al-2}U(\vt_\xi)u=0 \\
			u(0)=R=u(T_\xi).
		\end{cases}
	\end{equation}
	Now, since $U''(\vt_\xi)\geq 0$, we have that
	\[
	c(t)\leq\al u(t)^{-\al-2}U(\vt_\xi)
	\]
	and therefore, by the Sturm comparison theorem referred to \eqref{pb:slp} and \eqref{pb:u}, we have that there exists $\overline{T}\in(0,T_\xi)$ such that $u(\overline{T})=0$. This is finally a contradiction and concludes the proof, since $|\hat{x}_\xi(t)|$ can not be null in the interval $[0,T_\xi]$.
\end{proof}

At this point the proof of Theorem \ref{thm:invert} follows from Lemma \ref{lem:invert_no_pert}.
\begin{proof}[Proof of Theorem \ref{thm:invert}]
	It is enough to observe that
	\[
	\frac{\partial}{\partial v}x_1(v)\big\rvert_{v=v_\xi}=\frac{\partial}{\partial v}\pi_x(g(\xi,v))\Big\rvert_{v=v_\xi}=\pi_x\frac{\partial}{\partial v}g(x,v)\Big\rvert_{(x,v)=(\xi,v_\xi)},
	\]
	which is invertible for Lemma \ref{lem:invert_no_pert}.
\end{proof}

Now, we are ready to prove the main result of this section, which concerns the existence of outer arcs for the anisotropic Kepler problem.
\begin{theorem}\label{thm:implicit}
	Given $\xi=Re^{i\vt_\xi}$ central configuration for $\Wzero$ such that $U''(\vt_\xi)\geq 0$, there exists a neighbourhood $\mathcal{U}_\xi$ of $\xi$ on $\partial B_{R}$ such that, for any $p_0,p_1 \in \mathcal{U}_\xi$ there exist $\overline{T}>0$ and a unique solution $x=x(t)$ of 
	\[
	\begin{cases}
		\begin{aligned}
			&\ddot{x}(t)=\nabla\Wzero(x(t)), &t\in[0,\overline T] \\
			&\frac{1}{2}|\dot{x}(t)|^2-\Wzero(x(t))=-1, &t\in[0,\overline T]\\
			&|x(t)|>R, &t\in (0,\overline{T}) \\
			&x(0)=p_0,\ x(\overline{T})=p_1.   &
		\end{aligned}
	\end{cases}
	\]
	Moreover, $x$ depends on a $\mathcal{C}^1$-manner on the endpoints $p_0,p_1$.
\end{theorem}
\begin{proof}
	Define the shooting map
	\[
	\begin{aligned}
		\Psi\colon\mathcal{U}&\times\mathcal{U}\times\mathcal{V}\to\R^2 \\
		&(p_0,p_1,v_0)\mapsto\Psi(p_0,p_1,v_0)\uguale x(T(p_0,v_0);p_0,v_0)-p_1,
	\end{aligned}
	\]
	where the sets $\mathcal{U}$ and $\mathcal{V}$ are respectively the neighbourhoods of $\xi$ and $v_\xi$ found in Proposition \ref{prop:implicit}, $T\colon\mathcal{U}\times\mathcal{V}\to\R^+$ is the $\mathcal{C}^1$ first return map defined in the same proposition and  $x(\cdot;p_0,v_0)$ is the unique solution of the Cauchy problem
	\begin{equation}\label{pb:cauchy}
		\begin{cases}
			\ddot{x}(t)=\nabla \Wzero(x(t)) \\
			x(0)=p_0,\quad\dot{x}(0)=v_0
		\end{cases}
	\end{equation}
	in the time interval $[0,T(p_0,v_0)]$. Note that, following the notations of Lemma \ref{lem:invert_no_pert}, we have
	\[
	x(t;p_0,v_0)=\pi_x\Phi^t(p_0,v_0),\quad\mbox{for every}\ t\in[0,T(p_0,v_0)].
	\]
	The map $\Psi$ is $\mathcal{C}^1$ in its domain both for the $\mathcal{C}^1$-dependence of the solutions of the Cauchy problem \eqref{pb:cauchy} on initial data and time and for the differentiability of the first return map $T$ (see Proposition \ref{prop:implicit}). Moreover, we have that
	\[
	\Psi(\xi,\xi,v_\xi)=x(T(\xi,v_\xi);\xi,v_\xi)-\xi=\pi_x\Phi^{T_\xi}(\xi,v_\xi)-\xi=0
	\]
	and
	\[
	\frac{\partial\Psi}{\partial v_0}(p_0,p_1,v_0)\Big\rvert_{(\xi,\xi,v_\xi)}=\frac{\partial}{\partial v_0}x(T(p_0,v_0);p_0,v_0)\Big\rvert_{(\xi,v_\xi)}=\pi_x\frac{\partial}{\partial v}\left[\Phi^{T(p,v)}(p,v)\right]_{(\xi,v_\xi)},
	\]
	which is invertible thanks to Lemma \ref{lem:invert_no_pert}. Therefore, by the Implicit Function Theorem, we have that there exist a neighbourhood $\mathcal{V}'\sset\mathcal{V}$ of $v_\xi$, a neighbourhood $\mathcal{U}_\xi\sset\mathcal{U}$ of $\xi$ and a unique $\mathcal{C}^1$ function $\eta\colon\mathcal{U}_\xi\times\mathcal{U}_\xi\to\mathcal{V}'$ such that
	$\eta(\xi,\xi)=v_\xi$ and
	\[
	\Psi(p_0,p_1,\eta(p_0,p_1))=0\quad \mbox{for every}\ (p_0,p_1)\in\mathcal{U}_\xi\times\mathcal{U}_\xi.
	\]
	This actually means that, if we fix $(p_0,p_1)\in\mathcal{U}_\xi\times\mathcal{U}_\xi$, we can find a solution $x$ of \eqref{pb:cauchy}, defined in the time interval $[0,\overline{T}]$, with $v_0=\eta(p_0,p_1)$ and $\overline{T}=T(p_0,\eta(p_0,p_1))=T(p_0,v_0)$. Furthermore, note that this solution has constant energy $-1$, since
	\[
	(p_0,\eta(p_0,p_1))=(p_0,v_0)\in\mathcal{U}_\xi\times\mathcal{V}'\subset\mathcal{U}\times\mathcal{V}\subset\Sigma\subset\mathcal{E}.
	\]
	The $\mathcal{C}^1$-dependence on initial data is a straightforward consequence of the Implicit Function Theorem.
\end{proof}

\subsection{Outer solution arcs for the \texorpdfstring{$N$}{}-centre problem}

We conclude this section with the proof of the existence of an outer solution arc for the anisotropic $N$-centre problem driven by $\Veps$. As a starting point, we recall that, by Proposition  \ref{prop:pert}, if $|y|>R>0$, then
\[
\Veps(y)=\Wzero(y)+O(\ve^\gamma),\quad\mbox{as}\ \ve\to 0^+
\]
for a suitable $\gamma>0$. This suggests to repeat the proof of Theorem \ref{thm:implicit}, this time taking into account the perturbation induced by the presence of the centres. Before we start with the proof, it is useful to recall the set of strictly minimal central configurations of $\Wzero$, defined as
\[
\Xi=\{\vt^*\in\cerchio:\ U'(\vt^*)=0\ \mbox{and}\ U''(\vt^*)>0\}=\{\vt_1^*,\ldots,\vt_m^*\}.
\]
Note that, as it is clear from the assumptions of Theorem \ref{thm:implicit}, it would be enough to require the (not necessarily strict) minimality of the above central configurations. Beside that, the non-degeneration of such critical points will be a fundamental requirement on Section \ref{sec:glueing} and however we decide to keep it since it is a natural assumption in anisotropic settings (see for instance \cite{BTV,BTVplanar,BCTmin}).
\begin{theorem}\label{thm:outer_dyn}
	Assume that $N\geq 1$ and $m\geq 1$ and consider a function $V\in\mathcal{C}^2(\R^2\setminus\{c_1,\ldots,c_N\})$ defined as in \eqref{def:potential} and satisfying assumptions \eqref{hyp:V} at page \pageref{hyp:V}. Fix $R>0$ as in \eqref{eq:R} at page \pageref{eq:R} and, for $0<\ve<R$ consider the potential $\Veps$ defined in \eqref{def:scaling} at page \pageref{def:scaling}. Then, there exists $\ve_{ext}>0$ such that, for any $\vt^*\in\Xi$ minimal non-degenerate central configuration for $\Wzero$, defining $\xi^*\uguale Re^{i\vt^*}$, there exists a neighbourhood $\mathcal{U}_{ext}(\xi^*)$ of $\xi^*$ on $\partial B_R$ with the following property:
	
	for every $\ve\in(0,\ve_{ext})$, for any pair of endpoints $p_0,p_1\in\mathcal{U}_{ext}(\xi^*)$, there exist $T_{ext}=T_{ext}(p_0,p_1;\ve)>0$ and a unique solution $y_{ext}(t)=y_{ext}(t;p_0,p_1;\ve)$ of the outer problem
	\[
	\begin{cases}
		\begin{aligned}
			&\ddot{y}_{ext}(t)=\nabla \Veps(y_{ext}(t))  &t\in[0,T_{ext}]  \\
			&\frac{1}{2}|\dot{y}_{ext}(t)|^2-\Veps(y_{ext}(t))=-1 &t\in[0,T_{ext}]  \\
			&|y_{ext}(t)|>R  &t\in(0,T_{ext}) \\
			&y_{ext}(0)=p_0,\quad y_{ext}(T_{ext})=p_1. &	
		\end{aligned}
	\end{cases}
	\]	
	Moreover, the solution depends on a $\mathcal{C}^1$-manner on its endpoints $p_0$ and $p_1$.
\end{theorem}
\begin{proof}
	The proof goes exactly as the proof of Theorem \ref{thm:implicit}, this time introducing the variable $\ve\in[0,\tilde{\ve})$, with $\tilde{\ve}$ defined in \eqref{def:ve_tilde}. Therefore, we define the shooting map
	\[
	\begin{aligned}
		\Psi\colon[0,\tilde{\ve})&\times\mathcal{U}\times\mathcal{U}\times\mathcal{V}\to\R^2 \\
		&(\ve,p_0,p_1,v_0)\mapsto\Psi(\ve,p_0,p_1,v_0)\uguale y(T(p_0,v_0);p_0,v_0;\ve)-p_1,
	\end{aligned}
	\]
	where the sets $\mathcal{U}$ and $\mathcal{V}$ are respectively the neighbourhoods of $\xi^*$ and $v_{\xi^*}$ as in Proposition \ref{prop:implicit}, $T\colon\mathcal{U}\times\mathcal{V}\to\R^+$ is the $\mathcal{C}^1$ first return map defined in the same proposition and $y(\cdot;p_0,v_0;\ve)$ is the unique solution of the Cauchy problem
	\[
		\begin{cases}
			\ddot{y}(t)=\nabla \Veps(y(t)) \\
			y(0)=p_0,\quad\dot{y}(0)=v_0,	
		\end{cases}
	\]
	in the time interval $[0,T(p_0,v_0)]$. In this way, by the Implicit Function theorem, we have that there exist a neighbourhood $\mathcal{V}'\subset\mathcal{V}$ of $v_{\xi^*}$, $\ve_{ext}\in(0,\tilde{\ve})$, a neighbourhood $\mathcal{U}_{ext}(\xi^*)\subset\mathcal{U}$ of $\xi^*$ and a unique $\mathcal{C}^1$ function $\eta\colon[0,\ve_{ext})\times\mathcal{U}_{ext}(\xi^*)\times\mathcal{U}_{ext}(\xi^*)\to\mathcal{V}'$ such that $\eta(0,\xi^*,\xi^*)=v_{\xi^*}$ and 
	\[
	\Psi(\ve,p_0,p_1,\eta(\ve,p_0,p_1))=0\quad\mbox{for every}\ (\ve,p_0,p_1)\in[0,\ve_{ext})\times\mathcal{U}_{ext}(\xi^*)\times\mathcal{U}_{ext}(\xi^*).
	\]
	This actually means that, if we fix $\ve\in[0,\ve_{ext})$ and $(p_0,p_1)\in\mathcal{U}_{ext}(\xi^*)\times\mathcal{U}_{ext}(\xi^*)$, we can find a unique solution $y_{ext}$ of \eqref{pb:outer} at page \pageref{pb:outer}, defined in the time interval $[0,T_{ext}]$, starting with velocity $v_0=\eta(\ve,p_0,p_1)$ and such that $T_{ext}=T(p_0,\eta(\ve,p_0,p_1))$ in the fashion of Proposition \ref{prop:implicit}. Finally, note that this solution has constant energy $-1$, since
	\[
	(p_0,\eta(\ve,p_0,p_1))\in\mathcal{U}_{ext}(\xi^*)\times\mathcal{V}'\subset\mathcal{U}\times\mathcal{V}\subset\Sigma\subset\mathcal{E}.
	\]
	To conclude, the $\mathcal{C}^1$-dependence on the endpoints is a straightforward consequence on the perturbation technique used in the proof.
\end{proof}

\begin{figure}
	\centering
	\begin{tikzpicture}[scale=1.2]
		\coordinate (O) at (0,0);
		\coordinate (x0) at (6,0);
		\coordinate (p0) at (5.223,-0.102);
		\coordinate (p1) at (6.776,-0.102);
		\coordinate (p2) at (5.686,-0.016);
		\coordinate (p3) at (6.3,-0.016);
		\coordinate (p4) at (4.73,-0.28);
		\coordinate (p5) at (8.5,1);
		
		\draw [thick,-stealth] (O)--(0,3);
		\draw [thick,-stealth] (0,3)--(0,2);
		\draw [dashed] (0,-1)--(O);
		\draw [dashed] (6,-1)--(6,3);
		
		\draw [domain=60:120] plot ({3*cos(\x)}, {3*sin(\x)-3});
		\draw [domain=60:120] plot ({6+3*cos(\x)}, {3*sin(\x)-3});
		\draw [dashed,domain=60:120] plot ({5*cos(\x)}, {3*sin(\x)});
		\draw [dashed,domain=60:120] plot ({6+5*cos(\x)}, {3*sin(\x)});
		
		\fill (O) circle[radius=1.5pt];
		\fill (x0)
		circle[radius=1.5pt];  		
		\fill (p0)
		circle[radius=1.0pt];  		
		\fill (p1)
		circle[radius=1.0pt];
		\fill (p2)
		circle[radius=1.0pt];  		
		\fill (p3)
		circle[radius=1.0pt];
		
		\draw[thick,-stealth] (p0) parabola bend (6,2.7) (p1);
		\draw[thick,-stealth] (p2) parabola bend (6,2.9) (p3); 
		\draw[dashed] (p4) parabola bend (6,2) (p5);  
		
		\node[left] at (0,-0.65) {$\xi^*$};
		\node[left] at (6,-0.65) {$\xi^*$};
		\node[below] at (p0) {$p_0$};
		\node[below] at (p1) {$p_1$};
		\node at (6.7,2.6) {\small $y_{ext}(t)$};
		\node at (-2,3.1) {\small $\partial\mathcal{R}$};   		
		\node at (4,3.1) {\small $\partial\mathcal{R}$};
		\node[right] at (0,1.5) {\small{homothetic}};
		\node at (-1.7,-0.2) {\small $\partial B_R$};
		\node at (4.3,-0.2) {\small $\partial B_R$};
	\end{tikzpicture}
	\caption{The proof of Theorem \ref{thm:outer_dyn}: here $\partial\mathcal{R}$ denotes the boundary of the Hill's region for the rescaled $N$-centre problem driven by $\Veps$. On the left side we have drawn the homothetic trajectory through $\xi^*$: it is a 1-dimensional motion that starts in $\xi^*$, it reaches the boundary $\partial\mathcal{R}$ and then it hits again $\partial B_R$ in $\xi^*$. On the right we can see that, if we shoot with initial position sufficiently close to $\xi^*$, there will be a first return on the sphere, guaranteed by the transversality of the flow. On the other hand, the dashed trajectory on the right could never reach again the sphere since its starting point is outside the existence neighbourhood provided in the theorem.}	
\end{figure}
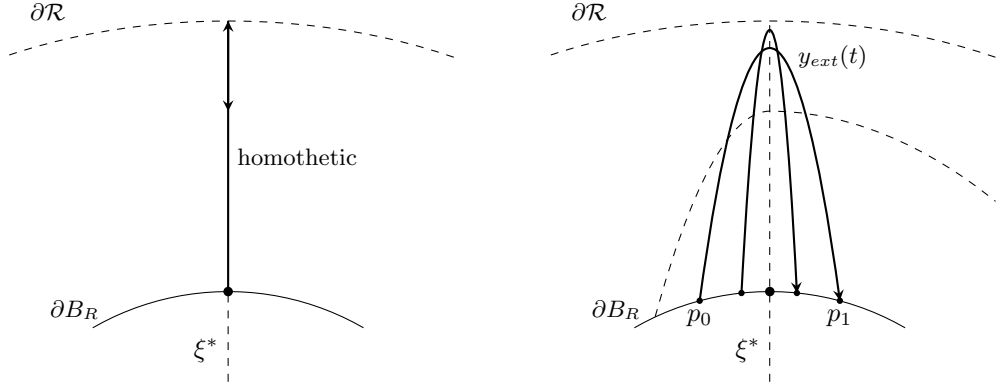

We conclude this section providing upper and lower bounds for the time interval in which an external solution is defined, that will be useful later in this work.

\begin{lemma}\label{lem:bound_ext}
	Let $\ve\in(0,\ve_{ext})$, let $\vt^*\in\cerchio$ be a minimal non-degenerate central configuration for $\Wzero$ and $\mathcal{U}_{ext}(\xi^*)$ be its neighbourhood on $\partial B_R$ found in Theorem \ref{thm:outer_dyn}. Let $p_0,p_1\in\mathcal{U}_{ext}(\xi^*)$ and let $y_{ext}(\cdot;p_0,p_1;\ve)$ be the unique solution  found in Theorem \ref{thm:outer_dyn}, defined in its time interval $[0,T_{ext}(p_0,p_1;\ve)]$. Then, there exist $c,C>0$ such that
	\[
	c\leq T_{ext}(p_0,p_1;\ve)\leq C.
	\]
	Such constants do not depend on the choice of $p_0,p_1$ inside the neighbourhood.
\end{lemma}
\begin{proof}
	The proof is a direct consequence of the continuous dependence of the solution on initial data and of its perturbative nature.
\end{proof}

\section{Inner dynamics}\label{sec:inner}

This section is named \emph{Inner dynamics} since we will look for solution arcs of the rescaled $N$-centre problem \eqref{pb:rescaled}, which bridge any pair of points of $\partial B_R$ ($R>>\ve>0$ already chosen in Section \ref{sec:outer}) and lie inside the ball $B_R$ along their motion. As far as we are looking for classical solutions with fixed end-points, we need to face the problem of collisions with the centres. Moreover, since we will be working inside $B_R$, we cannot make use of Proposition \ref{prop:pert} and thus perturbation techniques do not apply in this case. For this reason, following \cite{ST2012}, we opt for a variational approach and our inner solution arcs will be (reparametrizations of) minimizers of a suitable geometric functional. We briefly recall that, in the non-collisional case (see Theorem \ref{cor:symb} in the Introduction), partitions play a fundamental role in the alphabet of symbolic dynamics. Keeping this in mind, we will define a suitable topological constraint that forces every inner arc to separate the centres according to a prescribed partition. To be clear, the main result of this section is to prove that, for $\ve>0$ sufficiently small and for any $p_1,p_2\in\partial B_R$, there exists a solution $y(\cdot;p_1,p_2;\ve)$ of the following problem
\begin{equation}\label{pb:inner}
	\begin{cases}
		\begin{aligned}
			&\ddot{y}(t)=\nabla \Veps(y(t))  &t\in[0,T] \\
			&\frac{1}{2}|\dot{y}(t)|^2-\Veps(y(t))=-1   &t\in[0,T] \\
			&|y(t)|<R   &t\in(0,T) \\
			&y(0)=p_1,\quad y(T)=p_2, &
		\end{aligned}
	\end{cases}
\end{equation}
for some $T>0$, possibly depending on $\ve$, and such that the trajectory $y$ separates the centres according to a chosen partition.  

\subsection{Functional setting and variational principles}

We build our variational setting referring to the starting equations \eqref{eq:moto}-\eqref{eq:energy} and thus we take into account again the potential $V$ and the energy $-h<0$ is fixed. However, we notice that a scaling on the centres, and thus on the whole problem, does not affect the following discussion. We fix $p_1,p_2$ inside the open Hill's region $\mathring{\mathcal{R}}_h$ (see \eqref{def:hill_region} at page \pageref{def:hill_region}) and we define
\[
\hat{H}=\hat{H}_{p_1,p_2}([a,b])\uguale\left\lbrace u\in H^1([a,b];\R^2)\ \middle\lvert\ \begin{aligned}
&u(a)=p_1,\ u(b)=p_2, \\ 
&u(t)\neq\ c_j\ \forall\ t\in[a,b],\ \forall\ j
\end{aligned}
\right\rbrace,
\]
i.e., all the $H^1$-paths that join $p_1,p_2$ and do not collapse on the centres, and also the $H^1$-collision paths
\[
\mathfrak{Coll}=\mathfrak{Coll}_{p_1,p_2}([a,b])\uguale\left\lbrace u\in H^1([a,b];\R^2)\ \middle\lvert\ \begin{aligned} &u(a)=p_1,\ u(b)=p_2,\ \exists\ t\in[a,b], \\ &\exists\ j\in\{1,\ldots,N\}\ \mbox{s.t.}\ u(t)=c_j\end{aligned}\right\rbrace.
\]
We introduce also the set 
\[
\begin{aligned}
H= H_{p_1,p_2}([a,b])&\uguale\hat{H}_{p_1,p_2}([a,b])\cup\mathfrak{Coll}_{p_1,p_2}([a,b]) \\
&=\{u\in H^1([a,b];\R^2):\ u(a)=p_1,\ u(b)=p_2\}
\end{aligned}
\]
and it is easy to check that $H$ is the closure of $\hat{H}$ with respect to the weak topology of $H^1([a,b];\R^2)$. Let us define the Maupertuis' functional as
\[
\begin{aligned}
	\mathcal{M}_h(\cdot)\uguale\mathcal{M}_h([a,b];\cdot)\colon&H_{p_1,p_2}([a,b])\longrightarrow\R\cup\{+\infty\} \\
	&\qquad u\longmapsto\mathcal{M}_h(u)\uguale\frac12\int_a^b|\dot{u}(t)|^2\,dt\int_a^b(-h+V(u(t)))\,dt
\end{aligned}
\]
which is differentiable over the non-collision paths space $\hat{H}$. The next classical result, known as the \emph{Maupertuis' principle}, establishes a link between classical solutions of the equation $\ddot{x}=\nabla V(x)$ at energy $-h$ and critical points at a positive level of $\mathcal{M}_h$ in the space $\hat{H}$. Note that, if $\mathcal{M}_h(u)>0$ for some $u\in H$, then we can define the positive quantity
\begin{equation}\label{defn:omega}
	\omega^2\uguale\frac{\int_a^b(-h+V(u))}{\frac12\int_a^b|\dot{u}|^2},
\end{equation}
that plays an important role in the next classical result (see \cite[Theorem 4.1]{AC-Z}).

\begin{theorem}[The Maupertuis' principle]\label{thm:maupertuis}
	Let $u\in\hat{H}_{p_1,p_2}([a,b])$ be a critical point of $\mathcal{M}_h$ at a positive level and let $\omega>0$ be defined by \eqref{defn:omega}. Then, $x(t)\uguale u(\omega t)$ is a classical solution of the fixed-end problem
	\[
	\begin{cases}
		\ddot{x}(t)=\nabla V(x(t))  & t\in[a/\omega,b/\omega] \\
		\frac12|\dot{x}(t)|^2-V(x(t))=-h  & t\in[a/\omega,b/\omega] \\
		x(a/\omega)=p_1,\ x(b/\omega)=p_2 &
	\end{cases}
	\]
	while $u$ itself is a classical solution of 
	\[
	\begin{cases}
		\omega^2\ddot{u}(t)=\nabla V(u(t))  & t\in[a,b] \\
		\frac{\omega^2}{2}|\dot{u}(t)|^2-V(u(t))=-h  & t\in[a,b] \\
		u(a)=p_1,\ u(b)=p_2. &
	\end{cases}
	\]
	The converse holds also true, i.e., if $x$ is a classical solution of the fixed-end problem above in a certain interval $[a',b']$, then, setting $\omega=1/(a'-b')$, $u(t)=x(t/\omega)$ is a critical point of $\mathcal{M}_h([a,b];\cdot)$ at a positive level, for some suitable values $a,b$.
\end{theorem}

In order to apply direct methods of the Calculus of Variations to $\mathcal{M}_h$ we will work in $H$, which is weakly closed in $H^1$. As a first step we recall a standard result that shows that a (possibly colliding) minimizer of $\mathcal{M}_h$ in $H$ preserves the energy almost everywhere. 

\begin{lemma}\label{lem:energy}
	If $u\in H$ is a minimizer of $\mathcal{M}_h$ at a positive level, then
	\[
	\frac{\omega^2}{2}|\dot{u}(t)|^2-V(u(t))=-h\quad\mbox{for a.e.}\ t\in[a,b].
	\]
\end{lemma}

The lack of additivity of $\mathcal{M}_h$ induces the introduction of the Jacobi-length functional 
\[
\mathcal{L}_h(u)\uguale\int_0^1|\dot{u}(t)|\sqrt{-h+V(u(t))}\,dt
\]
whose domain is the weak $H^1$-closure of the set
\[
H_h^{p_1,p_2}([a,b])\uguale\left\lbrace u\in H^1([a,b];\R^2)\ \middle\lvert\ \begin{aligned}
	&u(a)=p_1,\ u(b)=p_2, \\
	&V(u(t))>h,\ |\dot{u}(t)|>0,\ \mbox{for every}\ t\in[a,b]
\end{aligned}\right\rbrace.
\]
Indeed, Theorem \ref{thm:maupertuis} could be rephrased for $\mathcal{L}_h$ and thus classical solutions will be suitable reparametrizations of critical points of $\mathcal{L}_h$ (see for instance \cite{MoMoVe2012} and Appendix \ref{app:var} for more precise details on this functional). Finally, we recall that the Jacobi-length functional, being a length, is additive and it is also invariant under reparametrizations. Despite that, exploiting the correspondence which stands between minimizers of $\mathcal{M}_h$ and minimizers of the Jacobi-length functional (see Proposition \ref{prop:app1}), an easy proof leads to the following proposition.
\begin{proposition}\label{prop:restriction}
	Let $u$ be a minimizer of $\mathcal{M}_h([a,b];\cdot)$ in $H_{p_1,p_2}([a,b])$. Then, for any subinterval $[c,d]\sset[a,b]$, the restriction $u|_{[c,d]}$ is a minimizer of $\mathcal{M}_h([c,d];\cdot)$ in the space $H_{u(c),u(d)}([c,d])$.
\end{proposition}

\subsection{Minimizing through direct methods}\label{subsec:dir_meth}

At this point, we go back to the $\ve$ $N$-centre problem \eqref{pb:inner}, introducing the notation $c_j'\uguale\ve c_j$ for the $\ve$-centres included in $B_\ve$. We aim to prove the existence of a minimizer for the Maupertuis' functional, requiring the following topological constraint: an inner arc has to cross the ball $B_\ve$, dividing the centres into two non-trivial subsets. Following \cite{ST2012}, this can be done introducing the winding number with respect to every centre; but since a path in $\hat{H}$ is not necessarily closed, we need to close it artificially. Let us fix $[a,b]\sset\R$, $p_1,p_2\in\partial B_R$ and write
\[
p_1=Re^{i\vt_1},\quad p_2=Re^{i\vt_2}
\]
for $\vt_1,\vt_2\in[0,2\pi)$. For $u\in\hat{H}_{p_1,p_2}([a,b])$, if $p_1\neq p_2$ we close $u$ glueing an arc of $\partial B_R$ in counter-clockwise direction, i.e., we define
\[
\Gamma_u(t)\uguale\begin{cases}
	\begin{cases}
		u(t) & t\in[a,b] \\
		Re^{i(t-b+\vt_2)} & t\in(b,b+\vt_1+2\pi-\vt_2)
	\end{cases} & \mbox{if}\ \vt_1<\vt_2 \\
	u(t)\qquad\qquad\quad t\in[a,b] & \mbox{if}\ \vt_1=\vt_2 \\
	\begin{cases}
		u(t) & t\in[a,b] \\
		Re^{i(t-b+\vt_2)} & t\in(b,b+\vt_1-\vt_2)
	\end{cases} & \mbox{if}\ \vt_1>\vt_2
\end{cases}
\]
and so, we can introduce the winding number of $u$ with respect to a centre $c_j'$ as
\[
\mbox{Ind}(u;c_j')\uguale\frac{1}{2\pi i}\int_{\Gamma_u}\frac{dz}{z-c_j'}\in\Z,\quad\mbox{for all}\ j=1,\ldots,N.
\] 
Since a path $u$ has to separate the centres with respect to a given partition in two non-trivial subsets, we can choose the parity of the winding numbers $\mbox{Ind}(u;c_j)$ as a dichotomy property. Following this, we introduce the set of \emph{admissible winding vectors}
\begin{equation}\label{def:w_vectors}
	\mathfrak{I}^N\uguale\{l\in\Z_2^N:\,\exists\,j,k\in\{1,\ldots,N\},\,j\neq k,\,\mbox{s.t.}\ l_j\neq l_k\}
\end{equation}
and, for $l\in\mathfrak{I}^N$ (which we fix from now on), we consider the class of paths
\[
\hat{H}_l\uguale\{u\in\hat H:\,\mbox{Ind}(u;c_j')\equiv l_j\Mod{2},\ \forall\ j=1,\ldots,N\}.
\]
Of course, the above set is not closed with respect to the weak topology of $H^1$ and so, as before, we include the collision paths in our minimization set. For $j\in\{1,\ldots,N\}$ define the set
\[
\mathfrak{Coll}_l^j\uguale\{u\in H:\,\mbox{Ind}(u;c_k')\equiv l_k\Mod{2}\ \forall\,k\neq j\ \mbox{and}\ \exists\,t\in[a,b]\ \mbox{s.t.}\ u(t)=c_j'\}
\]
i.e., the collision paths behaving like a path in $\hat{H}_l$ with respect to every centre, except for $c_j'$ in which the particle collides. In the same way, we can include two collision centres $c_{j_1}',c_{j_2}'$ defining
\[
\mathfrak{Coll}_l^{j_1,j_2}\uguale\left\lbrace u\in H\ \middle\lvert\ \begin{aligned} &\mbox{Ind}(u;c_k')\equiv l_k\Mod{2}\ \forall\,k\neq j_1,j_2\ \mbox{and} \\ &\exists\,t_1,t_2\in[a,b]\ \mbox{s.t.}\ u(t_1)=c_{j_1}', u(t_2)=c_{j_2}'\end{aligned}\right\rbrace
\]
and so on 
\[
\begin{aligned}
	&\mathfrak{Coll}_l^{j_1,j_2,j_3}\uguale\ldots, \\
	&\quad\vdots \\
	&\mathfrak{Coll}^{1,\ldots,N}_l=\mathfrak{Coll}^{1,\ldots,N}\uguale\{u\in H:\,u\ \mbox{collides in every centre}\}.
\end{aligned}
\]
At this point, we can collect together all the admissible collision paths with respect to a fixed winding vector $l\in\mathfrak{I}^N$ in the set
\[
\mathfrak{Coll}_l\uguale\bigcup\limits_{j=1}^N\mathfrak{Coll}_l^j\cup\bigcup\limits_{1\leq j_1<j_2\leq N}\mathfrak{Coll}_l^{j_1,j_2}\cup\cdots\cup\mathfrak{Coll}_l^{1,\ldots,N}
\]
and give the following result (the proof goes exactly as in \cite{ST2012}).

\begin{proposition}
	The set
	\[
	H_l\uguale\hat{H}_l\cup\mathfrak{Coll}_l
	\]
	is weakly closed in $H^1$.
\end{proposition}

Finally, we look for solution arcs which lie inside $B_R$ along their trajectory, and so it makes sense to add another constraint on them. For this reason we will restrict our investigation to the sets
\begin{equation}\label{def:k_l}
	\begin{aligned}
		\hat{K}_l&\uguale\hat{K}_l^{p_1,p_2}([a,b])\uguale\{u\in\hat{H}_l:\,|u(t)|\leq R,\ \forall\,t\in[a,b]\} \\
		K_l&\uguale K_l^{p_1,p_2}([a,b])\uguale\{u\in H_l:\,|u(t)|\leq R,\ \forall\,t\in[a,b]\}.
	\end{aligned}
\end{equation}
The following proposition guarantees that we are in the convenient setting to perform a variational argument and follows from the fact that  $K_l$ is stable under uniform convergence.

\begin{proposition}
	The set $K_l$ is weakly closed in $H^1$.
\end{proposition}

For any $u\in K_l=K_l^{p_1,p_2}([0,1])$, we take into account the Maupertuis' functional
\[
\mathcal{M}(u)=\frac12\int_0^1|\dot{u}(t)|^2\,dt\int_0^1(-1+\Veps(u(t)))\,dt
\] 
and we remark two facts:
\begin{itemize}
	\item since $\mathcal{M}$ is invariant under time re-parametrizations, we have put $a=0$ and $b=1$;
	\item actually, $\mathcal{M}=\mathcal{M}_1^\ve$, but we have omitted this dependence since we will mainly work with both $\ve>0$ and the energy fixed. When we will \emph{move} such $\ve$ or the energy, we will use the more explicit notations.
\end{itemize}
The next lemma provides a lower bound on the Maupertuis' functional and it is crucial in order to apply direct methods.

\begin{lemma}
	There exists $C>0$ such that 
	\[
	\mathcal{M}(u)\geq C>0,\quad\mbox{for every}\ u\in K_l.
	\]
\end{lemma}
\begin{proof}
	Since $u\in K_l$ then $|u(t)|\leq R$ for every $t\in[0,1]$ and so
	\[
	|u(t)-c_j'|\leq R+\ve
	\]
	for every $j=1,\ldots,N$ and for every $t\in[0,1]$. Now, recalling that 
	\[
	\mathfrak{m}=\min\limits_{j=1,\ldots,N}\min\limits_{\cerchio}U_j
	\]
	we have that
	\[
		\Veps(u(t))\geq |u(t)-c_1'|^{-\al}\mathfrak{m}\geq\frac{\mathfrak{m}}{(R+\ve)^\al},
	\]
	for every $t\in[0,1]$. Recalling \eqref{def:ve_tilde} and \eqref{eq:R}, we have that $R\in(\ve,\mathfrak{m}^{1/\al}-\ve)$ for every $\ve$, hence
	\[
	\Veps(u(t))-1\geq C>0.
	\]
	In this way, we have shown that there exists $C>0$ such that
	\[
		\mathcal{M}(u)\geq C\int_0^1|\dot{u}(t)|^2\,dt,
	\]
	for every $u\in K_l$. At this point, let us define $t^*\in(0,1)$ as the first instant at which $u$ crosses $B_\ve$. Using the H\"older inequality, we note that
	\[
		0<C_1\uguale R-\ve\leq |u(0)-u(t^*)|\leq\int_0^1|\dot{u}(t)|\,dt\leq\|\dot{u}\|_2
	\]
	and the proof is concluded.
\end{proof}

The next result, which claims the existence of a minimizer for the Maupertuis' functional in the set $K_l$, makes use of the direct method of Calculus of Variations. We take for granted the coercivity and the weakly lower semi-continuity of $\mathcal{M}$, since they are based on classical results of Functional Analysis.

\begin{proposition}\label{prop:existence}
	Assume that $N\geq 2$ and consider a function $V\in\mathcal{C}^2(\R^2\setminus\{c_1,\ldots,c_N\})$ defined as in \eqref{def:potential}. Fix $\ve\in(0,\tilde\ve)$ as in \eqref{def:ve_tilde} at page \pageref{def:ve_tilde} and consider the potential $\Veps$ defined in \eqref{def:scaling} at page \pageref{def:scaling}. Fix $R\in(\tilde{\ve},\mathfrak{m}^{1/\al}-\tilde{\ve})$ as in \eqref{eq:R} at page \pageref{eq:R} and fix $l\in\mathfrak{I}^N$. Then, for any $p_1,p_2\in\partial B_R$, the Maupertuis' functional
	\[
	\mathcal{M}(u)=\frac12\int_0^1|\dot{u}(t)|^2\,dt\int_0^1(-1+\Veps(u(t)))\,dt
	\]
	admits a minimizer $u\in K_l^{p_1,p_2}([0,1])$ at a positive level.
\end{proposition}

Now, if we show that the minimizer $u\in K_l$ verifies:
\begin{itemize}\label{def:cf_r}
	\item[$(CF)$] $u$ is collision-free;
	\item[$(R)$] $|u(t)|<R$ for every $t\in(0,1)$.
\end{itemize}  
we have that
\[
\frac{d}{d\la}\mathcal{M}(u+\la\vp)\big\lvert_{\la=0}=0\quad\mbox{for every}\ \vp\in\mathcal{C}_c^{\infty}(0,1),
\]
so that Theorem \ref{thm:maupertuis} applies and we can find a classical solution $y\colon [0,T]\to\R^2$ of the inner problem \eqref{pb:inner}. The next two sections are devoted to show respectively that $u$ joins the properties $(CF)$ and $(R)$, to finally obtain a classical solution arc for the anisotropic $N$-centre problem inside. As a starting point, we characterize the sets of colliding instants and of the times at which $|u|=R$. In particular, if we define
\begin{equation}\label{def:tctr}
	\begin{aligned}
		T_c(u)&\uguale\{t\in[0,1]:\,u(t)=c_j'\ \mbox{for some}\ j\in{1,\ldots,N}\}\sset(0,1) \\
		T_R(u)&\uguale\{t\in[0,1]:\,|u(t)|=R\}\sset[0,1],
	\end{aligned}
\end{equation}
we can easily notice that, since $\mathcal{M}(u)<+\infty$, $T_c(u)$ is a closed set of null measure and its complement $[0,1]\setminus T_c(u)$ is a union of a countable or finite number of open intervals. Moreover, when the minimizer travels along a connected component of $[0,1]\setminus( T_c(u)\cup T_R(u))$, it can be reparametrized to obtain a classical solution of the $N$-centre problem through Theorem \ref{thm:maupertuis} and the energy is conserved along this path. This is shown in the next lemma.

\begin{lemma}\label{lem:conn_comp}
	Given a minimizer $u\in K_l$ of the Maupertuis' functional $\mathcal{M}$:
	\begin{itemize}
		\item[$(i)$] $u$ verifies
		\[
		\frac12|\dot{u}(t)|^2-\Veps(u(t))=-\frac{1}{\omega^2}\quad\mbox{a.e. in}\ [0,1];
		\]
		\item[$(ii)$] if $(a,b)$ is a connected component of $[0,1]\setminus( T_c(u)\cup T_R(u))$ then $u|_{(a,b)}\in\mathcal{C}^2(a,b)$ and
		\[
		\omega^2 \ddot{u}(t)=\nabla \Veps(u(t))\quad\mbox{for every}\ t\in(a,b).
		\]
	\end{itemize}
\end{lemma}
\begin{proof}
	The proof is a consequence of the minimality of $u$ with respect to compact support variations in $[0,1]\setminus(T_c(u)\cup T_R(u))$ (see the proof of Theorem \ref{thm:maupertuis} and Lemma \ref{lem:energy}).
\end{proof}

\subsection{Qualitative properties of minimizers: absence of collisions and (self-)intersections}

In what follows we are going to provide the absence of collisions ($CF$) for a minimizer $u$ obtained in the previous subsections. In order to do that, we will carry out a local study near-collisions. Since we will be working close to the centres, the radius of the ball $B_\ve$ will play no role here and thus we fix $\ve>0$. Fix an admissible partition of the centres, that corresponds to fix $l\in\mathfrak{I}^N$ and consider a minimizer $u\in K_l$ (see \eqref{def:w_vectors} and \eqref{def:k_l} for their definitions). To start with, we show that the collisions are isolated. Recalling the definition of $T_c(u)$ in \eqref{def:tctr}, this is the content of the next lemma that, moreover, provides a \emph{Lagrange-Jacobi} identity for colliding arcs. The proof is based on a Lagrange-Jacobi-like inequality (see \cite{BCTmin} for anisotropic potentials) and it is rather classical, so it will be omitted.

\begin{lemma}\label{lem:coll_time}
	The set $T_c(u)$ is discrete and it has a finite number of elements. In particular, if the minimizer $u$ has a collision with the centre $c_j'$, the function $I(t)\uguale|u(t)-c_j'|^2$ is strictly convex in a neighbourhood of the colliding instant.
\end{lemma}

In the next two propositions we discuss some important properties of minimizers of $\mathcal{M}$, concerning the (self-)intersections at points which are different from the centres.

\begin{proposition}\label{prop:no_self_int}
	Let $u\in K_l^{p_1,p_2}$ be a minimizer of $\mathcal{M}$. Then, $u$ parametrizes a path without self-intersections at points different from the centres.
\end{proposition}
\begin{proof}
	The proof goes exactly as in \cite{ST2012}, Proposition 4.24. 
\end{proof}

\begin{remark}\label{rem:no_self_int}
	In light of the previous proposition, we can affirm that we could start this minimization process choosing among only those paths with winding index equal to $0$ or $1$ with respect to every centre, even if this choice could seem unnatural at the beginning.
\end{remark}

\begin{lemma}\label{lem:homotopy}
	Let $u\in K_l^{p_1,p_2}$ be a minimizer of $\mathcal{M}$, let $q_1=u(c)$ and $q_2=u(d)$ for some sub-interval $[c,d]\sset[0,1]$. If we define $K^{q_1,q_2}(u)$ as the weak $H^1$-closure of the space
	\[
	\hat{K}^{q_1,q_2}(u)\uguale\left\lbrace v\in H^1([c,d];\R^2)\ \middle\lvert\  \begin{aligned}
		v(c)&=q_1,\ v(d)=q_2,\ |v|\leq R,\ v\ \mbox{is homotopic} \\
		\mbox{to}&\ u|_{[c,d]}\ \mbox{in the punctured ball}\ B_R\setminus\{c_1',\ldots,c_N'\}\end{aligned}\right\rbrace,
	\]
	then
	\[
	\mathcal{M}(u|_{[c,d]})=\min\limits_{K^{q_1,q_2}(u)}\mathcal{M}.
	\]
\end{lemma}
\begin{proof}
	Assume by contradiction that there exists $w\in K^{q_1,q_2}(u)$ such that
	\[
	\min\limits_{K^{q_1,q_2}(u)}\mathcal{M}=\mathcal{M}(w)<\mathcal{M}(u|_{[c,d]}).
	\]
	The path 
	\[
	\tilde{u}\uguale\begin{cases}
		u|_{[0,c]}(t) & t\in[0,c]\cup[d,1] \\
		w  & t\in[c,d] 
	\end{cases}
	\]
	belongs to the space $K_l^{p_1,p_2}$ and minimizes $\mathcal{M}$ in that space. This is in contrast with the minimality of $u$ in the same space.
\end{proof}

\begin{proposition}\label{prop:no_int}	
	Let $u\in K_l^{p_1,p_2}$ be a minimizer of $\mathcal{M}$. Let $\tilde{l}\in\mathfrak{I}^N$,  $\tilde{p}_1,\tilde{p}_2\in\partial B_R$ and $v\in K_{\tilde l}^{\tilde{p}_1,\tilde{p}_2}$ be a minimizer of $\mathcal{M}$. Then, if $u$ intersects $v$ at least in two distinct points $q_1,q_2\in B_R\setminus\{c_1',\ldots,c_N'\}$, the portions of $u$ and $v$ between $q_1$ and $q_2$ are not homotopic paths in the punctured ball. As a consequence, if $l=\tilde{l}$, then $u$ cannot intersect $v$ more than once.
\end{proposition}
\begin{proof}
	Since the Maupertuis' functional is invariant under time-reparametrizations, to prove the assertion we can assume that there exist $q_1,q_2\in B_R\setminus\{c_1',\ldots,c_N'\}$ such that
	\[
	\begin{aligned}
		&u(c)=q_1=v(c) \\
		&u(d)=q_2=v(d),
	\end{aligned}
	\] 
	for some interval $[c,d]\sset[0,1]$. Assume by contradiction that the paths $u|_{[c,d]}$ and $v|_{[c,d]}$ are homotopic in the punctured ball $B_R\setminus\{c_1',\ldots,c_N'\}$; this means in particular that $K^{q_1,q_2}(u)=K^{q_1,q_2}(v)$ (for their definitions see the statement of Lemma \ref{lem:homotopy}). Now, again from Lemma \ref{lem:homotopy}, we deduce that
	\[
	\mathcal{M}(u;[c,d])=\min\limits_{K^{q_1,q_2}(u)}\mathcal{M}=\min\limits_{K^{q_1,q_2}(v)}\mathcal{M}=\mathcal{M}(v;[c,d]).
	\]
	For this reason, if we define the path (see Figure \ref{fig:no_int})
	\[
	\tilde{u}(t)\uguale\begin{cases}
		\begin{aligned}
			u(t)\ &\ \mbox{if}\ t\in[0,c)\cup(d,1] \\
			v(t)\ &\ \mbox{if}\ t\in[c,d]
		\end{aligned}
	\end{cases}
	\]
	we clearly have that $\tilde{u}\in K_l^{p_1,p_2}$ and
	\[
	\mathcal{M}(\tilde{u})=\mathcal{M}(u)=\min\limits_{K_l^{p_1,p_2}}\mathcal{M}.
	\]
	By Lemma \ref{lem:coll_time} the instants $c$ and $d$ belong to two connected components of $[0,1]\setminus (T_c(u)\cup T_R(u))$ and therefore Lemma \ref{lem:conn_comp} applies too. This is finally a contradiction since the path $\tilde{u}$ cannot be differentiable in $c$ and $d$ (note that $\dot{u}(c)\neq\dot{v}(c)$ for the uniqueness of solutions of Cauchy problems; the same holds at $d$ for time-reversibility). 
	
	For the case $l=\tilde{l}$ the proof is trivial, once provided Proposition \ref{prop:no_self_int}.
\end{proof}

\begin{figure}
	\begin{center}
		\begin{tikzpicture}
			\coordinate (0) at (0,0);
			\coordinate (p0) at (-0.26,1.477);
			\coordinate (p1) at (-1.5,2.59);   		
			\coordinate (p2) at (1.5,-2.59);    	
			\coordinate (tp1) at (1.92,2.29);   		
			\coordinate (tp2) at (-2.81,-1.02);   
			\coordinate (c1) at (-0.4,0.2);
			\coordinate (c2) at (0,0.5);   		
			\coordinate (c3) at (-0.4,-0.3);    		
			\coordinate (c4) at (0.1,-0.4); 
			\coordinate (c5) at (0.5,0);
			\coordinate (c6) at (0.3,0.4);
			
			\draw (0) circle (3cm);

			
			\draw[name path=arco 1] plot [smooth, tension=1] coordinates {(p1) (0.4,0) (-0.7,0)  (p2)};
			
			\draw[name path=arco 2] plot [smooth, tension=1] coordinates {(tp1) (1.2,0.7)  (tp2)};

			
			\fill (p1)
			circle[radius=1.5pt] node[above] {$p_1$};
			\fill (p2)
			circle[radius=1.5pt] node[right] {$p_2$};
			\fill (tp1)
			circle[radius=1.5pt] node[right] {$\tilde{p}_1$};
			\fill (tp2)
			circle[radius=1.5pt] node[left] {$\tilde{p}_2$};
			\fill (c1)
			circle[radius=1.0pt];
			\fill (c2)
			circle[radius=1.0pt];
			\fill (c3)
			circle[radius=1.0pt];
			\fill (c4)
			circle[radius=1.0pt];
			\fill (c5)
			circle[radius=1.0pt];
			\fill (c6)
			circle[radius=1.0pt];

			
			\fill[name intersections={of=arco 1 and arco 2}]
			(intersection-1) circle (1.5pt) node[right] {$q_1$}
			(intersection-2) circle (1.5pt) node[below] {$q_2$}
			(intersection-3) circle (1.5pt);
			
			\draw[blue,thick] plot [smooth, tension=1] coordinates {(p1) (-0.2,1) (intersection-1)};
			\draw[blue,thick] plot [smooth, tension=1] coordinates {(intersection-1) (intersection-2)};
			\draw[blue,thick] plot [smooth, tension=1] coordinates {(intersection-2) (intersection-3)};		
			\draw[blue,thick] plot [smooth, tension=1] coordinates {(intersection-3) (0,-0.98) (p2)};

			
			\node at (3,1.6) {$\partial B_R$};
			\node at (1.8,0.9) {$v$};
			\node at (-0.9,0.3) {$u$};
			\node[blue] at (0,-1.4) {$\tilde{u}$};
			
		\end{tikzpicture} 
		\caption{Illustration of the case of Proposition \ref{prop:no_int}}\label{fig:no_int}
	\end{center}	
\end{figure}
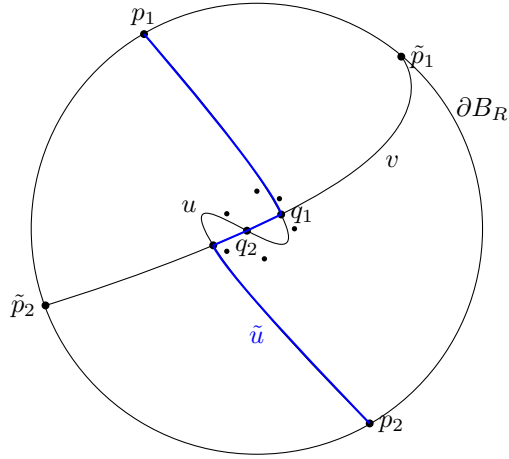

At this point we are ready to start a local analysis in order to rule out the presence of collisions with the centres. Let us now assume that the minimizer $u$ has a collision with the centre $c_j'$ at time $t_0$. By means of Lemma \ref{lem:coll_time}, we have that there exist $c,d\in[0,1]$ such that
\begin{itemize}
	\item $c<t_0<d$ and $t_0$ is the unique instant of collision of $u$ in $[c,d]$;
	\item the inertial moment $I(t)=|u(t)-c_j'|^2$ is strictly convex in $[c,d]$.
\end{itemize}
We define $\bar{p}_1=u(c)$ and $\bar{p}_2=u(d)$. Since $u\in\mathcal{C}([c,d];\R^2)$, then there exists $r^*>0$ such that 
\begin{equation}\label{eq:continuity}
	|u(t)-c_k'|\geq r^*>0\quad \mbox{for every}\ t\in[c,d]\ \mbox{and for every}\ k\neq j
\end{equation}
and, without loss of generality, we can assume that $\bar{p}_1,\bar{p}_2\in\partial B_r(c_j')$, for some $r<r^*$.

Since we are getting close to the collision, it makes sense to \emph{localize} the potential and to write
\[
\Veps(y)=V_j(y- c_j')+f^j(y),\quad\mbox{with}\ f^j(y)\uguale\sum\limits_{k\neq j}V_k(y-c_k').
\] 
Notice that inside the ball $B_r(c_j')$ the quantity $f^j(y)$ is smooth and bounded.

Let us introduce the space
\[
\hat{\mathcal{K}}_l^{\bar{p}_1,\bar{p}_2}\uguale\left\lbrace v\in H^1([c,d];\R^2)\ \middle\lvert\ \ \begin{aligned}
	&v(c)=\bar{p}_1,\ v(d)=\bar{p}_2,\ v(t)\neq c_j'\ \forall\ t\in[c,d],\ \forall\ j& \\
	&\mbox{the function}\ G_v(t)\uguale\begin{cases}
		u(t)\quad\mbox{if}\ t\in[0,c)\cup(d,1] \\
		v(t)\quad\mbox{if}\ t\in[c,d]
	\end{cases}& \\
	&\mbox{belongs to}\ \hat{K}_l^{p_1,p_2}&
\end{aligned}\right\rbrace
\]
and its weak $H^1$-closure
\[
\mathcal{K}_l^{\bar{p}_1,\bar{p}_2}\uguale\hat{\mathcal{K}}_l^{\bar{p}_1,\bar{p}_2}\cup\left\lbrace v\in H^1([c,d];\R^2):\,v(c)=\bar{p}_1,\ v(d)=\bar{p}_2,\ G_v\in\mathfrak{Coll}_l\right\rbrace
\]
and restrict the Maupertuis' functional to $\mathcal{K}_l^{\bar{p}_1,\bar{p}_2}$ in this way
\[
\mathcal{M}_l^{\bar{p}_1,\bar{p}_2}\colon\mathcal{K}_l^{\bar{p}_1,\bar{p}_2}\to\R\cup\{+\infty\},\quad\mathcal{M}_l^{\bar{p}_1,\bar{p}_2}(v)=\frac12\int_c^d|\dot{v}(t)|^2\,dt\int_c^d\left(-1+\Veps(v(t))\right)\,dt.
\]
We can repeat the proof of Subsection \ref{subsec:dir_meth} and show that $\mathcal{M}_l^{\bar{p}_1,\bar{p}_2}$ admits a minimizer in $\mathcal{K}_l^{\bar{p}_1,\bar{p}_2}$ at a positive level. Moreover, from Lemma \ref{lem:homotopy}, this minimizer is nothing but $v\uguale u|_{[c,d]}$.

For the sake of simplicity assume $c_j'=0$. 

\begin{proposition}\label{prop:pert_coll}
	The following behaviour holds:
	\[
	\Veps(y)=V_j(y)+C+\mathcal{O}(|y|),\ \mbox{as}\ y\to 0^+,
	\]
	for some constant $C>0$. In particular, when $|y|$ is sufficiently small, the problem is a small perturbation of an anisotropic Kepler problem driven by the $-\al_j$-homogeneous potential $V_j$.
\end{proposition}
\begin{proof}
    The proof follows the approach of Proposition \ref{prop:pert}. In particular, defining $r^*>0$ as in \eqref{eq:continuity}, for  $r\in(0,r^*)$ and for $|y|\leq r$ we can write
	\[
	\Veps(y)=V_j(y)+C+rG_r(|y|),
	\]
	with $G_r$ uniformly bounded with respect to $r$.
\end{proof}

At this point we need a result from \cite{BCTmin} on the properties of minimal collision orbits for a perturbed anisotropic Kepler problem. In order to take it into account, we need to introduce some further notations. Let $r^*$ be as in \eqref{eq:continuity}; for $r\in(0,r^*)$ and $q\in\partial B_r$ we define the set of $H^1$-colliding paths on a generic real interval $[c,d]\sset\R$
\[
H_{coll}^q\uguale\left\lbrace w\in H^1([c,d];\R^2):\, w(c)=q,\ w(d)=0,\ |w(t)|\leq r,\ \forall\,t\in[c,d]\right\rbrace.
\]
Moreover, for a potential $\Veps\in\mathcal{C}^2(B_r\setminus\{0\})$ which is a perturbation of an anisotropic potential as in Proposition \ref{prop:pert_coll}, consider the Maupertuis' functional
\[
\mathcal{M}(w)=\frac12\int_c^d|\dot{w}(t)|^2\,dt\int_c^d(-1+\Veps(w(t)))\,dt
\]
for $w\in H_{coll}^q$. Up to choose a smaller $r^*$, the authors proved the following result; in order to ease the notation, we will denote a minimal non-degenerate central configuration for $U_j$ as $\vt^*$ instead of $\vt_j$ (see \eqref{hyp:V}, page \pageref{hyp:V}).

\begin{lemma}[Theorem 5.2,\cite{BCTmin}]\label{lem:BCTmin}
	Let $\vt^*\in\cerchio$ be a minimal non-degenerate central configuration for $U_j$. There exists $r^*>0$ and $\delta>0$ such that, for every $q=re^{i\vt}$ with $r<r^*$ and $\vt\in(\vt^*-\delta,\vt^*+\delta)$ there exists a unique minimizer of the Maupertuis' functional in the set of colliding paths $H_{coll}^q$. In particular, this path cannot leave the cone emanating from the origin and bounded by the arc-neighbourhood $(\vt^*-\delta,\vt^*+\delta)$.
\end{lemma}

The presence of this \emph{foliation} of minimal arcs in a cone spanned by $\vt_j$, together with Proposition \ref{prop:no_int}, suggests to choose one of this paths and to use it as a \emph{barrier}, in order to determine a region of the ball $B_r$ in which a minimizer with end-points on $\partial B_r$ has to be confined. Indeed, in order to rule out the presence of collisions for a minimizer in $\mathcal{K}_l^{\bar{p}_1,\bar{p}_2}$, we aim to follow the ideas contained in a result from \cite{BTVplanar}, which holds true for minimizers that do not leave a prescribed angular sector. For a $r\in(0,r^*)$, a potential $\Veps\in\mathcal{C}^2(B_r\setminus\{0\})$ as in Proposition \ref{prop:pert_coll} and $T>0$, introduce the action functional $\mathcal{A}_T\colon H^1([0,T];\R^2)\to \R\cup\{+\infty\}$ such that
\[
\mathcal{A}_T(x)\uguale\int_0^T\left(\frac12|\dot{x}(t)|^2+\Veps(x(t))-1\right)\,dt.
\]

\begin{definition}
	We say that $x\in H^1([0,T];\R^2)$ is a \emph{fixed-time Bolza minimizer} associated with the endpoints $x_1=x(0),x_2=x(T)$, if, for every $y\in H^1([0,T];\R^2)$ there holds
	\[
	y(0)=x_1,\ y(T)=x_2\implies \mathcal{A}_T(x)\leq\mathcal{A}_T(y).
	\]
\end{definition}

As announced, we recall an important result from \cite{BTVplanar}. Note that Proposition \ref{prop:pert_coll} guarantees the applicability of Theorem 2 in \cite{BTVplanar} to problems driven by our family of potentials $\Veps$. For this reason, we restate below such result in our setting.

\begin{lemma}[Theorem 2, \cite{BTVplanar}]\label{lem:BTV}
	Fix $\ve>0$ and $r\in(0,r^*)$, with $r^*>0$ defined in \eqref{eq:continuity}. For a minimal non-degenerate central configuration $\vt^*\in\cerchio$ for $V_j$, define the set
	\[
	\Theta_j\uguale\left\lbrace \vt\in\R:\,\vt=\vt^*+2k\pi,\ \mbox{for some}\ k\in\Z\right\rbrace.
	\]
	Then, for every $\vt^-<\vt^+\in\Theta$ there exists $\bar{\al}(U_j,\vt^-,\vt^+)\in(0,2)$ such that if $\al_j>\bar{\al}$ all the fixed-time Bolza minimizers in the angular sector $[\vt^-,\vt^+]$ are collision-less.
\end{lemma} 

As a first step, we show that the previous lemma can be extended for those $H^1$-paths with fixed ends which minimize the Maupertuis' functional instead of the action functional. Note that the previous result holds for every fixed energy. Moreover, with the same proof of Subsection \ref{subsec:dir_meth}, one can prove the existence of a minimizer for the Maupertuis' functional 
\[
\mathcal{M}_h(u)=\frac12\int_0^1|\dot{u}|^2\int_0^1(-h+V(u))
\]
in the space of the $H^1$-paths which join two points within the sector $[\vt^-,\vt^+]$, for $h\in\R$.

\begin{lemma}\label{lem:BTV_maup}
	In the same setting of Lemma \ref{lem:BTV}, if $h\in\R$ and if $\al_j>\bar\al(U_j,\vt^-,\vt^+)$, then all the minimizers of the Maupertuis' functional $\mathcal{M}_h$ within the sector $[\vt^-,\vt^+]$ are collision-less.
\end{lemma}
\begin{proof}
	Assume that $u\in H^1([0,1];\R^2)$ minimizes the Maupertuis' functional in the set of the $H^1$-paths which join two points $q_1,q_2$ within the sector $[\vt^-,\vt^+]$ and assume also that $u$ has a collision with the origin. If we define $x(t)\uguale u(\omega t)$, with 
	\[
	\omega\uguale\left(\frac{\int_0^1(-h+\Veps(u))}{\frac12\int_0^1|\dot{u}|^2}\right)^{1/2}>0
	\]
	then, from Theorem \ref{thm:maupertuis}, we know that $x$ solves
	\[
	\begin{cases}
		\ddot{x}(t)=\nabla\Veps(x(t))  & t\in[0,1/\omega] \\
		\frac12|\dot{x}(t)|^2-\Veps(x(t))=-h  & t\in[0,1/\omega] \\
		x(0)=q_1,\ x(1/\omega)=q_2. &
	\end{cases}
	\]
	At this point we define $T=1/\omega$ and we find the fixed-time Bolza minimizer of the action functional associated with the sector $[\vt^-,\vt^+]$ and we define 
	\[
	H_T\uguale\left\lbrace y\in H^1([0,T];\R^2):\,y(0)=q_1,\ y(T)=q_2,\ y(t)\in[\vt^-,\vt^+]\ \forall\,t\in[0,T]\right\rbrace.
	\] 
	We call this minimizing path
	\[
	\bar{x}=\arg\min\limits_{y\in H_T}\mathcal{A}_T(y)
	\]
	and, by Lemma \ref{lem:BTV}, we know that it cannot collide with the origin; this also proves that $x\neq\bar{x}$.
	
	Now, we know that $x(t)=u(t/ T)$ for all $t\in[0,T]$, and so we can compute
	\[
	\begin{aligned}
		\mathcal{A}_T(x)&=\int_0^T\left(\frac12|\dot{x}(t)|^2+\Veps(x(t))-h\right)\,dt \\
		&=\int_0^1\left(\frac{1}{2 T}|\dot{u}(s)|^2+T\Veps(u(s))-Th\right)\,ds
	\end{aligned}
	\]
	and, from the conservation of the energy for $u$ and the definition of $\omega=1/T$, we can find that (see also Proposition \ref{prop:app1} in Appendix \ref{app:var})
	\[
	\mathcal{A}_{T}(x)=\int_0^1\frac{1}{T}|\dot{u}(s)|^2\,ds=\sqrt{2}\left(\int_0^1|\dot{u}(s)|^2\,ds\int_0^1(-h+\Veps(u(s)))\,ds\right)^{1/2}=2\sqrt{\mathcal{M}_h(u)}.
	\]
	At this point, since the Maupertuis functional is invariant over time-reparametrizations we have
	\[
	2\sqrt{\mathcal{M}_h(u)}=\min\limits_{T>0}\min\limits_{y\in H_T}\mathcal{A}_T(y)\leq\min\limits_{y\in H_T}\mathcal{A}_T(y)=\mathcal{A}_T(\bar x)<\mathcal{A}_T(x)=2\sqrt{\mathcal{M}_h(u)},
	\]
	which is a contradiction since $\bar{x}$ is collision-less.
\end{proof}

In the next result we show that it is possible to improve the previous lemma. In particular, a sequence of minimal paths cannot accumulate to a collision path, thanks to a uniform bound on the distance from the origin.

\begin{lemma}\label{lem:uniform}
	In the same setting of Lemma \ref{lem:BTV}, there exists $\bar r>0$ such that, for any $r_0\in(0,\bar r)$, for every $k\in\N_{\geq 1}$, for every sector $[\vt^-,\vt^+]$ such that $\vt^+-\vt^-=2k\pi$, for every $\al_j>\bar\al(U_j,\vt^-,\vt^+)$, there exists $\delta>0$ such that, for every $q_1,q_2\in\partial B_{r_0}$ the Maupertuis' minimizer $u$ considered in Lemma \ref{lem:BTV_maup} from $q_1$ to $q_2$ in the sector $[\vt^-,\vt^+]$ is such that
	\[
	\min\limits_{t\in[0,1]}|u(t)|>\delta r_0.
	\]
\end{lemma}
\begin{proof}
	We argue by contradiction; it is not restrictive to assume instead the following: 
	\begin{itemize}
		\item there exists $r_n\to 0^+$ sequence of positive real numbers;
		\item fix $k\in\N_{\geq 1}$;
		\item fix $\vt^*\in\cerchio$ minimal non-degenerate central configuration for $U_j$ (this is not restrictive since $U_j$ admits just a finite number of them);
		\item there exists $\al_j>\bar\al(U_j,\vt^*,\vt^*+2k\pi)$;
		\item take $\delta_n\to 0^+$ sequence of positive real numbers;
		\item take two sequences of points $(q_1^n),(q_2^n)\sset(\partial B_{r_n})$;
		\item consider the sequence of minimizers $(u_n)$ of the Maupertuis' functional
		\[
		\mathcal{M}(u_n)=\frac12\int_0^1|\dot{u}_n|^2\int_0^1(-1+\Veps(u_n)),
		\]
		every one of them respectively in the space
		\[
		H^n\uguale\{u_n\in H^1([0,1];\R^2):\,u_n(0)=q_1^n,\ u_n(1)=q_2^n,\ |u_n|\leq r_n \}
		\]
		and within the sector $[\vt^*,\vt^*+2k\pi]$,
	\end{itemize}
	such that
	\[
	\min\limits_{t\in[0,1]}|u_n(t)|\leq\delta_nr_n.
	\]
	Define the blow-up sequence
	\[
	v_n(t)\uguale\frac{1}{r_n}u_n(t)\quad\mbox{for}\ t\in[0,1],\ \mbox{for every}\ n\in\N
	\]
	which, for every $n\in\N$, verifies the following:	
	\begin{equation}\label{lem:uniform_step}
		\begin{cases}
			\bar{q}_1^n\uguale v_n(0),\ \bar{q}_2^n\uguale v_n(1)\in\partial B_1; \\
			|v_n(t)|\leq 1\ \mbox{for every}\ t\in[0,1]; \\
			\min\limits_{t\in[0,1]}|v_n(t)|\leq\delta_n.
		\end{cases}
	\end{equation}	
	Recalling the behaviour of $\Veps$ (see Proposition \ref{prop:pert_coll}), observe that, if we fix $y\in\R^2\setminus\{0\}$ we can compute
	\[
	\Veps(r_ny)=r_n^{-\al_j}V_j(y)+C+\mathcal{O}(r_n)=r_n^{-\al_j}\left(V_j(y)+r_n^{\al_j}C+\mathcal{O}(r_n^{\al_j+1})\right)
	\]
	as $n\to+\infty$. In this way we have
	\[
	\begin{aligned}
		\mathcal{M}(u_n)&=\mathcal{M}(r_nv_n)=\frac12\int_0^1|r_n\dot{v}_n|^2\int_0^1\left(-1+\Veps(r_nv_n)\right) \\
		&=r_n^2\frac12\int_0^1|\dot{v}_n|^2\int_0^1r_n^{-\al_j}\left(-r_n^{\al_j}+V_j(v_n)+r_n^\al C+\mathcal{O}(r_n^{\al_j+1})\right) \\
		&=r_n^{2-\al_j}\frac12\int_0^1|\dot{v_n}|^2\int_0^1\left(V_j(v_n)+\mathcal{O}(r_n^{\al_j})\right)
	\end{aligned}
	\]
	and so, if we define 
	\[
	\bar{\mathcal{M}}(v_n)\uguale\frac12\int_0^1|\dot{v}_n|^2\int_0^1\left(V_j(v_n)+\mathcal{O}(r_n^{\al_j})\right)
	\]
	we have shown that
	\[
	\bar{\mathcal{M}}(v_n)=r_n^{\al_j-2}\mathcal{M}(u_n),\quad\mbox{for every}\ n\in\N.
	\]
	Now, since $\bar{\mathcal{M}}(v_n)$ and $\mathcal{M}(u_n)$ are proportional and $u_n$ minimizes $\mathcal{M}$ in $H^n$, if we define 
	\[
	\bar{H}^n\uguale\{v_n\in H^1([0,1];\R^2):\,v_n(0)=\bar{q}_1^n,\ v_n(1)=\bar{q}_2^n,\ |v_n|\leq 1\}
	\]
	we easily deduce that 
	\[
	\bar{\mathcal{M}}(v_n)=\min\limits_{\bar{H}^n}\bar{\mathcal{M}}.
	\]
	At this point, we want to show that $(v_n)$ admits a weak limit in the $H^1$ topology. Since $V_j$ is bounded from below in $\cerchio$ we have that there exists $C_1>0$
	\[
	\bar{\mathcal{M}}(v_n)\geq C_1\int_0^1|\dot{v}_n|^2,\quad\mbox{for every}\ n\in\N.
	\]
	On the other hand, since $\mathcal{O}(r_n^{\al_j})$ is uniformly bounded when $n\to+\infty$ by a constant $C_2>0$, we have that there exists $C_3>0$ such that
	\[
	\bar{\mathcal{M}}(v_n)=\min\limits_{\bar{H}^n}\bar{\mathcal{M}}\leq\min\limits_{v\in\bar{H}^n}\frac12\int_0^1|\dot{v}|^2\int_0^1V_j(v)+C_2\leq C_3.
	\]
	Moreover, the sequence $(v_n)$ is uniformly bounded by 1 and so its $L^2$-norm is too. For this reason, we deduce that there exists $v_0\in H^1$ such that $v_n\wconv v_0$ in the $H^1$-topology and thus uniformly; in particular, from \eqref{lem:uniform_step} and the uniform convergence we have that
	\[
	\begin{cases}
		\bar{q}_1\uguale v_0(0),\ \bar{q}_2\uguale v_0(1)\in\partial B_1; \\
		|v_0(t)|\leq 1,\ \mbox{for every}\ t\in[0,1]; \\
		\min\limits_{t\in[0,1]}|v_0(t)|=0.
	\end{cases}
	\]
	In other words, we have shown that the blow-up limit $v_0$ is a collision path in the space 
	\[
	\bar H\uguale\{v\in H^1([0,1];\R^2):\,v(0)=\bar{q}_1,\ v(1)=\bar{q}_2,\ |v|\leq 1\}
	\]
	in the sector $[\vt^*,\vt^*+2\bar{k}\pi]$. For this reason, it is enough to show that $v_0$ minimizes the Maupertuis functional
	\[
	\bar{\mathcal{M}}_0(v_0)=\frac12\int_0^1|\dot{v}_0|^2\int_0^1V_j(v_0)
	\]
	in the space $\bar H$; indeed, we would reach a contradiction thanks to Lemma \ref{lem:BTV_maup}, since $\al_j>\bar\al(U_j,\vt^*,\vt^*+2k\pi)$ and a minimizer cannot have collisions.
	
	From the Fatou's lemma we have that
	\[
	\bar{\mathcal{M}}_0(v_0)=\frac12\int_0^1|\dot{v}_0|^2\int_0^1V_j(v_0)\leq\liminf\limits_{n\to+\infty}\bar{\mathcal{M}}(v_n);
	\]
	on the other hand, since $v_n$ minimizes $\bar{\mathcal{M}}$ in $\bar{H}_n$ for every $n\in\N$, we have that
	\[
	\bar{\mathcal{M}}(v_n)\leq\bar{\mathcal{M}}(v_0)\leq\frac12\int_0^1|\dot{v}_0|^2\int_0^1V_j(v_0)+C_4r_n^\al,
	\]
	for some $C_4>0$ and for every $n\in\N$. In this way, we also have that
	\[
	\liminf\limits_{n\to+\infty}\bar{\mathcal{M}}(v_n)\leq\frac12\int_0^1|\dot{v}_0|^2\int_0^1V_j(v_0)=\bar{\mathcal{M}}_0(v_0)
	\]
	and so $v_n\to v_0$ strongly in $H^1$. This shows that $v_0$ is a minimizer in $\bar H$ and concludes the proof.
\end{proof}

At this point, we want to prove something stronger than the previous lemma, which will involve Lemma \ref{lem:BCTmin}. Indeed, our idea is to show that it is possible to extend Lemma \ref{lem:uniform} to those \emph{sectors} that are determined by two minimal arcs of the \emph{foliation} provided in Lemma \ref{lem:BCTmin}. We are interested in those \emph{curved sectors} which have as barriers one minimal arc and its $2k\pi$-copy for some $k\in\N_{\geq 1}$. Note that in \cite{BCTmin}, the authors give a particular characterization of such foliation: it is possible to parametrize every minimal arc with respect to its distance from the origin, thanks to a monotonicity property of the radial variable (see \cite[Lemma 4.3]{BCTmin}). Recalling that $\vt^*$ is a minimal non-degenerate central configuration for $U_j$, we consider the unique minimal arc $\gamma^*$, parametrized as the polar curve $\gamma^*(r)=(r,\vp^*(r))$, such that $\vp^*(r_0)=\vt^*$. For $k\in\N_{\geq 1}$, we can define
\[
\Sigma(\vt^*,k)\uguale\left\lbrace (r,\vt(r)):\,\vp^*(r)\leq\vt(r)\leq\vp^*(r)+2k\pi,\ \mbox{for}\ 0\leq r\leq r_0\right\rbrace
\]
and we are able to prove the following result.

\begin{lemma}\label{lem:uniform_curved}
	In the same setting of Lemma \ref{lem:BTV}, there exists $r^*>0$ such that, for every $r_0\in(0,r^*)$, for every $k\in\N_{\geq 1}$, for every $\al_j>\bar{\al}(U_j,\vt^*,\vt^*+2k\pi)$, there exists $\delta>0$ such that, for every $q_1,q_2\in\Sigma(\vt^*,k)\cap\partial B_{r_0}$, the Maupertuis' minimizer $u$ which connects $q_1$ and $q_2$ is such that:
	\begin{itemize}
		\item[$(i)$] $u$ belongs pointwisely to the sector $\Sigma(\vt^*,k)$;
		\item[$(ii)$] $u$ verifies
		\[
		\min\limits_{t\in[0,1]}|u(t)|>\delta r_0.
		\]	
	\end{itemize}
\end{lemma}
\begin{proof}
	We start with the proof of ($ii$). Following the same technique used in the proof of Lemma \ref{lem:uniform}, assume by contradiction that:
	\begin{itemize}
		\item there exists $r_n\to 0^+$ sequence of positive real numbers and, without loss of generality, assume that $r_n\leq r^*$ for $n$ sufficiently large, with $r^*>0$ as in Lemma \ref{lem:BCTmin};
		\item fix $k\in\N$;
		\item fix $\vt^*\in\cerchio$ minimal non-degenerate central configuration for $U_j$ (this is not restrictive since $U_j$ admits just a finite number of them);
		\item there exists $\al_j>\bar\al(U_j,\vt^*,\vt^*+2k\pi)$;
		\item take $\delta_n\to 0^+$ sequence of positive real numbers;
		\item define the sequence of curved sectors
		\[
		\Sigma_n\uguale\left\lbrace (r,\vt(r)):\,\vp^*(r)\leq\vt(r)\leq\vp^*(r)+2k\pi,\ \mbox{for}\ 0\leq r\leq r_n\right\rbrace,
		\]
		where $\gamma^*(r)=(r,\vp^*(r))$ is the polar curve which parametrizes the unique minimal arc of the foliation provided in Lemma \ref{lem:BCTmin}, such that $\vp^*(r^*)=\vt^*$;
		\item take two sequences of points $(q_1^n),(q_2^n)\sset(\Sigma_n\cap\partial B_{r_n})$;		
		\item consider the sequence of minimizers $(u_n)$ of the Maupertuis' functional
		\[
		\mathcal{M}(u_n)=\frac12\int_0^1|\dot{u}_n|^2\int_0^1(-1+\Veps(u_n)),
		\]
		every one of them respectively in the space
		\[
		H^n\uguale\{u_n\in H^1([0,1];\R^2):\,u_n(0)=q_1^n,\ u_n(1)=q_2^n,\ |u_n|\leq r_n \}
		\]
		and within the curved sector $\Sigma_n$, requiring that every $u_n$ satisfies
		\[
		\min\limits_{t\in[0,1]}|u_n(t)|\leq\delta_nr_n.
		\] 
	\end{itemize}
	Define the blow-up sequence $(v_n)$ as in the proof of Lemma \ref{lem:uniform}, which again verifies \eqref{lem:uniform_step} for every $n\in\N$. In the same way, one can prove that every $v_n$ (at least for $n$ large) minimizes the functional 
	\[
	\bar{\mathcal{M}}(v_n)\uguale\frac12\int_0^1|\dot{v}_n|^2\int_0^1(V_j(v_n)+\mathcal{O}(r_n^{\al_j}))
	\]
	in the space
	\[
	\bar{H}^n\uguale\{v_n\in H^1([0,1];\R^2):\,v_n(0)=\bar{q}_1^n,\ v_n(1)=\bar{q}_2^n,\ |v_n|\leq 1\}.
	\]
	Moreover, defining the angular variable $\vp_n^*(r)\uguale\vp^*(r_nr)$ and the blow-up sector
	\[
	\bar{\Sigma}_n\uguale\{(r,\vt(r)):\,\vp_n^*(r)\leq\vt(r)\leq\vp_n^*(r)+2k\pi,\ \mbox{for}\ 0\leq r\leq 1\},
	\]
	one can easily verify that $v_n\in\bar{\Sigma}_n$ for every $n\in\N$.
	
	At this point, with the same technique of Lemma \ref{lem:uniform}, one can prove that $v_n\to v_0$ uniformly in $[0,1]$, with $v_0$ minimizer of the functional 
	\[
	\bar{\mathcal{M}}_0(v_0)=\frac12\int_0^1|\dot{v}_0|^2\int_0^1V_j(v_0)
	\]
	in the space
	\[
	\bar H\uguale\{v\in H^1([0,1];\R^2):\,v(0)=\bar{q}_1,\ v(1)=\bar{q}_2,\ |v|\leq 1\},
	\]
	for some $\bar{q}_1,\bar{q}_2\in\partial B_1$ and such that
	\begin{equation}\label{lem:uniform_curved_step1}
		\min\limits_{t\in[0,1]}|v_0(t)|=0.
	\end{equation}
	Moreover, from Lemma \ref{lem:BCTmin}, since $r_n\to 0^+$, we have that the sequence of functions $\vp_n^*=\vp_n^*(r)$ uniformly converges to $\vt^*$ on the $r$-interval $[0,1]$	and so
	\[
	\bar{\Sigma}_n\to[\vt^*,\vt^*+2k\pi]\quad\mbox{as}\ n\to+\infty.
	\]
	This means that $v_0$ minimizes $\bar{\mathcal{M}}$ in $\bar{H}$ within the sector $[\vt^*,\vt^*+2k\pi]$ and, thanks to \eqref{lem:uniform_curved_step1}, has a collision. This is a contradiction for Lemma \ref{lem:BTV_maup} and proves ($ii$).
	
	In order to prove ($i$) it is enough to observe that a minimizer of the Maupertuis' functional $\mathcal{M}$ with endpoints in the sector $\Sigma(\vt^*,k)$ cannot leave this sector. Indeed, $\Sigma(\vt^*,k)$ has a minimal collision arc and its $2k\pi$-copy as boundary; this arcs act as a barrier, since Proposition \ref{prop:no_int} applies also in this context and a Bolza minimizer cannot intersect another minimal arc more than once.
\end{proof}

We now we extend the previous local study to a global setting,which takes into account all the other centres. In order to do this, we need to show that the local minimization process provides two minimizers which do not collide in $c_j'$ and such that, if juxtaposed, have winding number equal to 1 with respect to $c_j'$. In this way, if one takes a minimizer $u\in K_l$ and assumes that $u$ collides in $c_j'$, then a contradiction arises. Indeed, the portion of $u$ close enough to $c_j'$ must correspond to one of the two local minimizers above, depending on if $l_j=0$ or $l_j=1$.
\begin{theorem}\label{thm:indice_uno}
	In the same setting of Lemma \ref{lem:BTV}, there exists $r^*>0$ such that, for every $r_0\in(0,r^*)$, for every $\al>\bar{\al}(U_j,\vt^*,\vt^*+4\pi)$, there exists $\delta>0$ such that, for every $q_1,q_2\in\partial B_{r_0}$, there exist two Maupertuis' minimizers $u_1$ and $u_2$ which connect $q_1$ and $q_2$ such that
	\begin{itemize}
		\item[$(i)$] for every $i=1,2$ we have
		\[
		\min\limits_{t\in[0,1]}|u_i(t)|>\delta r_0;
		\]
		\item[$(ii)$] the juxtaposition $u$ of $u_1$ and $u_2$ is a closed path which has winding number 1 with respect to the origin, up to choose a suitable time-parametrization.
	\end{itemize}
\end{theorem}
\begin{proof}
	Take $q_1=r_0e^{i\vt_1},q_2=r_0e^{i\vt_2}\in\partial B_{r_0}$ and, without loss of generality, assume that $q_1,q_2\in\Sigma(\vt_j,1)$ so that, in particular $|\vt_1-\vt_2|<2\pi$. Moreover, it is not restrictive to assume that $\vt_1<\vt_2$. By Lemma \ref{lem:uniform_curved} there exists a Maupertuis' minimizer $u_1$ which connects $q_1$ and $q_2$ and verifies properties (i) and (ii) of such lemma. At this point, define $\tilde{q}_1\uguale r_0e^{i(\vt_1+2\pi)}$ which, of course, coincides with $q_1$ in the Euclidean space, but not with respect to the curved sectors. Indeed, we have that $\tilde{q}_1\in\Sigma(\vt_j,2)\setminus\Sigma(\vt_j,1)$ and, of course, also $q_2\in\Sigma(\vt_j,2)$ (see Figure \ref{fig:cur_sec}). For this reason, again from Lemma \ref{lem:uniform_curved}, we deduce the existence of the second minimal arc $u_2$, which connects $q_2$ and $\tilde{q}_1$ with the same properties of $u_1$. Consider the juxtaposition $u$ of $u_1$ and $u_2$, which, of course, is a closed curve from $q_1$ to itself. Since both $u_1$ and $u_2$ are collision-less, the winding number of $u$ with respect to the origin is 1.	
\end{proof}

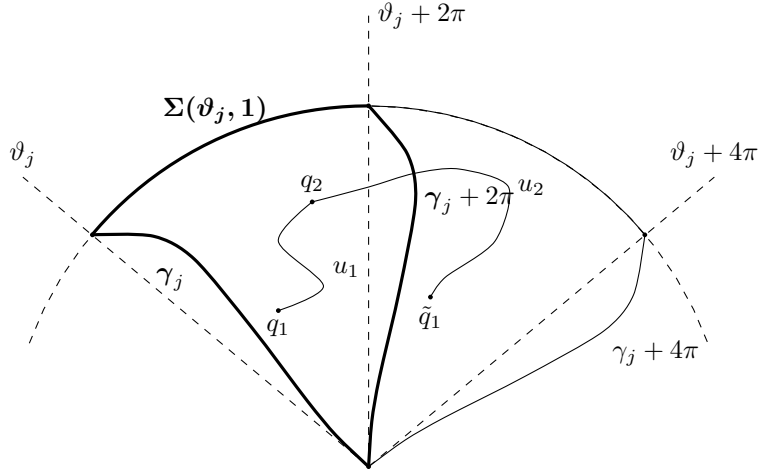
\begin{figure}
	\centering
	\begin{tikzpicture}[scale=0.6]
		\coordinate (O) at (0,0);
		\coordinate (x) at (0,8);
		\coordinate(x1) at (0.85,6.95);
		\coordinate (x2) at (1.04,5.9);
		\coordinate (x3) at (0.69,3.94);
		\coordinate (x4) at (0.24,1.98);
		\coordinate (x5) at (0.07,1); 
		\coordinate (-delta) at (-7.66,6.42);
		\coordinate (+delta) at (7.66,6.42);
		\coordinate (delta) at (0,10);
		\coordinate (p0) at (-6.12,5.14);
		\coordinate (p01) at (-4.77,5.11);
		\coordinate (p02) at (-3.86,4.6);
		\coordinate (p03) at (-2.57,3.06);	
		\coordinate (p04) at (-1.36,1.46);
		\coordinate (p05) at (-0.72,0.69);	
		\coordinate (p00) at (6.12,5.14);
		\coordinate (p001) at (5.87,3.81);
		\coordinate (p002) at (5.2,3);
		\coordinate (p003) at (3.46,2);
		\coordinate (p004) at (1.68,1.09);
		\coordinate (p005) at (0.81,0.59);
		\coordinate (u1) at (-2,3.46);
		\coordinate (u11) at (-1,4);
		\coordinate (u12) at (-2,5);
		\coordinate (ut1) at (1.37,3.76);
		\coordinate (u2) at (-1.25,5.87);
		\coordinate (u21) at (0,6.2);
		\coordinate (u22) at (1,6.5);
		\coordinate (u23) at (2,6.6);
		\coordinate (u24) at (3.1,6.2);
		\coordinate (u25) at (2.8,5);
		\coordinate (u26) at (1.7,4.2);
		
		\draw[dashed] (O)--(-delta);
		\draw[dashed] (0)--(delta);
		\draw[dashed] (0)--(+delta);	
		
		\draw [domain=-50:50] plot ( {8*sin(\x)},{8*cos(\x)});
		\draw [dashed,domain=-70:70] plot ( {8*sin(\x)},{8*cos(\x)});  
		
		\fill (O) circle[radius=1.5pt];
		\fill (p0) circle[radius=1.5pt];
		\fill (p00) circle[radius=1.5pt];
		\fill (x) circle[radius=1.5pt];
		\fill (u1) circle[radius=1.5pt];
		\fill (ut1) circle[radius=1.5pt];
		\fill (u2) circle[radius=1.5pt];
		
		\node[above] at (-delta) {$\vartheta_j$}; 	
		\node[right] at (delta) {$\vartheta_j+2\pi$}; 
		\node[above] at (+delta) {$\vartheta_j+4\pi$}; 	
		\node[below left] at (p02) {$\boldsymbol\gamma_j$};
		\node[right] at (x2) {$\boldsymbol\gamma_j+2\pi$};
		\node[below right] at (p002) {$\gamma_j+4\pi$};
		\node[above] at (-3.4,7.4) {$\boldsymbol{\Sigma(\vt_j,1)}$};
		\node[below] at (u1) {$q_1$};
		\node[above] at (u2) {$q_2$};
		\node[below] at (ut1) {$\tilde{q}_1$};
		\node[above right] at (u11) {$u_1$};
		\node[right] at (u24) {$u_2$};

		\draw[very thick] plot [smooth] coordinates {(x) (x1) (x2) (x3) (x4) (x5) (0)};
		
		\draw[very thick] plot [smooth] coordinates {(p0) (p01) (p02) (p03) (p04) (p05) (0)};
		
		\draw [very thick,domain=-50:0] plot ( {8*sin(\x)},{8*cos(\x)});
		
		\draw plot [smooth] coordinates {(p00) (p001) (p002) (p003) (p004) (p005) (0)};
		
		\draw plot [smooth] coordinates {(u1) (u11) (u12) (u2)};
		
		\draw plot [smooth] coordinates {(u2) (u21) (u22) (u23) (u24) (u25) (u26) (ut1)};

	\end{tikzpicture}
	\caption{Situation occurring in Theorem \ref{thm:indice_uno}. We remark that the picture is not referred to the Euclidean space. Indeed, here we denote by $\gamma_j$ the unique collision minimizer which starts from $r_0e^{i\vt_j}$ in the fashion of Lemma \ref{lem:BCTmin} and by  $\gamma_j+2\pi$ and $\gamma_j+4\pi$ its $2\pi$ and $4\pi$ copies respectively. This minimal arcs determine the curved sectors used in the proof, while the juxtaposition of $u_1$ and $u_2$ is a closed path which winds around the origin.}
	\label{fig:cur_sec}
\end{figure}

At this point, we are ready to prove that a minimizer $u\in K_l$ for the Maupertuis' functional joins property $(CF)$.
\begin{theorem}\label{thm:cf}
	Assume that $N\geq 2$, $m\geq 1$ and consider a function $V\in\mathcal{C}^2(\R^2\setminus\{c_1,\ldots,c_N\})$ defined as in \eqref{def:potential} and satisfying \eqref{hyp:V}-\eqref{hyp:V_al}. Fix $\ve\in(0,\tilde\ve)$ as in \eqref{def:ve_tilde} at page \pageref{def:ve_tilde} and consider the potential $\Veps$ defined in \eqref{def:scaling} at page \pageref{def:scaling}. Fix $R\in(\tilde{\ve},\mathfrak{m}^{1/\al}-\tilde{\ve})$ as in \eqref{eq:R} at page \pageref{eq:R} and fix $l\in\mathfrak{I}^N$. Then, there exists $\delta>0$ such that, for every $p_1,p_2\in\partial B_R$, every minimizer $u$ of the Maupertuis' functional
	\[
	\mathcal{M}(u)=\frac12\int_0^1|\dot{u}|^2\int_0^1\left(-1+\Veps(u)\right)
	\]
	in the space $K_l^{p_1,p_2}$ found in Proposition \ref{prop:existence} joins the following properties:
	\begin{itemize}
		\item[$(i)$] $u$ is free of self-intersections;
		\item[$(ii)$] $u$ satisfies
		\[
		\min\limits_{t\in[0,1]}|u(t)-c_j'|>\delta,\quad\mbox{for every}\ j=1,\ldots,N.
		\]
	\end{itemize}
	Therefore, in particular $u$ is collision-less.
\end{theorem}
\begin{proof}
	Fix $\ve$ and $R$ as in the statement, fix $l\in\mathfrak{I}^N$ and $p_1,p_2\in\partial{B}_R$. Assume by contradiction that a minimizer $u\in K_l^{p_1,p_2}$ of the Maupertuis' functional $\mathcal{M}$ has a collision with the centre $c_j'$ for some $j\in\{1,\ldots,N\}$; for the sake of simplicity we will assume again $c_j'=0$. Then, $u\in\mathfrak{Coll}_l^j$, there exists $t_0\in(0,1)$ such that $u(t_0)=0$ and in particular
	\[
	\mbox{Ind}(u;c_k')\equiv l_k\Mod{2},\quad\forall\,k\neq j.
	\]
	Then, localizing the collision as in the beginning if this section, we can find an interval $[c,d]\sset[0,1]$ such that:
	\begin{itemize}
		\item $t_0\in[c,d]$ and the collision is isolated therein;
		\item $\bar{p}_1\uguale u(c),\bar{p}_2\uguale u(d)\in\partial B_r$, with $r<r^*$ and $r^*>0$ as in Lemma \ref{lem:BCTmin}.
	\end{itemize}
	Then, by means of Lemma \ref{lem:homotopy}, the restriction $v\uguale u|_{[c,d]}$ is a minimizer of the Maupertuis' functional
	\[
	\mathcal{M}_l^{\bar{p}_1,\bar{p}_2}(v)=\frac12\int_c^d|\dot{v}|^2\int_c^d\left(-1+\Veps(v)\right)
	\]
	in the weak $H^1$-closure $\mathcal{K}_l^{\bar{p}_1,\bar{p}_2}$ of the $H^1$ restricted paths
	\[
	\hat{\mathcal{K}}_l^{\bar{p}_1,\bar{p}_2}\uguale\left\lbrace v\in H^1([c,d];\R^2)\,\middle\lvert\, \begin{aligned}
		&v(c)=\bar{p}_1,\ v(d)=\bar{p}_2,\ v(t)\neq c_j'\ \forall\ t\in[c,d],\ \forall\ j& \\
		&\mbox{the function}\ G_v(t)\uguale\begin{cases}
			u(t)\quad\mbox{if}\ t\in[0,c)\cup(d,1] \\
			v(t)\quad\mbox{if}\ t\in[c,d]
		\end{cases}& \\
		&\mbox{belongs to}\ \hat{K}_l^{p_1,p_2}&
	\end{aligned}\right\rbrace.
	\]
	Since $v$ solves a Bolza problem for the Maupertuis' functional inside $B_r$, by Theorem \ref{thm:indice_uno}, we know that, up to time reparametrizations, $v$ connects $\bar{p}_1$ and $\bar{p}_2$ belonging to $\Sigma(\vt_j,1)$ or to $\Sigma(\vt_j,2)\setminus\Sigma(\vt_j,1)$, depending on the value of the index $l_j$. Therefore, by claim (i) of Theorem \ref{thm:indice_uno}, a contradiction arises both if $l_j=0$ or $l_j=1$. Thus $u$ cannot have a collisions and in particular, again from Theorem \ref{thm:indice_uno}, (ii) is proved. Claim (i) follows from this property and Proposition \ref{prop:no_self_int}.
\end{proof}

\subsection{Classical solution arcs}

In this section we will conclude the proof of the existence of internal arcs, finally showing that the minimizer of the Maupertuis' functional satisfies property ($R$) introduced at page \pageref{def:cf_r}. In particular, in the next result we show that, given a minimizer $u$ provided in Proposition \ref{prop:existence}, if $u$ has endpoints sufficiently close to minimal non-degenerate central configurations of $\Wzero$, then $|u(t)|<R$ whenever $t\in(0,1)$. Recall that $\Wzero$ is $-\al$-homogeneous (see \eqref{def:wzero} at page \pageref{def:wzero}) and it is the \emph{leading} component of the total potential $\Veps(y)$ when $\ve\to 0^+$ and $|y|$ becomes very large. This suggests to use compactness properties of sequences of minimizers and their convergence to a minimal collision arc for the anisotropic Kepler problem driven by $\Wzero$. We will consider again Lemma \ref{lem:BCTmin} from \cite{BCTmin}, in which the authors show that all the collision minimizers starting sufficiently close to minimal non-degenerate central configurations describe a foliation which is strictly contained in a given cone. Recall that, from the second line of assumption \eqref{hyp:V} at page \pageref{hyp:V}, as already observed in Remark \ref{rem:wzero} at page \pageref{rem:wzero}, the sum potential $\Wzero$ admits a finite number of minimal non-degenerate central configurations
\[
\Xi\uguale\{\vt^*\in\cerchio:\ U'(\vt^*)=0\ \mbox{and}\ U''(\vt^*)>0\}=\{\vt_1^*,\ldots,\vt_m^*\},
\]
where, in polar coordinates $y=(r,\vt)$
\[
\Wzero(r,\vt)=r^{-\al}U(\vt)=r^{-\al}\sum\limits_{i=1}^kU_i(\vt).
\]
Therefore, it is not restrictive to work with two of this central configurations $\vt^*,\vt^{**}\in\cerchio$, since we are solving a Bolza problem, but it is clear that the result holds choosing any pair (not necessarily distinct) of central configurations.

\begin{theorem}\label{thm:R}
	Assume that $N\geq 2$, $m\geq 1$ and consider a function $V\in\mathcal{C}^2(\R^2\setminus\{c_1,\ldots,c_N\})$ defined as in \eqref{def:potential} and satisfying \eqref{hyp:V}. Fix $R>0$ as in \eqref{eq:R} at page \pageref{eq:R}. Then, there exists $\ve_{int}>0$ such that, for any $\vt^*,\vt^{**}\in\Xi$ minimal non degenerate central configurations for $\Wzero$, defining $\xi^*\uguale Re^{i\vt^*},\xi^{**}\uguale Re^{i\vt^{**}}\in\partial B_R$, there exist two neighbourhoods $\mathcal{U}_{\xi^*},\,\mathcal{U}_{\xi^{**}}$ on $\partial B_R$ with the following property:
	\[
	\forall\,\ve\in(0,\ve_{int}),\ \forall l\in\mathfrak{I}^N,\ \forall\,p_1\in\mathcal{U}_{\xi^*},\ \forall\,p_2\in\mathcal{U}_{\xi^{**}}\ \mbox{there holds}\quad |u(t)|<R,\ \mbox{for all}\ t\in(0,1),
	\]
	where $u$ is the minimizer of the Maupertuis' functional in the space $K_l^{p_1,p_2}$ provided in Proposition \ref{prop:existence}.
\end{theorem}
\begin{proof}
	Assume by contradiction that there exist the following sequences:
	\begin{itemize}
		\item $(\ve_n)\sset\R^+$, with $\ve_n\to 0^+$,
		\item $(p_1^n)\sset\mathcal{U}_{\xi^*}$ and $(p_2^n)\sset\mathcal{U}_{\xi^{**}}$ such that, up to subsequences, $p_1^n\to p_1\in\mathcal{U}_{\xi^*}$ and
		$p_2^n\to p_2\in\mathcal{U}_{\xi^{**}}$;
		\item $(t_n)\sset(0,1)$ such that $t_n\to\bar{t}\in[0,1]$;
		\item a sequence of minimizers $(u_n)\sset (K_l^{p_1^n,p_2^n})$ for the sequence of functionals $(\mathcal{M}_n)$ defined by
		\[
		\mathcal{M}_n(u_n)\uguale\frac12\int_0^1|\dot{u}_n|^2\int_0^1\left(-1+V^{\ve_n}(u_n)\right),
		\]	
	\end{itemize}
	such that
	\[
	|u_n(t_n)|=R,\quad\mbox{for all}\ n\in\N.
	\]
	Recalling the limiting behaviour of $\Veps$ as $\ve\to 0^+$ (see Proposition \ref{prop:pert} at page \pageref{prop:pert}), if we define
	\[
	\mathcal{M}_0(u)\uguale\frac12\int_0^1|\dot{u}|^2\int_0^1\left(-1+\Wzero(u)\right),
	\]
	from Lemma \ref{lem:BCTmin} we know that there exists a unique $u^*\in H_{coll}^{p_1}$ and a unique $u^{**}\in H_{coll}^{p_2}$ such that
	\[
	\mathcal{M}_0(u^*)=\min_{H_{coll}^{p_1}}\mathcal{M}_0,\quad\mathcal{M}_0(u^{**})=\min_{H_{coll}^{p_2}}\mathcal{M}_0.
	\]
	In particular, from Proposition \ref{prop:restriction}, we have that there exists a unique $u_0\in H^{p_1,p_2}$, where 
	\[
	H^{p_1,p_2}\uguale\left\lbrace u\in H^1([0,1];\R^2):\,u(0)=p_1,\ u(1)=p_2,\ u(t_0)=0,\ \mbox{for some}\ t_0\in(0,1)\right\rbrace,
	\]
	such that 
	\[
	\mathcal{M}_0(u_0)=\min\limits_{H^{p_1,p_2}}\mathcal{M}_0,
	\]
	where this path $u_0$ is nothing but the juxtaposition of $u^*$ and $u^{**}$.
	We claim that
	\begin{equation}\label{thm:R_claim}
		\lim_{n\to+\infty}\mathcal{M}_n(u_n)=\mathcal{M}_0(u_0)
	\end{equation}
	and we start by showing that
	\begin{equation}\label{thm:R_stima}
		\liminf\limits_{n\to+\infty}\mathcal{M}_n(u_n)\leq\mathcal{M}_0(u_0).
	\end{equation}
	For every $n\in\N$ let us introduce the Jacobi-length functionals
	\[
	\mathcal{L}_n(u_n)\uguale\int_0^1|\dot{u}_n|\sqrt{-1+V^{\ve_n}(u_n)},\quad\mathcal{L}_0(u_0)\uguale\int_0^1|\dot{u}_0|\sqrt{-1+\Wzero(u_0)}
	\]
	and, since $u_n$ and $u_0$ are minimizers, we have
	\[
	\mathcal{L}_n(u_n)=\sqrt{2\mathcal{M}_n(u_n)},\quad\mbox{for all}\ n\in\N,\quad \mathcal{L}_0(u_0)=\sqrt{2\mathcal{M}_0(u_0)}
	\]
	(see Proposition \ref{prop:app1} in Appendix \ref{app:var}). Our idea is to provide an explicit variation $w_n\in K_l^{p_1^n,p_2^n}$ such that, for large $n$
	\begin{equation}\label{thm:R_stima1}
		\mathcal{L}_n(u_n)\leq\mathcal{L}_n(w_n)\leq\mathcal{L}_0(u_0)+\mathcal{O}(\ve_n^\beta),
	\end{equation}
	for some $\beta >0$, which would prove \eqref{thm:R_stima}. In order to build such $w_n$ we need to define some points inside $B_R$:
	\begin{itemize}
		\item define $\hat{p}_1^n\in\ve_n\partial B_R$ as the first intersection between $u_0$ and the circle $\ve_n\partial B_R$;
		\item define $\hat{p}_2^n\in\ve_n\partial B_R$ as the second intersection between $u_0$ and the circle $\ve_n\partial B_R$;
		\item define $q_1^n\uguale\ve_nR\frac{\xi^*}{|\xi^*|}\in\ve_n\partial B_R$ and $q_2^n\uguale\ve_nR\frac{\xi^{**}}{|\xi^{**}|}\in\ve_n\partial B_R$.
	\end{itemize}
	Note that, once $n$ is fixed, the points $\hat{p}_1^n,\hat{p}_2^n$ are uniquely determined since both $u^*$ and $u^{**}$ are strictly decreasing with respect to $t$ thanks to the Lagrange-Jacobi inequality (see \cite[Lemma 4.3]{BCTmin}). Moreover, we define the building blocks of $w_n$ in this way:
	\begin{itemize}
		\item define $\mbox{arc}(p_1^n,p_1)$ as the shorter (in the
		Euclidean metric) parametrized arc of $\partial B_R$, connecting $p_1^n$ to $p_1$ with constant angular velocity;
		\item define $\gamma_n^*$ as the portion of $u_0$ that goes from $p_1$ to $\hat{p}_1^n$;
		\item define $\mbox{arc}(\hat{p}_1^n,q_1^n)$ as as the shorter (in the
		Euclidean metric) parametrized arc of $\partial B_R$, connecting $\hat{p}_1^n$ to $q_1^n$ with constant angular velocity;
		\item define $\vp_n$ as the minimizer of $\mathcal{L}_n$ in the space $K_l^{q_1^n,q_2^n}$;
		\item define as above the analogous path composed by the pieces $\mbox{arc}(q_2^n,\hat{p}_2^n)$, $\gamma_n^{**}$, $\mbox{arc}(p_2,p_2^n)$, which goes from $\hat{p}_2^n$ to $p_2^n$. 
	\end{itemize}
	At this point, we build $w_n$ as the juxtaposition of the previous pieces with a suitable time parametrization
	\[
	w_n=\begin{cases}
		\begin{aligned}
			& \mbox{arc}(p_1^n,p_1) & \mbox{from}\ p_1^n\ \mbox{to}\ p_1 \\
			& \gamma_n^* & \mbox{from}\ p_1\ \mbox{to}\ \hat{p}_1^n \\
			& \mbox{arc}(\hat{p}_1^n,q_1^n) & \mbox{from}\ \hat{p}_1^n\ \mbox{to}\ q_1^n \\
			& \vp_n & \mbox{from}\ q_1^n\ \mbox{to}\ q_2^n \\
			& \mbox{arc}(q_2^n,\hat{p}_2^n) & \mbox{from}\ q_2^n\ \mbox{to}\ \hat{p}_2^n \\
			& \gamma_n^{**} & \mbox{from}\ \hat{p}_2^n\ \mbox{to}\ p_2 \\
			& \mbox{arc}(p_2,p_2^n) & \mbox{from}\ p_2\ \mbox{to}\ p_2^n 
		\end{aligned}
	\end{cases}.
	\]
	(see Figure \ref{fig:R}). Now, since $\mathcal{L}_n$ is additive, the length of $w_n$ is exactly the sum of the length of every piece and, in particular, since $w_n\in K_l^{p_1^n,p_2^n}$ and $u_n$ is a minimizer of $\mathcal{L}_n$ in the same space, we have
	\[
	\mathcal{L}_n(u_n)\leq\mathcal{L}_n(w_n).
	\]	
	The next estimates on the arch lengths easily follow:
	\[
	\begin{aligned}
		& p_1^n\to p_1,\ p_2^n\to p_2 &\implies &\mathcal{L}_n(\mbox{arc}(p_1^n,p_1))=\mathcal{O}(1),\ \mathcal{L}_n(\mbox{arc}(p_2,p_2^n))=\mathcal{O}(1) \\
		& \hat{p}_1^n,\hat{p}_2^n,q_1^n,q_2^n\in\ve_n\partial B_R &\implies & \mathcal{L}_n(\mbox{arc}(\hat{p}_1^n,q_1^n))=\mathcal{O}(\ve_n),\ \mathcal{L}_n(\mbox{arc}(q_2^n,\hat{p}_2^n))=\mathcal{O}(\ve_n)				
	\end{aligned}
	\]	
	as $n\to+\infty$. From Proposition \ref{prop:pert}, we know that, if $y\in\R^2\setminus B_\delta$ with $\delta>\ve_n$, then
	\[
	V^{\ve_n}(y)=\Wzero(y)+\mathcal{O}(\ve_n^{\min\{1,\al_{k+1}-\al\}}),\quad\mbox{as}\ n\to+\infty
	\]
	hence
	\[
	\mathcal{L}_n(\gamma_n^*)+\mathcal{L}_n(\gamma_n^{**})\leq\mathcal{L}_0(u_0)+\mathcal{O}(\ve_n^{\min\{1,\al_{k+1}-\al\}/2}),\quad\mbox{as}\ n\to+\infty.
	\]
	Therefore, to prove the claim \eqref{thm:R_stima1}, we need to provide an estimate on $\mathcal{L}_n(\vp_n)$; to do that, let us define the blow-up sequence 
	\[
	\tilde{\vp}_n(t)\uguale\frac{1}{\ve_n}\vp_n(t),\quad\mbox{for}\ t\in[0,1]
	\]
	and note that
	\begin{equation}\label{thm:R_def}
		\tilde{\vp}_n(0)=\frac{\xi^*}{|\xi^*|}R\uguale q^*\in\partial B_R,\quad\tilde{\vp}_n(1)=\frac{\xi^{**}}{|\xi^{**}|}R\uguale q^{**}\in\partial B_R.
	\end{equation}
	Moreover, recalling the definition \eqref{def:scaling} of $V^{\ve_n}$ at page \pageref{def:scaling}, for $n\in\N$ and $y\in\R^2\setminus\{c_1,\ldots,c_N\}$ we can compute
	\[
		V^{\ve_n}(\ve_ny)=\sum\limits_{i=1}^kV_i(\ve_ny-\ve_nc_i)+\sum\limits_{j=k+1}^N\ve_n^{\al_j-\al}V_j(\ve_ny-\ve_nc_j)=\ve_n^{-\al}V(y)
	\]
	and thus we have
	\[
		\mathcal{L}_n(\vp_n)=\mathcal{L}_n(\ve_n\tilde{\vp}_n)=\ve_n\int_0^1|\dot{\tilde{\vp}}_n|\sqrt{-1+V^{\ve_n}(\ve_n\tilde{\vp}_n)}=\ve_n^{\frac{2-\al}{2}}\int_0^1|\dot{\tilde{\vp}}_n|\sqrt{-\ve_n^\al+V(\tilde{\vp}_n)}=\ve_n^{\frac{2-\al}{2}}\tilde{\mathcal{L}}_n(\tilde{\vp}_n),
	\]
	where we have put
	\[
	\tilde{\mathcal{L}}_n(\tilde{\vp}_n)\uguale\int_0^1|\dot{\tilde{\vp}}_n|\sqrt{-\ve_n^\al+V(\tilde{\vp}_n)},\quad\mbox{for every}\ n\in\N.
	\]
	Notice that, at this point, the function $\tilde{\vp}_n$ is clearly a minimizer for $\tilde{\mathcal{L}}_n$ in the space 
	\[
	S_l^{q^*,q^{**}}\uguale\left\lbrace \vp\in H^1([0,1];\R^2)\,\middle\lvert\,\begin{aligned} &\vp(0)=q^*,\ \vp(1)=q^{**}\ |\vp|\leq R,\\ &\mbox{Ind}(\vp;c_j)\equiv l_j\mod 2,\ \forall\,j=1,\ldots,N\end{aligned}\right\rbrace,
	\]
	where $q^*$ and $q^{**}$ have been defined in \eqref{thm:R_def}. If we furthermore introduce the length functional $\tilde{\mathcal{L}}_0(\vp)\uguale\int_0^1|\dot{\vp}|\sqrt{V(\vp)}$	we have that, for every $n\in\N$ and for every test function $\vp$
	\[
	\tilde{\mathcal{L}}_n(\tilde{\vp}_n)\leq\tilde{\mathcal{L}}_n(\vp)\leq\tilde{\mathcal{L}}_0(\vp)
	\]
	and so, in particular $\tilde{\mathcal{L}}_n(\tilde{\vp}_n)\leq\min_{S_l^{q^{*},q^{**}}}\tilde{\mathcal{L}}_0\leq C$, for some constant $C>0$.	This last inequality finally gives the estimates \eqref{thm:R_stima1} and thus \eqref{thm:R_stima}.
	
	At this point, in order to get the claim \eqref{thm:R_claim}, we prove the reverse inequality, i.e.,
	\begin{equation}\label{thm:R_disug2}
		\mathcal{M}_0(u_0)\leq\liminf\limits_{n\to+\infty}\mathcal{M}_n(u_n).
	\end{equation}
	Since every component of the potential $V^{\ve_n}$ is bounded on $\cerchio$ and $|u_n|\leq R$, together with \eqref{thm:R_stima} we can deduce that there exists $C_1,C_2>0$ such that
	\[
	C_1\geq\mathcal{M}_n(u_n)\geq C_2\int_0^1|\dot{u}_n|^2,
	\]
	at least for $n$ large enough. From this, we deduce a uniform bound on the $H^1$-norm of $(u_n)$ and the existence of a $H^1$-weak and uniform in $[0,1]$ limit $\bar{u}\in H^{p_1,p_2}$. The Fatou's lemma, the semi-continuity of the $H^1$-norm and the a.e. convergence of $V^{\ve_n}$ to $\Wzero$ in $\R^2$ then give
	\[
	\mathcal{M}_0(\bar{u})\leq\liminf_{n\to+\infty}\mathcal{M}_n(u_n).
	\]
	At this point, the minimality of $u_0$ for $\mathcal{M}_0$ in the space $H^{p_1,p_2}$ gives the inequality \eqref{thm:R_disug2} and, together with \eqref{thm:R_stima}, we get the claim \eqref{thm:R_claim}
	\[
	\lim\limits_{n\to+\infty}\mathcal{M}_n(u_n)=\mathcal{M}_0(u_0)
	\]
	with, in particular
	\[
	u_n\to u_0\quad\mbox{uniformly in}\ [0,1].
	\]
	At this point, a bootstrap technique helped by the conservation of the energy for $(u_n)$ leads to a $\mathcal{C}^1$-convergence outside the collision instant $t_0$ of $u_0$; this proves that $|u_0(\bar{t})|=R$. If $\bar t\in(0,1)$ this is a contradiction for Lemma \ref{lem:BCTmin}, because the minimizer $u_0$ cannot leave the cone therein defined; otherwise, if for instance $t_n\to 0^+$, we would find that $\dot{u}_0(0)$ is tangent to $\partial B_R$. This indeed is also a contradiction: up to make $\mathcal{U}_{\xi^*}$ smaller, the unique collision trajectory from $p_1\in\mathcal{U}_{\xi^*}$ must have initial velocity direction close to the initial velocity of the homothetic motion starting from $\xi^*$, which is normal to sphere. 
\end{proof}

\begin{figure}
	\centering
	\includegraphics[scale=1.2]{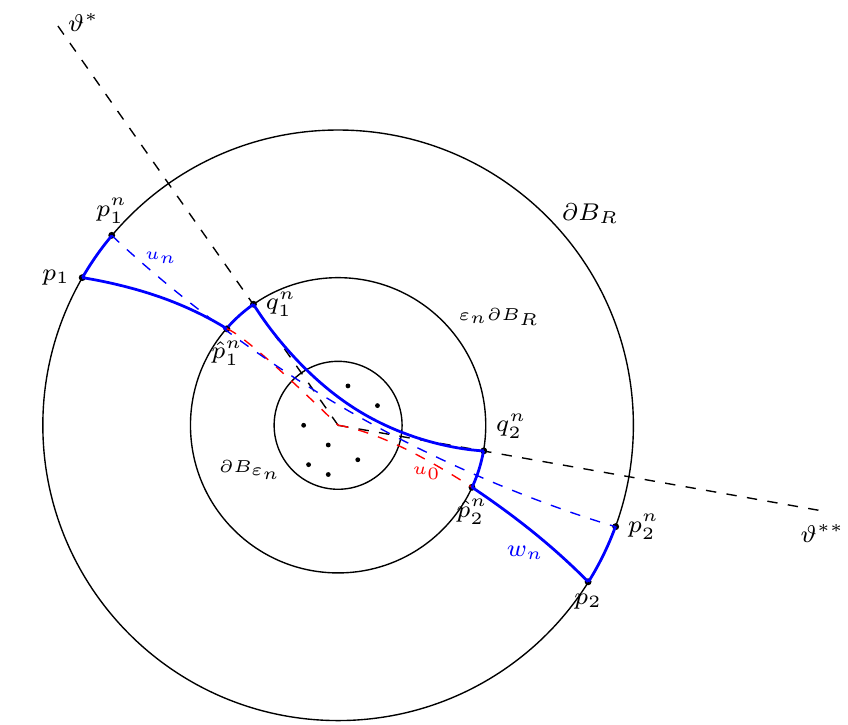}	
	\caption{The proof of Theorem \ref{thm:R}: the blue path $w_n$ piecewise built in the proof belongs to the space $K_l^{p_1^n,p_2^n}$, as well as the blue dashed path $u_n$. This makes it a suitable competitor for $u_n$ and allows to use the minimization argument. In this picture, the red dashed path $u_0$ represents the limit collision path which belongs to the space $H^{p_1,p_2}$ and that actually connects $p_1$ and $p_2$ on $\partial B_R$.}
	\label{fig:R}
\end{figure}

Now, we can finally show that a minimizer of the Maupertuis' functional is actually a reparametrization of a classical solution arc of the inner problem.
\begin{theorem}\label{thm:inner_dyn_wind}
	Assume that $N\geq 2$, $m\geq 1$ and consider a function $V\in\mathcal{C}^2(\R^2\setminus\{c_1,\ldots,c_N\})$ defined as in \eqref{def:potential} and satisfying \eqref{hyp:V}-\eqref{hyp:V_al}. Fix $R>0$ as in \eqref{eq:R} at page \pageref{eq:R}. Then, there exists $\ve_{int}>0$ such that, for any $\vt^*,\vt^{**}\in\Xi$ minimal non degenerate central configurations for $\Wzero$ (they could be equals), defining $\xi^*\uguale Re^{i\vt^*},\xi^{**}\uguale Re^{i\vt^{**}}\in\partial B_R$, there exist two neighbourhoods $\mathcal{U}_{\xi^*},\,\mathcal{U}_{\xi^{**}}$ on $\partial B_R$ with the following property:
	
	for any $\ve\in(0,\ve_{int})$, for any $l\in\mathfrak{I}^N$, for any pair of endpoints $p_1\in\mathcal{U}_{\xi^*}$,$p_2\in\mathcal{U}_{\xi^{**}}$, given the potential $\Veps$ defined in \eqref{def:scaling} at page \pageref{def:scaling}, there exist $T>0$ and a classical (collision-less) solution $y\in \hat{K}_l^{p_1,p_2}([0,T])$ of the inner problem
	\[
	\begin{cases}
		\begin{aligned}
			&\ddot{y}(t)=\nabla \Veps(y(t))  &t\in[0,T] \\
			&\frac{1}{2}|\dot{y}(t)|^2-\Veps(y(t))=-1   &t\in[0,T] \\
			&|y(t)|<R   &t\in(0,T) \\
			&y(0)=p_1,\quad y(T)=p_2. &
		\end{aligned}
	\end{cases}
	\]
	In particular, $y$ is a re-parametrization of a minimizer of the Maupertuis' functional in the space $\hat{K}_l^{p_1,p_2}([0,1])$ and it is free of self-intersections and there exists $\delta>0$
	\[
	\min_{t\in[0,T]}|y(t)-c_j'|>\delta,\quad\mbox{for any}\ j\in\{1,\ldots,N\}.
	\]	
\end{theorem}
\begin{proof}
	The proof is a direct consequence of Theorem \ref{thm:cf} (property ($CF$)), Theorem \ref{thm:R} (property ($R$)) and Theorem \ref{thm:maupertuis} (the Maupertuis' Principle).
\end{proof}

In order to conclude the construction of the interior arcs for the $N$-centre problem, we need to give a version of Theorem \ref{thm:inner_dyn_wind} which takes into account the language of partitions. We invite the reader to go back at page \pageref{intro:partitions} and we note that a minimizer $u\in\hat{K}_l^{p_1,p_2}$ which is free of self intersections satisfies the topological constraint of separation of the centres (see Proposition \ref{prop:no_self_int} and Remark \ref{rem:no_self_int} at page \pageref{prop:no_self_int}). In particular, recalling the definition of the set of all the partitions in two non-trivial subsets of the centres
\[
\mathcal{P}=\{P_j:\ j=0,\ldots,2^{N-1}-2\}
\]
a choice of $l\in\mathfrak{I}^N$ will induce a choice of $P_j\in\mathcal{P}$, for some $j$ and this is not 1-1. 
Notice that the lack of the 1-1 property is due to the fact that, for instance, if $N=3$ the winding vectors $(1,0,0)$ and $(0,1,1)$ produce respectively two minimizers that separate the centres with respect to the same partition (see Figure \ref{fig:partitions}).

\begin{figure}
	\centering
	\begin{tikzpicture}
		\coordinate (0) at (0,0);
		\coordinate (p1) at (-0.26,2.99);   		
		\coordinate (p2) at (-1.02,-2.82);    	\coordinate (c1) at (-0.2,0);
		\coordinate (c2) at (0.8,-0.05);   		
		\coordinate (c3) at (0.5,-0.4);    
		
		\draw (0,0) circle (3cm);
		
		\fill (p1)
		circle[radius=1.0pt];
		\fill (p2)
		circle[radius=1.0pt];
		\fill (c1)
		circle[radius=1.0pt];
		\fill (c2)
		circle[radius=1.0pt];
		\fill (c3)
		circle[radius=1.0pt];

		
		\draw[-stealth] plot [smooth, tension=1] coordinates {(p1) (0.05,0.2)  (p2)};
		\draw[-stealth] plot [smooth,tension=0.5] coordinates {(p1) (-1,-0.5) (0,-0.1) (1,0.3) (1,-0.4) (p2)};
		
		\node[above] at (p1) {$p_1$};
		\node[below] at (p2) {$p_2$};
		\node[left] at (c1) {$c_1$};
		\node[left] at (c2) {$c_2$};
		\node[left] at (c3) {$c_3$};
		\node at (2,2.7) { $\partial B_R$};
		\node[left] at (0.7,1) {$u_1$};
		\node at (1.5,0) {$u_2$};
	\end{tikzpicture}  
	\caption{An example of two minimal arcs which realize the same partition of the centres. Indeed, both the paths divide the centres with respect to the partition $\{\{c_1\},\{c_2,c_3\}\}$, but $u_1\in K_{(0,1,1)}^{p_1,p_2}$ while $u_2\in K_{(1,0,0)}^{p_1,p_2}$.} 
	\label{fig:partitions}		
\end{figure}
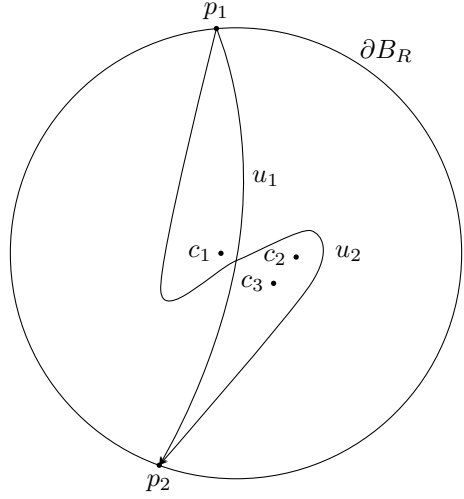

Following the notations introduced in \cite{ST2012}, define the map $\mathcal{A}\colon\mathfrak{I}^N\to\mathcal{P}$ which associates to every winding vector
\[
l=(l_1,\ldots,l_N),\quad\mbox{with the property}\ \begin{cases}
	\mbox{if}\ l_k=0\ \mbox{then}\ k\in A_0\sset\{1,\ldots,N\} \\
	\mbox{if}\ l_j=1\ \mbox{then}\ j\in A_1\sset\{1,\ldots,N\}
\end{cases} 
\] 
the partition
\[
\mathcal{A}(l)=\left\lbrace\{c_k:\,l_k\in A_0\},\{c_j:\,l_j\in A_1\}\right\rbrace.
\]
As already observed, the map $\mathcal{A}$ is surjective, but not injective, since $\mathcal{A}(l)=\mathcal{A}(\tilde{l})$, for every $l,\tilde{l}\in\mathfrak{I}^N$ such that
\[
l_j+\tilde{l}_j=1,\quad\mbox{for every}\ j=1,\ldots,N.
\]
At this point, for any $j\in\{0,\ldots,2^{N-1}-2\}$ and for any $p_1,p_2\in\partial B_R$. this set are well-defined
\begin{equation}\label{def:K_P}
	\begin{aligned}
		\hat{K}_{P_j}&\uguale\hat{K}_{P_j}^{p_1,p_2}([0,1])=\{u\in \hat{K}_l^{p_1,p_2}([0,1]):\,l=\mathcal{A}^{-1}(P_j)\} \\
		K_{P_j}&\uguale K_{P_j}^{p_1,p_2}([0,1])=\{u\in K_l^{p_1,p_2}([0,1]):\,l=\mathcal{A}^{-1}(P_j)\}.
	\end{aligned}
\end{equation}
The set $K_{P_j}$ is the weak $H^1$-closure of $\hat{K}_{P_j}$ and, if $P_j=\mathcal{A}(l)=\mathcal{A}(\tilde{l})$, it turns out that it is exactly the union of two disjoint connected components, i.e., 
\begin{equation}\label{eq:conn_comp}
	K_{P_j}=K_l\cup K_{\tilde{l}}.
\end{equation}
We can now state the main theorem of this section which is readily proven.
\begin{theorem}\label{thm:inner_dyn}
	Assume that $N\geq 2$, $m\geq 1$ and consider a function $V\in\mathcal{C}^2(\R^2\setminus\{c_1,\ldots,c_N\})$ defined as in \eqref{def:potential} and satisfying \eqref{hyp:V}-\eqref{hyp:V_al}. Fix $R>0$ as in \eqref{eq:R} at page \pageref{eq:R}. Then, there exists $\ve_{int}>0$ such that, for any $\vt^*,\vt^{**}\in\cerchio$ minimal non degenerate central configurations for $\Wzero$ (they could be equals), defining $\xi^*\uguale Re^{i\vt^*},\xi^{**}\uguale Re^{i\vt^{**}}\in\partial B_R$, there exist two neighbourhoods $\mathcal{U}_{\xi^*},\,\mathcal{U}_{\xi^{**}}$ on $\partial B_R$ with the following property:
	
	for any $\ve\in(0,\ve_{int})$, for any $P_j\in\mathcal{P}$, for any pair of endpoints $p_1\in\mathcal{U}_{\xi^*}$,$p_2\in\mathcal{U}_{\xi^{**}}$, given the potential $\Veps$ defined in \eqref{def:scaling} at page \pageref{def:scaling}, there exist $T_1,T_2>0$ and two classical (collision-less) solutions $y_1\in \hat{K}_{P_j}^{p_1,p_2}([0,T_1])$ and $y_2\in \hat{K}_{P_j}^{p_1,p_2}([0,T_2])$ of the inner problems
	\[
	\begin{cases}
		\begin{aligned}
			&\ddot{y}(t)=\nabla \Veps(y(t))  &t\in[0,T_i] \\
			&\frac{1}{2}|\dot{y}(t)|^2-\Veps(y(t))=-1   &t\in[0,T_i] \\
			&|y(t)|<R   &t\in(0,T_i) \\
			&y(0)=p_1,\quad y(T)=p_2 &
		\end{aligned}
	\end{cases}
	\]
	for $i=1,2$. In particular, $y_1$ and $y_2$ are re-parametrizations of two minimizers of the Maupertuis' functional, every one of them in a different connected component of $K_{P_j}$ (see \eqref{eq:conn_comp}). Moreover, $y_1$ and $y_2$ are free of self-intersections and there exists $\delta>0$ such that
	\[
	\min_{t\in[0,T_i]}|y_i(t)-c_j'|>\delta,\quad\mbox{for any}\ j\in\{1,\ldots,N\},
	\]
	for $i=1,2$.
\end{theorem}

We conclude this section with a property of the inner solution arcs found in Theorem \ref{thm:inner_dyn}, i.e., we show that there exists a uniform bound on the time intervals of such solutions

\begin{lemma}\label{lem:bound_int}
	Let $\ve\in(0,\ve_{int})$, let $\vt^*,\vt^{**}\in\cerchio$ be two minimal non-degenerate central configurations for $\Wzero$ and $\mathcal{U}^*,\mathcal{U}^{**}$ be their suitable neighbourhoods, let $p_1\in\mathcal{U}^*$ and $p_2\in\mathcal{U}^{**}$, let $P_j\in\mathcal{P}$ and let $y_{P_j}(\cdot;p_1,p_2;\ve)$ be one of the two classical solutions found in Theorem \ref{thm:inner_dyn}, defined in its time interval $[0,T_{P_j}(p_1,p_2;\ve)]$. Then, there exist $c,C>0$ such that
	\[
	c\leq T_{P_j}(p_1,p_2;\ve)\leq C.
	\]
	Such constants do not depend on the choice of $p_1,p_2$ and $P_j$.
\end{lemma}
\begin{proof}	
	The proof is an easy generalization of Lemma 5.2 in \cite{ST2012}.
\end{proof}

\begin{remark}\label{rem:notations_int} In the following, when the partition of the centres will not have a relevance, we will denote one of the internal arcs provided in Theorem \ref{thm:inner_dyn} in this way
	\[
	\intarc{\cdot}{1}{2}{\ve},
	\]
	in order to highlight that it is an arc lying inside $B_R$ that connects $p_1$ and $p_2$. Moreover, the corresponding neighbourhoods on $\partial B_R$ will be denoted as
	\[
	\mathcal{U}_{int}(\xi^*),\ \mathcal{U}_{int}(\xi^{**}).
	\]
\end{remark}

\section{Glueing pieces and multiplicity of periodic solutions}\label{sec:glueing}

Since we have proved the existence of outer and inner fixed-ends solution arcs, respectively in Section \ref{sec:outer} and Section \ref{sec:inner}, this section is devoted to build periodic trajectories which solve 
\[
\begin{cases}
	\ddot{y}=\nabla\Veps(y) \\
	\frac12|\dot{y}|^2-\Veps(y)=-1,
\end{cases}
\]
glueing together solution pieces on $\partial B_R$. We will focus on the non-collisional case, in order to succeed in the proof of Theorem \ref{th:main1} and thus we will require hypotheses \eqref{hyp:V}-\eqref{hyp:V_al} on $\Veps$. An equivalent result for possibly colliding periodic solutions will be given in Subsection \ref{subsec:symb_dyn_coll} (see Theorem \ref{thm:per_coll}), but it is worth noticing that the glueing process is not affected by collisions, since it concerns the behaviour of trajectories far from the singularity set. 

Let $\ve\in(0,\min\{\ve_{int},\ve_{ext}\})$, with $\ve_{int},\ve_{ext}>0$ provided in Theorem \ref{thm:outer_dyn} and Theorem \ref{thm:inner_dyn}, and let $n\in\N_{\geq 1}$ be the number of pairs of inner and outer arcs; the idea is to relate a periodic trajectory in the punctured plane with a double sequence of this kind
\begin{equation}\label{eq:partition}
	(P_0,\xi_0^*),(P_1,\xi_1^*),\ldots,(P_
	{n-1},\xi_{n-1}^*)
\end{equation}
where $P_j\in\mathcal{P}$ is a partition of the centres and $\xi_j^*=Re^{i\vt_j^*}$, with $\vt_j^*\in\Xi$ minimal non-degenerate central configuration for $\Wzero$, for every $j=0,\ldots, n-1$. Note that we admit the case when two or more elements of the sequence could be equals. 

From Theorem \ref{thm:outer_dyn} we know that, for every $j=0,\ldots,n-1$ there exists a neighbourhood $\mathcal{U}_{ext}(\xi_j^*)\sset\partial B_R$ of $\xi_j^*$ such that, for every $(p_{2j},p_{2j+1})\in\mathcal{U}_{ext}(\xi_j^*)\times\mathcal{U}_{ext}(\xi_j^*)$ there exists an outer arc $\extarc{\cdot}{2j}{2j+1}{\ve}$, defined on a suitable time interval $[0,T_{\ve,2j}]$, which starts in $p_{2j}$ and arrives in $p_{2j+1}$. We then select $2n$ points on $\partial B_R$ 
\[
\{p_0,p_1,p_2,\ldots,p_{2n-2},p_{2n-1}\},
\]
so that $p_{2j},p_{2j+1}\in\mathcal{U}_{ext}(\xi_j^*)$ are connected through an outer solution arc, for every $j=0,\ldots,n-1$.

Now, for any $j=1,\ldots,n-1$, thanks to Theorem \ref{thm:inner_dyn} and in view of the notations introduced in Remark \ref{rem:notations_int}, if $p_{2j-1}\in\mathcal{U}_{int}(\xi_{j-1}^*)$ and $p_{2j}\in\mathcal{U}_{int}(\xi_j^*)$, we can connect them through a minimizing inner arc $\intarc{\cdot}{2j-1}{2j}{\ve}$ verifying the partition $P_j$. Up to time re-parametrizations, the inner arc will be defined in the interval $[0,T_{\ve,2j-1}]$.

At this point, in order to build a closed orbit, for any $j=0,\ldots,n-1$ we introduce a smaller neighbourhood of $\xi_j^*$ 
\[
\mathcal{U}_j\uguale\mathcal{U}_{ext}(\xi_j^*)\cap\mathcal{U}_{int}(\xi_j^*)
\]
and we select an ordinate sequence of pairs $(p_{2j},p_{2j+1})\in\mathcal{U}_j\times\mathcal{U}_j$. Moreover, to close the orbit, we add a last inner arc, joining $p_{2n-1}$ and $p_{2n}\uguale p_0$ and that realizes the partition $P_0$. We denote this arc as $\intarc{\cdot}{2n-1}{2n}{\ve}$ and we parametrize it on the interval $[0,T_{\ve,2n-1}]$. 

In addition, it is useful to define
\[
\mathbf{U}\uguale(\mathcal{U}_0\times\mathcal{U}_0)\times(\mathcal{U}_1\times\mathcal{U}_1)\times\ldots\times(\mathcal{U}_{n-1}\times\mathcal{U}_{n-1})\times\mathcal{U}_0\sset(\partial B_R)^{2n+1}
\]
and thus to introduce the following closed set
\[
\mathcal{S}\uguale\left\lbrace\mathbf{p}=(p_0,p_1,\ldots,p_{2n})\in\overline{\mathbf{U}}:\, p_0=p_{2n}\right\rbrace\sset(\partial B_R)^{2n+1},
\]
which describes all the possible cuts on the sphere $\partial B_R$ which an orbit could do, once sequence \eqref{eq:partition} is fixed. In this way, for every $\mathbf{p}\in\mathcal{S}$, we can define the periodic trajectory $\gamma_{\ve,\mathbf{p}}$ as the alternating juxtaposition of $n$ outer arcs and $n$ inner arcs, up to time re-parametrizations. This curve will be $\mathbf{T}_{\ve}$-periodic, where $\mathbf{T}_{\ve}\uguale\sum\limits_{j=0}^{2n-1}T_{\ve,j}$ and piecewise-differentiable thanks to Theorem \ref{thm:outer_dyn} and Theorem \ref{thm:inner_dyn}. In general, the function $\gamma_{\ve,\mathbf{p}}$ is not $\mathcal{C}^1$ in the junction points and, indeed, the main result of this section is to prove this differentiability through a variational technique.

\begin{remark}\label{rem:w_convex}
	We are going to minimize a length functional over $\mathcal{S}$ in order to provide the smoothness of the junctions. Indeed, we have defined $\mathcal{S}$ as a subset of the closure of $\mathbf{U}$ to induce compactness. However, we know from Theorem \ref{thm:outer_dyn} and Theorem \ref{thm:inner_dyn} that the construction of outer and inner arcs works just for the interior points of $\mathcal{U}_j$. For this reason, we might make such neighbourhoods smaller, keeping the same notations. Furthermore, without loss of generality, we assume that the non-degeneracy of every central configuration $\vt_j^*$ of $\Wzero$ is preserved along its corresponding neighbourhood $\mathcal{U}_j$.
\end{remark}

In order to proceed, let us first fix a sequence \eqref{eq:partition} and $\ve\in(0,\min\{\ve_{ext},\ve_{int}\})$. Define the total Jacobi length function $\mathbf{L}\colon\mathcal{S}\to\R$ as
\[
\begin{aligned}
	\mathbf{L}(\mathbf{p})&\uguale\mathcal{L}([0,\mathbf{T}_\ve];\gamma_{\ve,\mathbf{p}}) \\
	&=\sum\limits_{j=0}^{n-1}\mathcal{L}([0,T_{\ve,2j}];\extarc{t}{2j}{2j+1}{\ve})+\sum\limits_{j=1}^{n}\mathcal{L}([0,T_{\ve,2j-1}];\intarc{t}{2j-1}{2j}{\ve}) \\
	&=\sum\limits_{j=0}^{n-1}\int_0^{T_{\ve,2j}}|\dextarc{t}{2j}{2j+1}{\ve}|\sqrt{(\Veps(\extarc{t}{2j}{2j+1}{\ve})-1)}\,dt \\
	&\quad+\sum\limits_{j=1}^{n}\int_0^{T_{\ve,2j-1}}|\dintarc{t}{2j-1}{2j}{\ve}|\sqrt{(\Veps(\intarc{t}{2j-1}{2j}{\ve})-1)}\,dt.
\end{aligned}
\]
The compactness of the set $\mathcal{S}$ implies the following result.
\begin{lemma}\label{lem:glueing_exist}
	There exists $\mathbf{\bar{p}}\in\mathcal{S}$ that minimizes $\mathbf{L}$. 
\end{lemma}
\begin{proof}
	The proof goes exactly as in Step 1) of Theorem 5.3. of \cite{ST2012}.
\end{proof}
The aim of this section is thus to prove through several steps the following result, which then will induce the proof of Theorem \ref{th:main1} (and Theorem \ref{thm:per_coll} for the possibly collisional case).
\begin{theorem}\label{thm:glueing}
	There exists $\bar\ve>0$ such that, for every $\ve\in(0,\bar\ve)$, for every $n\geq 1$ there exists $\mathbf{\bar{p}}=(\bar{p}_0,\bar{p}_1,\ldots,\bar{p}_{2n})\in\mathring{\mathcal{S}}$ such that
	\begin{itemize}
		\item[$(i)$] the following holds
		\[
		\min\limits_{\mathbf{p}\in\mathcal{S}}\mathbf{L}(\mathbf{p})=\mathbf{L}(\mathbf{\bar{p}});
		\]
		\item[$(ii)$] the corresponding function $\gamma_{\ve,\mathbf{\bar{p}}}$ is a periodic solution in $[0,\mathbf{T}_\ve]$ of the $N$-centre problem
		\begin{equation}\label{pb:gamma_N}
			\begin{cases}
				\ddot{\gamma}=\nabla\Veps(\gamma) \\
				\frac12|\dot{\gamma}|^2-\Veps(\gamma)=-1.
			\end{cases}
		\end{equation}
	\end{itemize}	
\end{theorem}

The idea is to provide global smoothness of $\gamma_{\ve,\mathbf{\bar{p}}}$ as a consequence of the Euler-Lagrange equation
\[
\nabla\mathbf{L}(\mathbf{\bar{p}})=0.
\]
In order to compute the partial derivatives of $\mathbf{L}$ we need the uniqueness for each of the $2n$ pieces that compose the juxtaposition $\gamma_{\ve,\mathbf{\bar{p}}}$. The $\mathcal{C}^1$-dependence on initial data guarantees this property for the outer arcs (see Theorem \ref{thm:outer_dyn}). On the contrary, the Maupertuis' principle that we have used so far to find internal solution arcs, does not provide the uniqueness of such paths (see indeed \ref{thm:inner_dyn}). To overcome this, it would be necessary to proceed as in \cite{ST2012,ST2013} and to restrict again the neighbourhoods $\mathcal{U}_j$ in order to work inside a strictly convex neighbourhood. Indeed, it is known that there exists a unique geodesic that connects two points which belong to some neighbourhoods with such property. Since a rigorous treatment in this direction would be very technical and, actually, a repetition of what has been made in the quoted addendum, we will assume that the neighbourhoods $\mathcal{U}_j$ fit this uniqueness property and thus we will compute directly the partial derivatives and we assume the validity of this lemma without further details.
\begin{lemma}\label{lem:length_par_der_exist}
	The function $\mathbf{L}$ admits partial derivatives in $\mathring{\mathcal{S}}$.
\end{lemma}

\subsection{Partial derivatives of the Jacobi length with respect to the endpoints}

In this paragraph we make the explicit computations of the partial derivatives of $\mathbf{L}$. The only non-trivial contributions involved in the computation of the partial derivative of $\mathbf{L}$ with respect to some $p_k$ are given by the length of a selected pair of outer and inner arcs, i.e., the ones that share the contact point $p_k$. Therefore, for the sake of simplicity, in the following proofs we are going to consider only the length of the first outer arc $\extarc{t}{0}{1}{\ve}$ and the last of the inner arcs $\intarc{t}{2n-1}{2n}{\ve}=\intarc{t}{2n-1}{0}{\ve}$, which connect in $p_0$. 

From Theorem \ref{thm:outer_dyn} and following the notations introduced at the beginning of this section, given $\vt_0^*\in\Xi$ and  $\xi_0^*=Re^{i\vt_0^*}\in\partial B_R$, we know that, for every $\ve\in(0,\min\{\ve_{ext},\ve_{int}\})$ and for every $p_0,p_1\in\mathcal{U}_0$, there exists a unique outer arc $y_{ext}(t)=y_{ext}(t;p_0,p_1;\ve)$ which solves the problem
\begin{equation}\label{pb:bvp}
	\begin{cases}
		\begin{aligned}
			&\ddot{y}_{ext}(t)=\nabla \Veps(y_{ext}(t))  &t\in[0,T_{\ve,0}]&\\
			&\dfrac{1}{2}|\dot{y}_{ext}(t)|^2-\Veps(y_{ext}(t))=-1  &t\in[0,T_{\ve,0}]& \\
			&|y_{ext}(t)|>R  &t\in(0,T_{\ve,0})&\\
			&y_{ext}(0)=p_0,\quad y_{ext}(T_{\ve,0})=p_1.	&&
		\end{aligned}
	\end{cases}
\end{equation}
It is important to remark that both $T_{\ve,0}$ and $y_{ext}$, with its first and second derivatives, depend on $p_0$ and $p_1$, while $\ve$ does not depend on $p_0$ and $p_1$. In particular, from the proof of Theorem \ref{thm:outer_dyn}, we have that $T_{\ve,0}=T_{\ve,0}(\ve,p_0,p_1)\uguale T(p_0,\eta(\ve,p_0,p_1))$, where $\eta$ is the implicit function defined by the shooting map. 

Consider the length of the external arc $\mathcal{L}_{ext}\colon\mathcal{U}_0\times\mathcal{U}_0\to\R_0^+$ such that
\[
\mathcal{L}_{ext}(p_0,p_1)\uguale\int_0^{T_{\ve,0}}|\dot{y}_{ext}(t)|\sqrt{(\Veps(y_{ext}(t))-1)}\,dt,
\]
which, using the conservation of energy, can be written in the following two equivalent forms
\begin{equation}\label{defn:length}
	\mathcal{L}_{ext}(p_0,p_1)=\frac{1}{\sqrt{2}}\int_0^{T_{\ve,0}}|\dot{y}_{ext}(t)|^2\,dt=\frac{1}{\sqrt{2}}\int_0^{T_{\ve,0}}\left(\frac12|\dot{y}_{ext}(t)|^2+\Veps(y_{ext}(t))-1\right)\,dt.
\end{equation}

\begin{lemma}\label{lem:diff_L_ext}
	The function $\mathcal{L}_{ext}\in\mathcal{C}^1(\mathcal{U}_0\times\mathcal{U}_0)$ and its differential, for every $(p_0,p_1)\in\mathcal{U}_0\times\mathcal{U}_0$, is 
	\[d\mathcal{L}_{ext}(p_0,p_1)\colon\mathcal{T}_{p_0}(\partial B_R)\times\mathcal{T}_{p_1}(\partial B_R)\to\R
	\]
	such that
	\[
		d\mathcal{L}_{ext}(p_0,p_1)[\varphi,\psi]=-\frac{1}{\sqrt{2}}\langle\dot{y}_{ext}(0),\vp\rangle+\frac{1}{\sqrt{2}}\langle\dot{y}_{ext}(T_{\ve,0}),\psi\rangle.
	\]
\end{lemma}
\begin{proof}
	Thanks to the $\mathcal{C}^1$-dependence of problem \eqref{pb:bvp} on the initial data and time, the function $\mathcal{L}_{ext}$ is of class $\mathcal{C}^1$ in $\mathcal{U}_0\times\mathcal{U}_0$. Moreover, from \eqref{defn:length} and integrating by parts we obtain
	\begin{equation}\label{eq:length_der}	
			\frac{\partial}{\partial p_0}\mathcal{L}_{ext}(p_0,p_1)=\frac{1}{\sqrt{2}}\left(|\dot{y}_{ext}(T_{\ve,0})|^2\frac{\partial T_{\ve,0}}{\partial p_0}+\left[\dot{y}_{ext}(t)\frac{\partial y_{ext}(t)}{\partial p_0}\right]_{0}^{T_{\ve,0}}\right).
	\end{equation}
	Note that the term $\dot{y}_{ext}(t)\frac{\partial\dot{y}_{ext}(t)}{\partial p_0}$ is actually a $(1\times 2)\cdot(2\times 2)$ matrix product, so that we would have to transpose the vector $\dot{y}_{ext}(t)$ firstly. Anyway, we are going to omit this and other transpositions in order to ease the notations. 
	
	Now, we can also compute the total derivative of the boundary conditions in \eqref{pb:bvp} with respect to $p_0$, obtaining
	\begin{equation}\label{eq:in_cond}
		\frac{d}{dp_0}y_{ext}(0)=\frac{\partial y_{ext}(0)}{\partial p_0}=I_2
	\end{equation}
	and
	\begin{equation}\label{eq:fin_cond}
		\frac{d}{dp_0}y_{ext}(T_{\ve,0})=\dot{y}_{ext}(T_{\ve,0})\frac{\partial T_{\ve,0}}{\partial p_0}+\frac{\partial y_{ext}(T_{\ve,0})}{\partial p_0}=0_2.   
	\end{equation}
	Moreover, multiplying both sides of \eqref{eq:fin_cond} by $\dot{y}_{ext}(T_{\ve,0})$, we have
	\[
	|\dot{y}_{ext}(T_{\ve,0})|^2\frac{\partial T_{\ve,0}}{\partial p_0}+\dot{y}_{ext}(T_{\ve,0})\frac{\partial y_{ext}(T_{\ve,0})}{\partial p_0}=0.
	\]
	This, together with \eqref{eq:length_der} and \eqref{eq:in_cond}, leads to
	\[
	\frac{\partial}{\partial p_0}\mathcal{L}_{ext}(p_0,p_1)=-\frac{1}{\sqrt{2}}\dot{y}_{ext}(0)=-\frac{1}{\sqrt{2}}v_0(\ve,p_0,p_1).
	\]
	With the same computations, one obtains	
	\[
	\frac{\partial}{\partial p_1}\mathcal{L}_{ext}(p_0,p_1)=\frac{1}{\sqrt{2}}\dot{y}_{ext}(T_{\ve,0})=\frac{1}{\sqrt{2}}v_1(\ve,p_0,p_1). \qedhere
	\]
\end{proof}

As before, for $\ve\in(0,\min\{\ve_{ext},\ve_{int}\})$, we can consider the length of the inner arc $y_{int}(t)\uguale y_{int}(t;p_{2n-1},p_0;\ve)$ as the function $\mathcal{L}_{int}\colon\mathcal{U}_{2n-1}\times\mathcal{U}_0\to\R_0^+$ such that
\[
\mathcal{L}_{int}(p_{2n-1},p_0)\uguale\int_0^{T_{\ve,2n-1}}|\dot{y}_{int}(t)|\sqrt{(\Veps(y_{int}(t))-1)}\,dt
\]
and prove the following lemma.
\begin{lemma}\label{lem:diff_L_int}
	The function $\mathcal{L}_{int}\in\mathcal{C}^1(\mathcal{U}_{2n-1}\times\mathcal{U}_0)$ and its differential, for every $(p_{2n-1},p_0)\in\mathcal{U}_{2n-1}\times\mathcal{U}_0$, is \[
	d\mathcal{L}_{int}(p_{2n-1},p_0)\colon\mathcal{T}_{p_{2n-1}}(\partial B_R)\times\mathcal{T}_{p_0}(\partial B_R)\to\R
	\]
	such that
	\[
	d\mathcal{L}_{int}(p_{2n-1},p_0)[\nu,\varphi]=-\frac{1}{\sqrt{2}}\langle \dot{y}_{int}(0),\nu\rangle+\frac{1}{\sqrt{2}}\langle \dot{y}_{int}(T_{\ve,2n-1}),\varphi\rangle.
	\]
\end{lemma}
\begin{proof}
	As we have already remarked at page \pageref{lem:length_par_der_exist}, the differentiability of this length function is a consequence of the results contained in \cite{ST2013}, up to restrict the neighbourhoods $\mathcal{U}_{2n-1}$ and $\mathcal{U}_0$. Concerning the computation of the differential, the proof goes exactly as in Lemma \ref{lem:diff_L_ext}.
\end{proof}

\subsection{The minimizing points of the Jacobi length are not in the boundary}

The purpose of this section is to prove the first statement of Theorem \ref{thm:glueing}. From Lemma \ref{lem:glueing_exist} it sufficient to show that the minimizer $\mathbf{\bar{p }}$ does not occur on the boundary of $\mathcal{S}$. In the next two rather technical paragraphs we will study the local behaviour of the external and internal arcs with respect to small variations on the endpoints. In short, what happens is that, if $\mathbf{\bar{p}}\in\partial\mathcal{S}$, then a particular variation on the endpoints gives a contradiction against the minimality of $\mathbf{\bar{p}}$.

\subsubsection{An explicit variation on the external path}

In order to do this, we first associate to the outer problem \eqref{pb:bvp} its flow $\Phi^t(p_0,v_0)$ which actually depends on $\ve$ too; we omit this dependence to ease the notations. Moreover, keeping in mind the notations of Theorem \ref{thm:outer_dyn}, we have that
\[
y_{ext}(T_{\ve,0};p_0,p_1;\ve)=\pi_x\Phi^{T_{\ve,0}}(p_0,v_0),
\]
where $v_0=v_0(\ve,p_0,p_1)=\dot{y}_{ext}(0;p_0,p_1;\ve)$ and 
\begin{equation}\label{eq:b_cond}
	\Phi^0(p_0,v_0)=(p_0,v_0),\quad	\Phi^{T_{\ve,0}}(p_0,v_0)=(p_1,v_1),
\end{equation}
with $v_1=v_1(\ve,p_0,p_1)$. Finally, we observe that the outer problem \eqref{pb:bvp} is time-reversible, since it is not difficult to prove that
\begin{equation}\label{eq:time_rev}
		y_{ext}(t;p_0,p_1;\ve)=y_{ext}(T_{\ve,0}-t;p_1,p_0;\ve)\quad\mbox{for every}\ t\in[0,T_{\ve,0}].
\end{equation}
In order to see how the solution $y_{ext}$ behaves under small variations on $p_1$ let us start by defining the matrix function $M$ as the derivative of $y_{ext}$ with respect to the arriving point (cf. Remark \ref{rem:var_eq} at page \pageref{rem:var_eq}), i.e.,
\[
M(t)\uguale\frac{\partial}{\partial p_1}y_{ext}(t;p_0,p_1;\ve)=\pi_x\frac{\partial}{\partial p_1}\Phi^t(p_0,v_0),
\]
which, from \eqref{pb:bvp}, satisfies 
\[
\ddot{M}(t)=\nabla^2\Veps(y_{ext}(t))M(t)\quad\mbox{for every}\ t\in[0,T_{\ve,0}].
\]
Moreover, from \eqref{eq:b_cond}, we have that
\[
	M(0)=\pi_x\frac{\partial}{\partial p_1}\Phi^0(p_0,v_0)=0_2,\quad
	M(T_{\ve,0})=\pi_x\frac{\partial}{\partial p_1}\Phi^{T_{\ve,0}}(p_0,v_0)=I_2,
\]
so that $M(t)$ is a solution of the boundary value problem
\begin{equation}\label{pb:lin_W}
	\begin{cases}
		\ddot{M}(t)=\nabla^2\Veps(y_{ext}(t;p_0,p_1;\ve))M(t),\quad t\in[0,T_{\ve,0}] \\
		M(0)=0_2,\ M(T_{\ve,0})=I_2.
	\end{cases}
\end{equation}

\begin{lemma}\label{lem:jac_v}
	Let $M=M(t)$ be a solution of \eqref{pb:lin_W}. Then,
	\[
	J_{p_0,p_1}v_0(\ve,p_0,p_1)=\left(-\dot{M}(T_{\ve,0}),\dot{M}(0)\right).
	\]
\end{lemma}
\begin{proof}
	Recalling that $v_0(\ve,p_0,p_1)=\dot{y}_{ext}(0;p_0,p_1;\ve)$, since problem \eqref{pb:bvp} is time-reversible (see \eqref{eq:time_rev}), we have that
	\[    
	\begin{aligned}
		\frac{\partial}{\partial p_0}v_0(\ve,p_0,p_1)&=\frac{\partial}{\partial p_0}\left(\frac{d }{d t}y_{ext}(t;p_0,p_1;\ve)\Big\rvert_{t=0}\right)=\frac{d}{d t}\left(\frac{\partial }{\partial p_0}y_{ext}(t;p_0,p_1;\ve)\right)_{t=0}
		\\
		&=\frac{d}{dt}\left(\frac{\partial}{\partial p_1}y_{ext}(T_{\ve,0}-t;p_1,p_0;\ve)\right)_{t=0}=\frac{d}{dt}M(T_{\ve,0}-t)\Big\rvert_{t=0}
		=-\dot{M}(T_{\ve,0}). 
	\end{aligned}
	\]
	\[
	\frac{\partial}{\partial p_1}v_0(\ve,p_0,p_1)=\frac{\partial}{\partial p_1}\left(\frac{d }{d t}y_{ext}(t;p_0,p_1;\ve)\Big\rvert_{t=0}\right)=\frac{d}{d t}\left(\frac{\partial}{\partial p_1} y_{ext}(t;p_0,p_1;\ve)\right)_{t=0} 
	=\dot{M}(0). 
	\qedhere
	\]
\end{proof}

\label{discussion}Let us focus on some details for a moment. The function $M(t)$ is actually the solution of the variational equation (see Remark \ref{rem:var_eq} at page \pageref{rem:var_eq}) around an external arc $y_{ext}$, which gives information on how the flow associated to such solution changes under infinitesimal variations on the boundary conditions. Moreover, we know that $y_{ext}$ (and thus, $M$) depends on $\ve$ in a $\mathcal{C}^1$ manner, since we have shown that the anisotropic $N$-centre problem is a perturbation of a Kepler problem driven by $\Wzero$ (see Proposition \ref{prop:pert} in Section \ref{sec:scaling}). Therefore, it makes sense to simplify the proof and to consider a particular problem, i.e., the one around the homothetic solution emanating from some $\xi^*=Re^{i\vt^*}$ and to put $\ve=0$. The characterization of the Hessian matrix of $\Wzero$ given in Remark \ref{rem:hess} at page \pageref{rem:hess} will be very useful in this context. Finally, the non-degeneracy of $\Wzero$, which is fundamental in the forthcoming proof, is guaranteed also in a neighbourhood of $\xi^*$, so that the argument can be easily extended near $\xi^*$ (see Remark \ref{rem:w_convex}). Therefore, let us consider the homothetic trajectory $\homy{t}{\xi^*}=y_{ext}(t;\xi^*,\xi^*;0)$. Around this solution, problem \eqref{pb:lin_W} becomes
\begin{equation}\label{pb:lin_hom}
	\begin{cases}
		\ddot{M}(t)=\nabla^2\Wzero(\homy{t}{\xi^*})M(t),\quad t\in[0,T] \\
		M(0)=0_2,\ M(T)=I_2,
	\end{cases}
\end{equation}
where clearly $T>0$ is the first return time of the homothetic motion. Following the approach of Section \ref{sec:outer}, we study the projection of \eqref{pb:lin_hom} on the tangent space, thus reducing to a 1-dimensional setting. 

\begin{lemma}\label{lem:Wwv}
	Let $M=M(t)$ be a solution of \eqref{pb:lin_hom} and let us define $s_\xi\uguale\xi^*/|\xi^*|$ and its unitary orthogonal vector $s_\tau=s_\xi^\perp$. Moreover, define the 1-dimensional functions
	\[
	\begin{aligned}
		w(t)&\uguale\langle M(t)s_\tau,s_\tau\rangle,\quad v(t)\uguale\langle M(t)s_\tau,s_\xi\rangle \\
		c(t)&\uguale-\langle\nabla^2\Wzero(\homy{t}{\xi^*})s_\tau,s_\tau\rangle=|\homy{t}{\xi^*}|^{-\al-2}\left(\al U(\vt^*)-U''(\vt^*)\right),
	\end{aligned}
	\]
	where the last equality has been proven in \eqref{eq:tang_hessian_tau} at page \pageref{rem:hess}. Then $w$ solves
	\begin{equation}\label{pb:1dim}
		\begin{cases}
			\ddot{w}+c(t)w=0 \\
			w(0)=0,\ w(T)=1
		\end{cases}
	\end{equation}
	and 
	\[
	v\equiv 0\ \mbox{in}\ [0,T].
	\]
\end{lemma}
\begin{proof}
	From the definition of $w$ and since $M(t)$ solves \eqref{pb:lin_hom} it is clear that 
	\[
	\ddot{w}(t)-\langle\nabla^2\Wzero(\homy{t}{\xi^*})M(t)s_\tau,s_\tau\rangle=0\quad\mbox{for every}\ t\in[0,T].
	\]
	Now, since $\nabla^2\Wzero(\homy{t}{\xi^*})$ is symmetric and admits $s_\tau$ as eigenvector with eigenvalue $\la_\tau(t)=\langle\nabla^2\Wzero(\homy{t}{\xi^*})s_\tau,s_\tau\rangle$ (see \eqref{eq:tang_hessian} at page \pageref{rem:hess}), we can write
	\[
		\langle\nabla^2\Wzero(\homy{t}{\xi^*})M(t)s_\tau,s_\tau\rangle=\langle M(t)s_\tau,\nabla^2\Wzero(\homy{t}{\xi^*})s_\tau\rangle=\langle\nabla^2\Wzero(\homy{t}{\xi^*})s_\tau,s_\tau\rangle w(t).
	\]
	Hence, considering the boundary conditions satisfied by $M(t)$ in \eqref{pb:lin_hom} we conclude that $w$ solves \eqref{pb:1dim}. 
	
	Concerning $v$, using the same argument, we deduce that it solves
	\[
	\begin{cases}
		\ddot{v}=\la_\xi(t)v \\
		v(0)=0=v(T),
	\end{cases}
	\]
	with $\la_\xi(t) = |\homy{t}{\xi^*}|^{-\al-2}\al(\al+1)U(\vt^*)$ (see \eqref{eq:tang_hessian_xi} at page \pageref{eq:tang_hessian_xi}). Since $\la_\xi(t)>0$ then $v\equiv 0$ in $[0,T]$.
\end{proof}
As a consequence of Lemma \ref{lem:jac_v} and Lemma \ref{lem:Wwv}, we can directly deduce this useful result.
\begin{corollary}\label{coro:Wwv}
	Let $M$, $w$ and $v$ as in Lemma \ref{lem:Wwv} and recall the notations of Lemma \ref{lem:jac_v}. Then
	\[
	J_{p_0,p_1}v_0(0,\xi^*,\xi^*)=\left(-\dot{M}(T),\dot{M}(0)\right)
	\]
	and, for every $k\in[-1,1]$, 
	\begin{equation*}
		\left\langle J_{p_0,p_1}v_0(0,\xi^*,\xi^*)\begin{bmatrix}
			s_\tau \\
			ks_\tau
		\end{bmatrix},s_\tau\right\rangle = -\langle\dot{M}(T)s_\tau,s_\tau\rangle+k\langle\dot{M}(0)s_\tau,s_\tau\rangle=-\dot{w}(T)+k\dot{w}(0),
	\end{equation*}
	and 
	\begin{equation}\label{eq:jac_v_xi}
		\left\langle J_{p_0,p_1}v_0(0,\xi^*,\xi^*)\begin{bmatrix}
			s_\tau \\
			ks_\tau
		\end{bmatrix},s_\xi\right\rangle = -\dot{v}(T) + k\dot{v}(0) = 0.
	\end{equation}
\end{corollary}
At this point, our aim is to prove a result which actually gives an estimate of the tangential component of the velocity $v_0=v_0(0,p_0,p_1)$, with respect to oscillations of $p_0$ and $p_1$ around the central configuration $\xi^*$. In order to do this, it is useful to introduce polar coordinates to provide an explicit dependence of $p_0$ and $p_1$ on an angular variation. Indeed, we characterize every point $p\in\mathcal{U}_0=\mathcal{U}_0(\xi^*)$ as a function of the counter-clockwise angle $\phi\in(-\pi,\pi)$ joining $p$ and $\xi^*$, so that
\begin{equation}\label{eq:p_phi}
	p(0) = \xi^* = Rs_{\xi},
	\quad \text{and} \quad
	p(\phi) = R\cos\phi \, s_\xi+R\sin\phi \, s_\tau,
\end{equation}
(see Figure \ref{fig:polar}). We furthermore remark that when we write the orthogonal of a vector we mean a counter-clockwise rotation of $\pi/2$ of such vector. 
\begin{figure}
	\centering
	\begin{tikzpicture}[scale=0.9]
		\coordinate (O) at (0,0);
		\coordinate (x) at (6,0);
		\coordinate (p0) at (5.48,2.45);
		\coordinate (p0m) at (5.48,-2.45);
		\coordinate (p1) at (5.91,-1.04);
		
		
		\draw [dashed] (-2,0)--(8,0);
		\draw [dashed] (4.5,4.64)--(6.1,1.06);
		\draw [dashed,red] (6.7,1.9)--(5.9,1.51);
		\draw  (O)--(x);
		\draw  (O)--(p0);
		\draw  (O)--(p1);
		\draw  (O)--(p0m);
		
		
		\draw [thick,-stealth,blue] (6,0)--(7,0);
		\draw [thick,-stealth,blue] (6,0)--(6,1);
		\draw [thick,-stealth,red] (p0)--(6.7,1.9);
		
		
		\draw [thick,domain=-25:25] plot ({6*cos(\x)}, {6*sin(\x)});
		\draw [dashed,domain=-40:40] plot ({6*cos(\x)}, {6*sin(\x)});
		\draw [-stealth](3,0) arc (0:24:3);
		\draw [-stealth](4,0) arc (0:-13:3);
		\draw [-stealth](3,0) arc (0:-24:3); 
		
		\fill (O) circle[radius=1.5pt];
		\fill (x) circle[radius=1.5pt];
		\fill (p0) circle[radius=1.5pt];
		\fill (p1) circle[radius=1.5pt];
		
		\node[right] at (6,-0.3) {$\xi^*$};
		\node[left] at (0,-0.3) {$0$};
		\node[right] at (5.4,2.8) {$p_0(\phi_0)$};
		\node[blue,left] at (7,0.3) {$s_{\xi}$};
		\node[blue,right] at (6,0.7) {$s_{\tau}$};
		\node[right] at (2.9,0.6) {$\phi_0$};
		\node[right] at (2.8,-0.8) {$-\phi_0$};
		\node[right] at (4,-0.3) {$\phi_1$};
		\node[right] at (6,-1.1) {$p_1(\phi_1)$};
		\node[right,red] at (6.7,2) {$v_0(p_0,p_1)$};
		\node at (4.5,5) {$p_0(\phi_0)^{\perp}$};	
		\node[right] at (5.2,-3) {$\partial B_R$};
		
	\end{tikzpicture}
	\caption{The notations of Lemma \ref{lem:var_ext} and the behaviour of the initial velocity with respect to variations on the endpoints $p_0$ and $p_1$.}
	\label{fig:polar}
\end{figure}
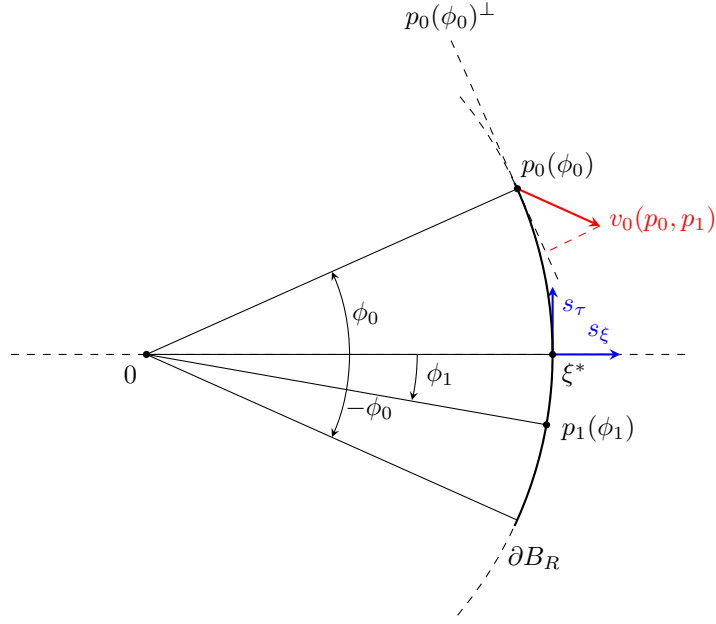

\begin{lemma}\label{lem:var_ext}
	There exists $\delta=\delta(\vt^*),\,C=C(\vt^*)>0$ such that, for any $\phi_0,\phi_1\in\R$ verifying $0<|\phi_0|<\delta$ and  $|\phi_1|\leq|\phi_0|$, the following holds
	\[
	\frac{-\langle v_0(0,p_0(\phi_0),p_1(\phi_1)),p_0(\phi_0)^\perp\rangle}{\phi_0}\geq C,
	\]
	(see Figure \ref{fig:polar}).
\end{lemma}

\begin{proof}
	For the sake of simplicity we introduce these notations
	\[
		v_0(p_0,p_1)\uguale v_0(0,p_0,p_1),\  v_{\xi^*}\uguale v_0(p_0(0),p_1(0)).
	\]
	Furthermore, we prove the statement for $\phi_0>0$; if $\phi_0$ is negative, the proof is the same up to minor changes. Since $v_{\xi^*}$ is the initial velocity of the homothetic motion along a central configuration, it is orthogonal to $s_\tau$, then $\langle v_{\xi^*},p_0(0)^\perp \rangle=0$ and so we can write
	\[
	\langle v_0(p_0(\phi_0),p_1(\phi_1)),p_0(\phi_0)^\perp\rangle=\langle v_0(p_0(\phi_0),p_1(\phi_1))-v_{\xi^*},p_0(\phi_0)^\perp\rangle +\langle v_{\xi^*},p_0(\phi_0)^\perp-p_0(0)^\perp\rangle.
	\]
	From the notations in \eqref{eq:p_phi}, if $\phi \to 0^+$, than we have
	\[
	p(\phi)-p(0) = R\phi s_\tau +o(\phi)
	\quad \text{and} \quad 
	p(\phi)^\perp - p(0)^\perp = -R \phi s_{\xi} +o(\phi),
	\]
	hence, as $\phi_0\to 0^+$,
	\[
	\langle v_{\xi^*},p_0(\phi_0)^\perp-p_0(0)^\perp\rangle = -R \lvert v_{\xi^*}\rvert \phi_0 +o(\phi_0).
	\]
	Furthermore, again as $\phi_0 \to 0^+$ (and thus, as $\phi_1\to 0^+$)
	\[
	\begin{aligned}
		v_0(p_0(\phi_0),p_1(\phi_1))-v_{\xi^*} 
		&= J _{p_0,p_1} v_0(\xi^*,\xi^*)\begin{bmatrix}
			p_0(\phi_0)-p_0(0)\\
			p_1(\phi_1)-p_1(0)
		\end{bmatrix}  + o(\phi_0) \\
		&= R\,J _{p_0,p_1} v_0(\xi^*,\xi^*)\begin{bmatrix}
			\phi_0 s_\tau\\
			\phi_1 s_\tau		\end{bmatrix}  + o(\phi_0).
	\end{aligned}
	\]
	Assuming now $\phi_1 = k\phi_0$ for some $k\in[-1,1]$, and using \eqref{eq:jac_v_xi} in Corollary \ref{coro:Wwv} and the fact that $p_0(0)^\perp= Rs_\tau$, we have
	\[
	\left\langle v_0(p_0(\phi_0),p_1(\phi_1))-v_{\xi^*},p_0(\phi_0)^\perp\right\rangle = 
	R^2\phi_0\left\langle J_{p_0,p_1}v_0(\xi^*,\xi^*)\begin{bmatrix}
		s_\tau \\
		ks_\tau   	
	\end{bmatrix},s_\tau\right\rangle + o(\phi_0).
	\]
	so that
	\[
	\langle v_0(p_0(\phi_0),p_1(\phi_1)),p_0(\phi_0)^\perp\rangle
	= 
	R\phi_0\left(R\left\langle J_{p_0,p_1}v_0(\xi^*,\xi^*)\begin{bmatrix}
		s_\tau \\
		ks_\tau   	
	\end{bmatrix},s_\tau\right\rangle - \lvert v_{\xi_0^*}\rvert\right)  + o(\phi_0).
	\]
	In order to conclude we need to prove the existence of a positive constant $C$, depending on $\phi_0$ and $k$, such that 
	\[
	-\left(R\left\langle J_{p_0,p_1}v_0(\xi^*,\xi^*)\begin{bmatrix}
		s_\tau \\
		ks_\tau   	
	\end{bmatrix},s_\tau\right\rangle - \lvert v_{\xi^*}\rvert\right) \geq C >0.
	\]
	By means of Corollary \ref{coro:Wwv}, this is equivalent to prove that
	\begin{equation}\label{eq:w}
	R\dot{w}(T) - kR\dot{w}(0)
	+\lvert v_{\xi^*}\rvert \geq C >0,
	\end{equation}
	where $w(t)$ solves \eqref{pb:1dim}.
	
	Let now $u(t) = |\homy{t}{\xi^*}|$; then $u$ solves the 1-dimensional problem
	\begin{equation}\label{pb:hom}
		\begin{cases}
			\displaystyle
			\ddot{u} + \frac{\al U(\vt^*)}{u^{\al+2}}u=0 \\
			u(0) = R = u(T).
		\end{cases}
	\end{equation}
	Since by assumption $U''(\vt^*)> 0$ we have
	\begin{equation}\label{d_c}
		d(t)\uguale\frac{\al U(\vt^*)}{u(t)^{\al+2}}> \frac{\al U(\vt^*)-U''(\vt^*)}{u(t)^{\al+2}}\uguale c(t)\quad\mbox{in}\ [0,T],
	\end{equation}
	recalling that the function $c(t)$ has been introduced in Lemma \ref{lem:Wwv}.\\
	Since $u(t) \neq 0$ for any $t \in [0,T]$, we can define $f(t)=\frac{w(t)}{u(t)}$ and compute
	\[
		\ddot f =-2\frac{\dot u}{u}\left(\frac{\dot w}{u}-\frac{w\dot u}{u^2}\right)+\frac{w}{u}\left(\frac{\ddot w}{w}-\frac{\ddot u}{u}\right).
	\]	
	Multiplying both sides by $u^2$ we deduce that $f$ solves the problem
	\begin{equation}\label{pb:omega}
		\begin{cases}
			\displaystyle
			\frac{d}{dt}\left(u^2\dot{f}\right)= (d(t)-c(t)) u^2f \\
			f(0)=0,\; f(T)=1/Rf.
		\end{cases}
	\end{equation}
	We want to prove that $f$ is strictly positive in $(0,T]$; hence, suppose by contradiction that there exists $t^*\in(0,T)$ such that $f(t^*)=0$. Then, it is clear that there exists $t_m\in(0,t^*]$ such that
	\[
	f(t_m)\leq 0,\ \dot{f}(t_m)=0,\ \ddot{f}(t_m)>0.
	\]
	In this way, considering the equation in \eqref{pb:omega} and the inequality \eqref{d_c}, we get
	\[
	0<u^2(t_m)\ddot{f}(t_m)=(d(t_m)-c(t_m))u^2(t_m)f(t_m)\leq 0,
	\]
	which is a contradiction. Therefore, $f$ is strictly positive in the interval $(0,T]$.
	
	Now, integrating the equation in \eqref{pb:omega} in $[0,T]$ we get
	\[
	u^2(T)\dot{f}(T)-u^2(0)\dot{f}(0)>0
	\]
	and, using the explicit expression of $\dot{f}$, we have
	\begin{equation}\label{eq:gamma_1}
		R\dot{w}(T)-R\dot{w}(0)>\dot{u}(T)=\frac{1}{R}\langle\homy{T}{\xi^*},\dot{\hat{y}}_{\xi^*}(T)\rangle=-|v_{\xi^*}|.
	\end{equation}
	Moreover, since $f$ can not vanish in $(0,T]$, we have that necessarily $\dot{f}(0)\geq 0$ and thus 
	\begin{equation}\label{eq:w_geq_0}
		\dot{w}(0)\geq\frac{w(0)\dot{u}(0)}{u^2(0)}=0.
	\end{equation}
	
	%
	%
	At this point, let us consider  $\gamma\in(0,1)$ and define the function
	\[
	\begin{aligned}
		u_\gamma\colon[0,&T]\to\R^+ \\
		&t\longmapsto u_\gamma(t)\uguale u(t)^\gamma,
	\end{aligned}
	\]
	which verifies  
	\[
	\ddot{u}_\gamma=\gamma(\gamma-1)u^{\gamma-2}\dot{u}^2+\gamma u^{\gamma-1}\ddot{u}.    	
	\]
	In this way, by \eqref{pb:hom} we have
	\[
	-\frac{\ddot{u}_\gamma}{u_\gamma}=\gamma\left[-\frac{\ddot{u}}{u}+(1-\gamma)\frac{\dot{u}^2}{u^2}\right]=\gamma\left[\frac{\alpha U(\vt^*)}{u^{\alpha+2}}+(1-\gamma)\frac{\dot{u}^2}{u^2}\right];
	\]
	in other words, $u_\gamma$ solves the problem
	\[
	\begin{cases}
		\ddot{u}_\gamma+d_\gamma(t)u_\gamma=0 \\
		u_\gamma(0)=R^\gamma=u_\gamma(T),
	\end{cases}
	\]
	where 
	\begin{equation}\label{defn:d_gamma}
		d_\gamma(t)=\gamma\left[\frac{\al U(\vt^*)}{u(t)^{\al+2}}+(1-\gamma)\frac{\dot{u}(t)^2}{u(t)^2}\right].
	\end{equation}
	Moreover 
	\[
	\frac{\dot{u}_\gamma}{u_\gamma}=\frac{\gamma u^{\gamma-1}\dot{u}}{u^\gamma}=\gamma\frac{\dot{u}}{u};
	\]
	therefore, if we show that there exists $\gamma\in(0,1)$ such that $d_\gamma(t)\geq c(t)$ in $[0,T]$, we can repeat the previous argument and, as in \eqref{eq:gamma_1}, show that for such $\gamma$
	\begin{equation}\label{eq:no_k}
		R\dot{w}(T)-R\dot{w}(0)+\gamma|v_{\xi^*}|\geq 0.
	\end{equation}
	By \eqref{defn:d_gamma}, such inequality is satisfied if there exists $\gamma\in(0,1)$ such that, for every $t\in[0,T]$
	\[
	d_\gamma(t)=\gamma\left[\frac{\al U(\vt^*)}{u(t)^{\al+2}}+(1-\gamma)\frac{\dot{u}(t)^2}{u(t)^2}\right]\geq\frac{\al U(\vt^*)}{u(t)^{\al+2}}-\frac{U''(\vt^*)}{u(t)^{\al+2}}=c(t)
	\]
	or, equivalently, if
	\[
	\gamma(1-\gamma)\frac{\dot{u}(t)^2}{u(t)^2}\geq(1-\gamma)\frac{\alpha U(\vt^*)}{u^{\alpha+2}}-\frac{U''(\vt^*)}{u^{\alpha+2}}.
	\]
	Since the left-hand side is always non-negative, it is enough to find a $\gamma\in(0,1)$ such that 
	\[
	\frac{U''(\vt^*)}{\al U(\vt^*)}>1-\gamma,
	\]
	but such $\gamma$ clearly exists since $U''(\vt^*)>0$ and $U(\vt^*)>0$ and, moreover, does not depends on $t$. Since now \eqref{eq:no_k} is proved, using \eqref{eq:w_geq_0} we easily deduce that
	\[
	\dot{w}(T)-k\dot{w}(0)+\gamma|v_{\xi^*}|\geq 0\quad \mbox{for every}\ k\in[-1,1]
	\]
	and, choosing $C=(1-\gamma)|v_{\xi^*}|>0$, then \eqref{eq:w} is proven, together with the lemma.
\end{proof}

\begin{remark}\label{rem:uniformity}
	One could think that the proof of Lemma \ref{lem:var_ext} could be concluded with the inequality \eqref{eq:gamma_1} since, for every $k\in[-1,1]$
	\[
	\dot{w}(T)-k\dot{w}(0)+|v_{\xi^*}|\geq\dot{w}(T)-\dot{w}(0)+|v_{\xi^*}|=C_1>0.
	\]
	However, this estimate would not be enough for the purposes of this work, since we need a uniform estimate which does not depend on $\ve$. On the other hand, the constant $C=(1-\gamma)|v_{\xi^*}|$ provided at the end of the lemma depends only on $\vt_0^*$ and so it joins this uniformity and allows us to extend this argument also for the $N$-centre problem driven by $\Veps$ when $\ve$ is sufficiently small.
\end{remark}

As a consequence of Lemma \ref{lem:var_ext}, of Remark \ref{rem:uniformity} and of the uniform behaviour of the dynamical system when $\ve$ is small and $p_0$ and $p_1$ are not far from $\xi^*$ (see the discussion at page \pageref{discussion}), we can obtain the same result for the $N$-centre problem. Note that here we will refer to the notations of Lemma \ref{lem:diff_L_ext} and thus we will consider only the Jacobi length from $p_0$ to $p_1$; of course, an equivalent result holds for any pair $(p_{2j},p_{2j+1})$, for $j=1,\ldots,n-1$.

\begin{theorem}\label{thm:var_ext}
	There exists $\bar{\ve}_{ext}>0$ such that, for any $\ve\in(0,\bar{\ve}_{ext})$, if
	\[
	\mathbf{\bar{p}}=(\bar{p}_0,\bar{p}_1,\ldots,\bar{p}_{2n})\in\mathcal{S}
	\]
	is a minimizer of $\mathbf{L}$ provided in Lemma \ref{lem:glueing_exist}, then there exists $\psi\in\mathcal{T}_{p_0}(\partial B_R)$ and there exists $C_{ext}>0$ such that
	\[
	\frac{\partial\mathcal{L}_{ext}}{\partial p_0}(\bar{p}_0,\bar{p}_1)[\psi]=-\frac{1}{\sqrt{2}}\langle\dot{y}_{ext}(0),\psi\rangle\leq -C_{ext}<0.
	\]
\end{theorem}

\subsubsection{An explicit variation on the internal path}

To conclude this section and to finally prove that a minimizer of $\mathbf{L}$ is actually an inner point of $\mathcal{S}$, we need another preliminary result. Indeed, it is necessary to give an estimate of the final velocity of the inner arc $y_{int}=y_{int}(p_{2n-1},p_0)$ defined in $[0,T_{0,2n-1}]$, with respect to the tangent space spanned by $p_0^\perp$. As for the external arc, we are going to provide a result for $\ve=0$ and then we will extend it for $\ve$ sufficiently small by uniformity. In order to do this, consider again the notation introduced before the Lemma \ref{lem:var_ext}. We prove the following result.

\begin{lemma}\label{lem:var_int}
	There exists $\delta=\delta(\vt^*),C=C(\vt^*)>0$ such that, for any $\phi\in\R$ verifying $0<|\phi|<\delta$ , the following holds
	\[
	\frac{\langle\dot{y}_{int}(T_{0,2n-1}),p_0(\phi)^\perp\rangle}{\phi}\geq C,
	\]
	where $p_0(\phi)$ follows the notations in \eqref{eq:p_phi} at page \pageref{eq:p_phi}.
\end{lemma}
\begin{proof}
	We have put $\ve=0$, therefore we are now studying an anisotropic Kepler problem driven by $\Wzero$ (see Proposition \ref{prop:scaling}) and, in particular, $y_{int}$ in this setting is exactly one of the collision trajectories studied in \cite{BCTmin}. Actually, $y_{int}$ is a trajectory which emanates from collision; therefore, thanks to the time reversibility, we can consider $w(t)\uguale y_{int}(T_{0,2n-1}-t)$ which is defined again in $[0,T_{2n-1}]$, starts from $\partial B_R$ and finishes in collision with the origin (see also Figure \ref{fig:var_int}). For a vector $y\in\R^2$ we will denote by $\widehat{y}$ its angle with the horizontal axis with respect to the canonical basis of $\R^2$. As a starting point, without loss of generality we assume $\phi>0$ (as in the proof of Lemma \ref{lem:var_ext}) and we note that 
	\begin{equation}\label{eq:angoli}
		\begin{aligned}
			\langle\dot{y}_{int}(T_{0,2n-1}),p_0(\phi)^\perp\rangle&=\langle-\dot{w}(0),p_0(\phi)^\perp\rangle \\
			&=|\dot{w}(0)||p_0(\phi)^\perp|\cos\left(\widehat{\dot{w}(0)}-\left(\vt_0^*+\phi-\frac{\pi}{2}\right)\right) \\
			&=|\dot{w}(0)||p_0(\phi)^\perp|\sin\left(\widehat{\dot{w}(0)}-\left(\vt_0^*+\phi\right)-\pi\right)
		\end{aligned}
	\end{equation}
	and so, since $|\dot{w}(0)|$ and $|-p_0(\phi)^\perp|$ are far from $0$ when $\phi\to 0^+$, our proof reduces to study the behaviour of the angles (cf Figure \ref{fig:var_int}). Actually, it is clear that the angle $\widehat{\dot{w}(0)}$ depends on $\phi$ and, in particular
	\[
	\widehat{\dot{w}(0)}(\phi)\to\vt^*+\pi\quad\mbox{as}\ |\phi|\to 0^+,
	\]
	since the inner arc tends to the collision homothetic motion when $|\phi|\to 0^+$. This suggests to follow the approach of \cite{BCTmin} and to take into account McGehee coordinates. The change of variables, which of course again depends on $\phi$, with respect to the trajectory $w$ then reads    
	\[
	\begin{cases}
		r(t)=|w(t)| \\
		\vt(t)=\widehat{w(t)} \\
		\vp(t)=\widehat{\dot{w}(t)}
	\end{cases}
	\quad\mbox{with}\quad\begin{cases}
		r(0)=|w(0)|=R \\
		\vt(0)=\widehat{w(0)}=\vt^*+\phi \\
		\vp(0)=\widehat{\dot{w}(0)}
	\end{cases}.
	\]
	On the other hand, since $w$ is a collision solution at energy $-1$ of the anisotropic Kepler problem driven by $\Wzero(w)=r^{-\al}U(\vt)$, following Section 2 of \cite{BCTmin} $(r,\vt,\vp)$, after a time rescaling, solves
	\begin{equation}\label{eq:mcg}
		\begin{cases}
			r'=2r(U(\vt)-r^\al)\cos(\vp-\vt) \\
			\vt'=2(U(\vt)-r^\al)\sin(\vp-\vt) \\
			\vp'=U'(\vt)\cos(\vp-\vt)+\al U(\vt)\sin(\vp-\vt).
		\end{cases}
	\end{equation}
	Following again Section 2 of \cite{BCTmin}, since $\vt^*$ is such that $U'(\vt^*)=0$ and $U''(\vt^*)$, then the triplet $(0,\vt^*,\vt^*+\pi)$ is a hyperbolic equilibrium point for \eqref{eq:mcg} such that:
	\begin{itemize}
		\item its stable manifold is two-dimensional;
		\item the two eigendirections that span its stable manifold are 
		\[
		\begin{aligned}
			v_r&=(1,0,0) \\
			v^-&=\left(0,1,\frac12+\frac{\al}{4}+\frac14\sqrt{(2-\al)^2+8\frac{U''(\vt^*)}{U(\vt^*)}}\right)
		\end{aligned}
		\]
		and the corresponding eigenvalues are
		\[
		\begin{aligned}
			\la_r&=-2U(\vt^*)<0 \\
			\la^-&=\frac{2-\al}{2}U(\vt^*)-\frac12\sqrt{(2-\al)^2U(\vt^*)^2+8U(\vt^*)U''(\vt^*)}<0.
		\end{aligned}
		\]    	
	\end{itemize}
	Note that the third component of $v^-$ is greater than 1. Moreover, by the main result (Theorem 5.2) in \cite{BCTmin}, we have that $w=(r,\vt,\vp)$ belongs to the stable manifold of $(0,\vt^*,\vt^*+\pi)$ and in particular  $\vp(0)$ is a $\mathcal{C}^2$ function of $\vt(0)=\vt^*+\phi$. This is one of consequences of Lemma 3.2 in \cite{BCTmin}, together with the fact that, when $\vt(0)\to\vt^*$ in some way, the growth ratio of $\vp=\vp(0)$ as a function of $\vt=\vt(0)$ tends to the slope of $v^-$ projected on $(\vt,\vp(\vt))$. In other words, recalling that $\vp(\vt^*)=\vt^*+\pi$, we have that
	\[
	\frac{\vp(\vt)-(\vt^*+\pi)}{\vt-\vt^*}\to 1-\frac{\la^-}{2U(\vt^*)},\quad\mbox{as}\ \vt\to\vt^*
	\]
	and thus
	\[
	\frac{\vp(\vt)-(\vt+\pi)}{\vt-\vt^*}\to-\frac{\la^-}{2U(\vt^*)}>0,\quad\mbox{as}\ \vt\to\vt^*.
	\]
	Now, from \eqref{eq:angoli} and the above limiting behaviour, since $\vt\to\vt^*$ as $\phi\to 0^+$, we have that 
	\[
	\begin{aligned}
		\frac{\langle\dot{y}_{int}(T_{0,2n-1}),p_0(\phi)^\perp\rangle}{\phi}&=|\dot{w}(0)||p_0(\phi)^\perp|\frac{\sin\left(\widehat{\dot{w}(0)}-(\phi+\vt^*)-\pi\right)}{\phi} \\
		&=R|\dot{w}(0)|\frac{\sin\left(\vp(\vt)-(\vt+\pi)\right)}{\vt-\vt^*} \\
		&\sim R|\dot{w}(0)|\frac{\vp(\vt)-(\vt+\pi)}{\vt-\vt^*}.
	\end{aligned}
	\]
	To conclude the proof, note that $\dot{w}(0)$ is uniformly bounded in $\phi$ by a constant $c>0$ which depends on the initial velocity of the homothetic collision motion. Therefore, the proof is concluded choosing the constant
	\[
	C=C(\vt^*)\uguale-\frac{c\la^-}{2U(\vt^*)}>0.\qedhere
	\]
\end{proof}
As we have said, it is possible to extend the previous result for $\ve$ sufficiently small. Indeed, from Proposition \ref{prop:scaling} we know that the potential $\Veps$ converges uniformly to $\Wzero$ on every compact set of $\R^2\setminus\{0\}$; moreover, we have already seen in Lemma \ref{lem:uniform} and Theorem \ref{thm:R} that when $\ve\to 0^+$, a sequence of minimizers of the Maupertuis' functional $\mathcal{M}$ converges uniformly to a minimizer of the Maupertuis' functional 
\[
\mathcal{M}_0(u)\uguale\frac12\int_0^1|\dot{u}|^2\int_0^1\left(-1+\Wzero(u)\right).
\]
As for Theorem \ref{thm:var_ext}, this is enough for the proof of the next result.
\begin{theorem}\label{thm:var_int}
	There exists $\bar{\ve}_{int}>0$ such that, for any $\ve\in(0,\bar{\ve}_{int})$, if \[
	\mathbf{\bar{p}}=(\bar{p}_0,\bar{p}_1,\ldots,\bar{p}_{2n-1})\in\mathcal{S}
	\]
	is a minimizer of $\mathbf{L}$ provided in Lemma \ref{lem:glueing_exist}, then there exists $\psi\in\mathcal{T}_{p_0}(\partial B_R)$ and there exists $C_{int}>0$ such that
	\[
	\frac{\partial\mathcal{L}_{int}}{\partial p_0}(\bar{p}_{2n-1},\bar{p}_0)[\psi]=\frac{1}{\sqrt{2}}\langle\dot{y}_{int}(T_{\ve,2n-1}),\psi\rangle\leq -C_{int}<0.
	\]
\end{theorem}

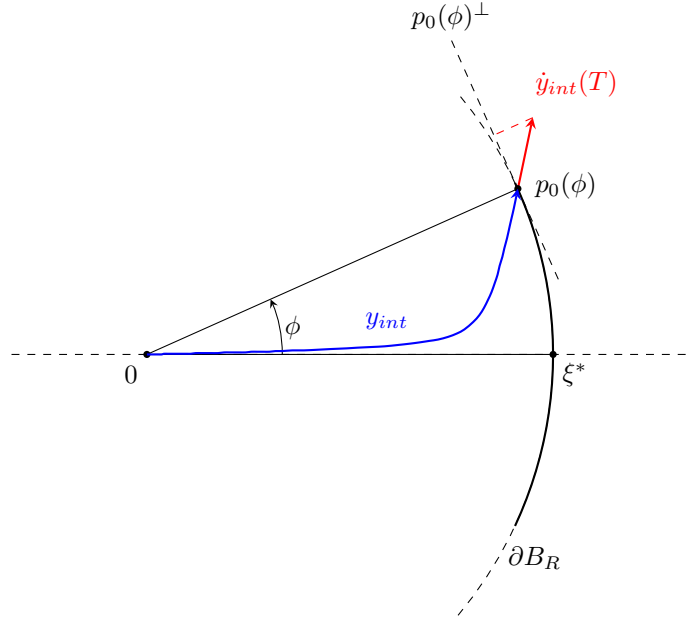
\begin{figure}
	\centering
	\begin{tikzpicture}[scale=0.9]
		\coordinate (O) at (0,0);
		\coordinate (x0) at (6,0);
		\coordinate (p0) at (5.48,2.45);
		
		\draw [dashed] (-2,0)--(8,0);
		\draw [dashed] (4.5,4.64)--(6.1,1.06);
		\draw  (O)--(x0);
		\draw  (O)--(p0);	
		\draw [dashed,red] (5.7,3.5)--(5.12,3.24);
		
		\draw [thick,-stealth,red] (p0)--(5.7,3.5); 
		
		\draw [thick,domain=-25:25] plot ({6*cos(\x)}, {6*sin(\x)});
		\draw [dashed,domain=-40:40] plot ({6*cos(\x)}, {6*sin(\x)});
		\draw [-stealth](2,0) arc (0:24:2); 	
		
		\fill (O) circle[radius=1.5pt];
		\fill (x) circle[radius=1.5pt];
		\fill (p0) circle[radius=1.5pt];
		
		\draw[thick,blue,-stealth] plot [smooth] coordinates {(0,0) (0.06,0) (0.09,0) (0.12,0) (0.14,0) (0.17,0) (0.22,0) (0.25,0.01) (0.31,0.01) (0.34,0.01) (0.39,0.01) (0.42,0.01) (0.46,0.01) (0.49,0.01) (0.53,0.01) (0.55,0.01) (0.6,0.01) (0.65,0.01) (0.72,0.02) (0.81,0.02) (0.89,0.02) (0.94,0.02) (1.05,0.02) (1.14,0.03) (1.22,0.03) (1.3,0.03) (1.36,0.03) (1.46,0.03) (1.54,0.04) (1.63,0.04) (1.72,0.04) (1.81,0.04) (1.9,0.05) (2,0.05) (2.15,0.06) (2.32,0.06) (2.53,0.07) (2.72,0.08) (2.95,0.09) (3.11,0.1) (3.32,0.11) (3.51,0.12) (3.71,0.14) (3.89,0.16) (4.04,0.18) (4.17,0.2) (4.3,0.22) (4.45,0.26) (4.58,0.31) (4.69,0.37)	(4.82,0.48) (4.91,0.58) (4.98,0.69) (5.04,0.81) (5.09,0.92) (5.14,1.06) (5.18,1.18) (5.21,1.31) (5.25,1.44) (5.28,1.57) (5.32,1.72) (5.36,1.9) (5.39,2.03) (5.42,2.16) (5.48,2.45)};
		
		\node[right] at (6,-0.3) {$\xi^*$};
		\node[right,red] at (5.6,4) {$\dot{y}_{int}(T)$};
		\node[right] at (5.2,-3) {$\partial B_R$};
		\node[left] at (0,-0.3) {$0$};
		\node[right] at (5.6,2.5) {$p_0(\phi)$};
		\node[right] at (1.9,0.4) {$\phi$};
		\node at (4.5,5) {$p_0(\phi)^{\perp}$};
		\node[left,blue] at (4,0.5) {$y_{int}$};
		
	\end{tikzpicture}
	\caption{Situation of Lemma \ref{lem:var_int}: the behaviour of the final velocity of an internal arc.}
	\label{fig:var_int}
\end{figure}

\subsubsection{Proof of (i) of Theorem \ref{thm:glueing}}

Define $\bar{\ve}\uguale\min\{\bar{\ve}_{ext},\bar{\ve}_{int}\}$, with $\bar{\ve}_{ext}$ and $\bar{\ve}_{int}$ introduced respectively in Theorem \ref{thm:var_ext} and Theorem \ref{thm:var_int}, and take $\ve\in(0,\bar{\ve})$. Assume by contradiction that the minimizer $\mathbf{\bar{p}}=(\bar{p}_0\,\bar{p}_1,\ldots,\bar{p}_{2n})$ of $\mathbf{L}$ provided in Lemma \ref{lem:glueing_exist} belongs to the boundary $\partial S$. To accomplish this absurd hypothesis it is not restrictive to assume that $\bar{p}_0\in\partial\mathcal{U}_0$ and thus to produce a variation on $\bar{p}_0$ such that the total length $\mathbf{L}$ decreases along this variation. This would lead to a contradiction and would conclude the proof. 

As a consequence of Theorem \ref{thm:var_ext} and Theorem \ref{thm:var_int} there exist a variation $\psi\in\mathcal{T}_{p_0}(\partial B_R)$ and a constant $C>0$ such that
\[
\begin{aligned}
	\frac{\partial\mathbf{L}}{\partial p_0}(\bar{p}_0,\bar{p}_1,\ldots,\bar{p}_{2n-1})[\psi]&=\frac{\partial\mathcal{L}_{ext}}{\partial p_0}(\bar{p}_0,\bar{p}_1)[\psi]+\frac{\partial\mathcal{L}_{int}}{\partial p_0}(\bar{p}_{2n-1},\bar{p}_0)[\psi] \\
	&=-\frac{1}{\sqrt{2}}\langle\dot{y}_{ext}(0),\psi\rangle+\frac{1}{\sqrt{2}}\langle\dot{y}_{int}(T_{\ve,2n-1}),\psi\rangle \\
	&\leq -2C<0.	
\end{aligned}
\]
Therefore, the minimality of $\mathbf{\bar{p}}$ is in contradiction with the above inequality and thus $\mathbf{\bar{p}}$ is necessarily an inner point of $\mathcal{S}$.

\subsection{Smoothness of the minimizers and existence of the corresponding periodic solutions: proof of (ii) of Theorem \ref{thm:glueing} and of Theorem \ref{th:main1}}

We conclude this section with the proof of the smoothness of the trajectory $\gamma_{\ve,\mathbf{\bar{p}}}$ and thus we provide the existence of a periodic solution of the anisotropic $N$-centre problem driven by $\Veps$.  

\begin{proof}[Proof of $(ii)$ of Theorem \ref{thm:glueing}]
	Since point (i) of Theorem \ref{thm:glueing} has been proved in the previous paragraph, for $\ve\in(0,\bar{\ve})$ we can consider a minimizer $\mathbf{\bar{p}}=(\bar{p}_0,\bar{p}_1,\ldots,\bar{p}_{2n})\in\mathring{\mathcal{S}}$ for $\mathbf{L}$ and now we know that 
	\[
	\frac{\partial\mathbf{L}}{\partial p_k}(\mathbf{\bar{p}})=0,\quad\mbox{for every}\ k=0,\ldots,2n.
	\] 
	Again, we can assume $k=0$ and thus, Lemma \ref{lem:diff_L_ext} and Lemma \ref{lem:diff_L_int} give that, for every $\psi\in\mathcal{T}_{p_0}(\partial B_R)$ we have
	\[
	\frac{\partial\mathbf{L}}{\partial p_0}(\mathbf{\bar{p}})[\psi]=\frac{\partial\mathcal{L}_{int}}{\partial p_0}(\bar{p}_{2n-1},\bar{p}_0)[\psi]+\frac{\partial\mathcal{L}_{ext}}{\partial p_0}(\bar{p}_0,\bar{p}_1)[\psi]=\frac{1}{\sqrt{2}}\langle\dot{y}_{int}(T_{2n-1})-\dot{y}_{ext}(0),\psi\rangle=0.
	\]
	The tangent space $\mathcal{T}_{p_0}(\partial B_R)$ is one-dimensional and it is spanned by a unit vector $\nu$ which is orthogonal to $p_0$. Therefore, if we denote by  $(\widehat{y,\nu})$ the angle included between the two vectors, we can deduce that 
	\[
	|\dot{y}_{int}(T_{\ve,2n-1})|\cos(\widehat{\dot{y}_{int}(T_{\ve,2n-1}),\nu})=|\dot{y}_{ext}(0)|\cos(\widehat{\dot{y}_{ext}(0),\nu})
	\]
	and the conservation of the energy both for $y_{int}$ and $y_{ext}$ at the point $\bar{p}_0$ leads to the equality
	\begin{equation}\label{eq:glueing_energy}
		|\dot{y}_{int}(T_{\ve,2n-1})|=|\dot{y}_{ext}(0)|.
	\end{equation}
	This implies that
	\[
	\cos(\widehat{\dot{y}_{int}(T_{\ve,2n-1}),\nu})=\cos(\widehat{\dot{y}_{ext}(0),\nu})
	\]
	and thus, since $\dot{y}_{int}(T_{\ve,2n-1})$ and $\dot{y}_{ext}(0)$ point outside $B_R$, we obtain that
	\[
	(\widehat{\dot{y}_{int}(T_{\ve,2n-1}),\nu})=(\widehat{\dot{y}_{ext}(0),\nu}).
	\]
	This, together with \eqref{eq:glueing_energy}, shows that $\dot{y}_{int}(T_{\ve,2n-1})=\dot{y}_{ext}(0)$ and thus $\gamma_{\ve,\mathbf{\bar{p}}}$ is $\mathcal{C}^1([0,\mathbf{T}_\ve])$.
	
	At this point, we have shown that if $\mathbf{\bar{p}}\in\mathring{\mathcal{S}}$ is a minimizer of $\mathbf{L}$, then the corresponding periodic trajectory $\gamma_{\ve,\mathbf{\bar{p}}}(t)$ is a classical solution of the $N$-centre problem \eqref{pb:gamma_N} at energy -1 when $t\in[0,\mathbf{T}_{\ve}]\setminus\{0,T_{\ve,0},\ldots,T_{\ve,2n-1}\}$ and it is a $\mathcal{C}^1$ function in $[0,\mathbf{T}_\ve]$. Since the junctions on $\partial B_R$ are smooth, we can extend the trajectory $\gamma_{\ve,\mathbf{\bar{p}}}$ by $\mathbf{T}_\ve$-periodicity to all $\R$. To conclude the proof of Theorem \ref{thm:glueing} we need to show that $\gamma_{\ve,\mathbf{\bar{p}}}$ is $\mathcal{C}^2(\R)$. In order to do that, let us consider again the solution arcs $y_{int}$ and $y_{ext}$ which glue on the point $\bar{p}_0$ (the same argument applies for all the other building blocks) and let us compute
	\[
	\begin{aligned}
		\lim\limits_{t\to T_{\ve,2n-1}^-}\ddot{\gamma}_{\ve,\mathbf{\bar{p}}}(t)&=\lim\limits_{t\to T_{\ve,2n-1}^-}\ddot{y}_{int}(t)=\lim\limits_{t\to T_{\ve,2n-1}^-}\nabla\Veps(y_{int}(t))\\
		&=\lim\limits_{t\to 0^+}\nabla\Veps(y_{ext}(t))=\lim\limits_{t\to 0^+}\ddot{y}_{ext}(t) \\ &=\lim\limits_{t\to 0}\ddot{\gamma}_{\ve,\mathbf{\bar{p}}}(t).
	\end{aligned}
	\] 
	This shows that $\gamma_{\ve,\mathbf{\bar{p}}}\in\mathcal{C}^2(\R)$ and concludes the proof of Theorem \ref{thm:glueing}.
\end{proof}

At this point, it remains to show that the existence of multiple periodic solutions holds also for the original $N$-centre problem, i.e., the problem
\begin{equation}\label{pb:N}
	\begin{cases}
		\ddot{x}=\nabla V(x) \\
		\frac12|\dot{x}|^2-V(x)=-h,
	\end{cases}
\end{equation}
where $V$ is the potential referred to the original centres $c_1,\ldots,c_N$ (see \eqref{def:potential} at page \eqref{def:potential}) and $h$ has to be chosen small enough. We recall that the multiplicity of periodic solutions for problem \eqref{pb:N} is determined both by a choice of a partition of the centres and by a minimal non-degenerate central configuration for $\Wzero$. As we have already discussed at page \pageref{intro:partitions}, we can describe all the possible behaviours of a periodic solution choosing a finite sequence of elements in the set
\[
\mathcal{Q}=\{Q_j:\,j=0,\ldots,m(2^{N-1}-1)-1\}.
\]
We need to link a sequence of $n$ elements in $\mathcal{Q}$ with a double sequence, composed by $n$ partitions and $n$ minimal non-degenerate central configurations of $\Wzero$. This can be done using Remark \ref{rem:division} at page \pageref{rem:division}, which yields the following correspondence
\[
(P_{l_0},\ldots,P_{l_{n-1}}),\,(\vt_{l_0}^*,\ldots,\vt_{l_{n-1}}^*) \longleftrightarrow (Q_{j_0},\ldots Q_{j_{n-1}})
\]
with $j_k=l_km+r_k$, for $l_k\in\{0,\ldots,2^{N-1}-2\}$, $r_k\in\{0,\ldots,m-1\}$, and thus $j_k\in\{0,\ldots,m(2^{N-1}-1)-1\}$. Finally, it is useful to characterize a solution provided in Theorem \ref{thm:glueing} with respect to its dependence on $P_{l_k}$ and $\vt_{r_k}$. Once $n\geq 1$ and $\ve\in(0,\bar\ve)$ are fixed, we have a periodic solution $\gamma_{\ve,\mathbf{\bar{p}}}$, with $\mathbf{\bar{p}}=(\bar{p}_0,\ldots,\bar{p}_{2n})\in\mathring{\mathcal{S}}$. Actually, it is clear from the discussion at page \pageref{eq:partition} that this solution depends on a choice of $n$ partitions and $n$ minimal non degenerate central configurations, i.e.,
\[
\gamma_{\ve;\mathbf{\bar{p}}}=\gamma(\ve;P_{l_0},\ldots,P_{l_{n-1}};\vt_{r_0}^*,\ldots,\vt_{r_{n-1}}^*).
\]

\begin{proof}[Proof of Theorem \ref{th:main1}]\label{proof:per}
	First of all, from Proposition \ref{prop:scaling}, in order to obtain a solution of \eqref{pb:N} as a rescaling of a solution of the problem driven by $\Veps$ at energy $-1$, the energy $h$ has to be in $(0,\bar{h})$, with $\bar{h}\uguale\bar{\ve}^\al$ and $\bar{\ve}>0$ is the one defined in Theorem \ref{thm:glueing}. Moreover, when such $h$ is fixed, a unique $\ve=h^{1/\al}$ is determined such that $B_\ve$ contains all the scaled centres, as well as a ball $B_R$ which is included in the Hill's region of $\Veps$ and that allows to build periodic solutions for the $\ve$-problem (as in \eqref{eq:R} at page \pageref{eq:R}, replacing $\tilde{\ve}$ with $\bar{\ve}$).  To such $R$, via Proposition \ref{prop:scaling} at page \pageref{prop:scaling}, we can associate a radius $\bar{R}\uguale h^{-1/\al}R>0$ which plays the same role for problem \eqref{pb:N}. Therefore, again by Proposition \ref{prop:scaling} and Remark \ref{rem:division}, when $n\geq 1$ and $(Q_{j_0},\ldots,Q_{j_{n-1}})\in\mathcal{Q}^n$ are fixed, we can define the function $x=x(Q_{j_0},\ldots,Q_{j_{n-1}};h)$ as the rescaling via $h$ of the solution $\gamma(\ve;P_{l_0},\ldots,P_{l_{n-1}};\vt_{r_0}^*,\ldots,\vt_{r_{n-1}}^*)$, with the rule
	\[
	j_k=l_km+r_k,\ \mbox{for every}\ k=0,\ldots,n-1.
	\]
	This $x$ will be clearly a classical and periodic solution of problem \eqref{pb:N} that crosses $2n$-times the circle $\partial B_{\bar{R}}$ in chosen neighbourhoods of the points $\bar{R}e^{i\vt_{r_k}^*}$.
\end{proof}

\section{Construction of a symbolic dynamics}\label{sec:symb_dyn}

In this final section we apply the results obtained in the previous ones in order to prove the existence of a symbolic dynamics for the anisotropic $N$-centre problem. In particular, we will firstly see that the multiplicity of periodic trajectories found in Theorem \ref{th:main1} allows to prove the existence of a collision-less symbolic dynamics as stated in Theorem \ref{cor:symb}. Exploiting similar techniques, such result will be obtained both for collision-orbits (Theorem \ref{thm:symb_dyn}) and for the anisotropic 2-centre problem with a unique minimal non degenerate central configuration (Theorem \ref{thm:symb_dyn2}).

\subsection{Collision-less symbolic dynamics, proof of Theorem \ref{cor:symb}}\label{subsec:sym_dyn_nocoll}
	
To start with, we recall that $N$ is the number of the centres, while $m$ represents the number of minimal non-degenerate central configurations for the leading potential $\Wzero$ far from the singularity. We assume that $N\geq 3$ and $m\geq 1$ or, equivalently, that $N\geq 2$ and $m\geq 2$, and we require again hypotheses \eqref{hyp:V}-\eqref{hyp:V_al} on the potential $V$ (see Remark \ref{rem:division}). Moreover, we fix $h\in(0,\bar{h})$, where the threshold $\bar{h}$ has been determined in the previous section. By means of Theorem \ref{th:main1} we have that, for any $n\geq 1$ and for any finite sequence $(Q_{j_0},\ldots,Q_{j_{n-1}})\sset\mathcal{Q}^n$, there exists a classical periodic solution $x=x(Q_{j_0},\ldots,Q_{j_{n-1}};h)$ of the equation $\ddot{x}=\nabla V(x)$ at energy $-h$ and there exists $\bar{R}=\bar{R}(h)>0$ such that the solution $x$ crosses the circle $\partial B_{\bar{R}}$ $2n$-times in one period at the instants $(t_{k})_{k=0}^{2n-1}$. In particular, for any $k=0,\ldots,n-1$ there exists a neighbourhood $\mathcal{U}_{r_k}\uguale\mathcal{U}(\bar{R}e^{i\vt_{r_k}^*})$ on $\partial B_{\bar{R}}$ such that, if we define $x_{k}\uguale x(t_k)$, we have that
\begin{itemize}
	\item when $t\in(t_{2k},t_{2k+1})$ the solution stays outside $B_{\bar{R}}$ and
	\[
	x_{2k},\, x_{2k+1}\in\mathcal{U}_{r_k}
	\]
	\item in the interval $(t_{2k+1},t_{2k+2})$ (we clearly set $t_{2n}\uguale t_0$ to close the trajectory) the solution stays inside $B_{\bar{R}}$, it separates the centres according to the partition $P_{l_k}$ and
	\[
	x_{2k+2}\in\mathcal{U}_{r_{k+1}},
	\]
\end{itemize}
keeping in mind the correspondence
\[
Q_{j_k}\longleftrightarrow(P_{l_k},\vt_{r_k}^*),\ \mbox{with}\ j_k=l_km+r_k.
\]
We recall that this piecewise solution has been determined with several steps in the previous sections, working with a normalized version of the $N$-centre problem, driven by $\Veps$. In the same way, thanks to Theorem \ref{thm:outer_dyn}, Theorem \ref{thm:inner_dyn} and Proposition \ref{prop:scaling}, we can distinguish between the solution arcs outside and inside $B_{\bar{R}}$ in this way:
\begin{itemize}
	\item we denote by $x_{ext}(\cdot;x_{2k},x_{2k+1};h)$ the piece of outer solution which connects $x_{2k}$ and $x_{2k+1}$, defined on its re-parametrized interval $[0,T_{ext}(x_{2k},x_{2k+1};h)]$;
	\item we denote by $x_{P_{j_k}}(\cdot;x_{2k+1},x_{2k+2};h)$ the piece of inner solution which connects $x_{2k+1}$ and $x_{2k+2}$ and separates the centres with respect to the partition $P_{l_k}$, defined, up to reparametrizations, on the interval $[0,T_{P_{j_k}}(x_{2k+1},x_{2k+2};h)]$. 	
\end{itemize}
We recall that the inner arc for the $\ve$-problem has been determined as a re-parametrization of a minimizer of the Maupertuis' functional in Section \ref{sec:inner}. On the other hand, we know that the Maupertuis' principle (Theorem \ref{thm:maupertuis}) joins also a vice-versa, i.e., if we set $\omega(x_{2k+1},x_{2k+2};h)\uguale 1/T_{P_{j_k}}(x_{2k+1},x_{2k+2};h)$, the function
\[
v_{P_{j_k}}(t;x_{2k+1},x_{2k+2};h)\uguale x_{P_{j_k}}(t/\omega(x_{2k+1},x_{2k+2};h);x_{2k+2},x_{2k+2};h)
\]
will be a critical point of the Maupertuis' functional $\mathcal{M}_h$ at a positive level, in a suitable space. In particular, we can introduce the set of $H^1$-paths
\[
\hat{\mathcal{H}}_{x_{2k+1},x_{2k+2}}([0,1])\uguale\left\lbrace v\in H^1([0,1];\R^2)\ \middle\lvert\ \begin{aligned} 
&v(0)=x_{2k+1},\ v(1)=x_{2k+2}, \\ &v(t)\neq c_j\ \forall\,t\in[0,1]
\forall\,j
\end{aligned}\right\rbrace
\]
and its $H^1$-closure
\[
\mathcal{H}_{x_{2k+1},x_{2k+2}}([0,1])\uguale\{v\in H^1([0,1];\R^2):\,v(0)=x_{2k+1},\ v(1)=x_{2k+2}\}.
\]
Now, recalling that in the $\ve$-problem we have studied the existence of inner solutions inside $B_R$, with $R=h^{1/\al}\bar{R}$, for any $p_{2k+1},p_{2k+2}\in\partial B_R$ we can consider the points $x_{2k+1}=h^{-1/\al}p_{2k+1}$, $x_{2k+2}=h^{-1/\al}p_{2k+2}\in\partial B_{\bar{R}}$. Moreover, we recall the analogue of the space defined above in the $\ve$-context, i.e., the spaces $\hat{H}_{p_{2k+1},p_{2k+2}}([0,1])$ and $H_{p_{2k+1},p_{2k+2}}([0,1])$ introduced in Section \ref{sec:inner} and we consider the bijective map
\[
J\colon H_{p_{2k+1},p_{2k+2}}([0,1])\to\mathcal{H}_{x_{2k+1},x_{2k+2}}([0,1])
\] 
such that $J(u)=h^{1/\al}u$, for any $u\in H_{p_{2k+1},p_{2k+2}}([0,1])$. It is clear that the topological behaviour of an arc in $H_{p_{2k+1},p_{2k+2}}([0,1])$ with respect to the centres $c_j'$ naturally translates on the same behaviour for its image through $J$ with respect to the centres $c_j$. In light of this, for any $P_j\in\mathcal{P}$, we recall the definition \eqref{def:K_P} of the minimization space $\hat{K}_{P_j}^{p_1,p_2}$ and its $H^1$-closure $K_{P_j}^{p_1,p_2}$ given at page \pageref{def:K_P}, and we set
\[
\begin{aligned}
	\hat{\mathcal{K}}_{P_j}^{x_{2k+1},x_{2k+2}}([0,1])&=J\left(\hat{K}_{P_j}^{p_{2k+1},p_{2k+2}}([0,1])\right), \\
	\mathcal{K}_{P_j}^{x_{2k+1},x_{2k+2}}([0,1])&=J\left(K_{P_j}^{p_{2k+1},p_{2k+2}}([0,1])\right).
\end{aligned}
\]
Now, since the inner arc $y_{int}(\cdot;p_{2k+1},p_{2k+2};\ve)$ with respect to the partition $P_j$ provided in Theorem \ref{thm:inner_dyn} re-parametrizes a minimizer of the Maupertuis functional in $K_{P_j}^{p_{2k+1},p_{2k+2}}([0,1])$, we can immediately conclude that $v_{P_j}
(\cdot;x_{2k+1},x_{2k+2};h)$ will be a minimizer of the Maupertuis' functional $\mathcal{M}_h$ in $\mathcal{K}_{P_j}^{x_{2k+1},x_{2k+2}}([0,1])$.

The rest of this section is devoted to the proof of Theorem \ref{cor:symb} as a consequence of Theorem \ref{th:main1}, i.e., to prove the existence of a symbolic dynamics with set of symbols $\mathcal{Q}$. For this reason, we start with the definition of a suitable subset $\Pi_h$ of the energy shell 
\[
\mathcal{E}_h=\left\lbrace(x,v)\in(\R^2\setminus\{c_1,\ldots,c_N\})\times\R^2:\,\frac12|v|^2-V(x)=-h\right\rbrace
\]
which is a 3-dimensional submanifold of $\R^2\setminus\{c_1,\ldots,c_N\})\times\R^2$ and it is invariant for the flow $\Phi^t$ induced by the vector field 
\[
\begin{aligned}
	F\colon\R^2\setminus\{c_1,\ldots,c_N\})&\times\R^2\to\R^2\times\R^2 \\
	&(x,v)\mapsto F(x,v)=(v,\nabla V(x)).
\end{aligned}
\]
As a starting point, for a neighbourhood $\mathcal{U}_r=\mathcal{U}(\bar{R}e^{i\vt_r^*})$ provided in Theorem \ref{th:main1} ($r=0,\ldots,m-1$), we define the sets of pairs $(x,v)$ such that $x\in\mathcal{U}_r$ and $v$ is not tangent to the same circle, i.e.,
\[
\mathcal{E}_{h,\bar{R},r}^{\pm}\uguale\left\lbrace (x,v)\in\mathcal{E}_h:\,x\in\mathcal{U}_r,\ \langle x,v\rangle\gtrless 0\right\rbrace.
\]
We note that for a pair in $\mathcal{E}_{h,\bar{R},r}^+$ the velocity points towards the outer of $B_{\bar{R}}$ (the converse holds for $\mathcal{E}_{h,\bar{R},r}^-$) and that both sets are included in the 2-dimensional inertial surface 
\[
\Sigma_h=\{(x,v)\in\mathcal{E}_h:\,|x|=\bar{R}\}.
\] 
Therefore, it is clear that, if we consider the restriction $F_h\uguale F|_{\mathcal{E}_h}$ of the vector field, it is transverse to $\mathcal{E}_{h,\bar{R},r}^+$ (for more details we refer to Section \ref{sec:outer}). 

For every $(x,v)\in\mathcal{E}_{h,\bar{R},r}^+$ we introduce the sets
\[
\mathbb{T}^\pm(x,v)\uguale\left\lbrace t\in(0,+\infty):\,\Phi^t(x,v)\in\mathcal{E}_{h,\bar{R},s}^\pm,\ \mbox{for some}\ s\in\{0,\ldots,m-1\}\right\rbrace
\]
and the sets 
\[
\left(\mathcal{E}_{h,\bar{R},r}^+\right)^\pm\uguale\left\lbrace (x,v)\in\mathcal{E}_{h,\bar{R},r}^+:\,\mathbb{T}^\pm(x,v)\neq\emptyset\right\rbrace.
\]
Note that in general the sets $\mathbb{T}^\pm(x,v)$ could be empty, since the piece of trajectory which starts in a neighbourhood $\mathcal{U}_r$ and points towards the outer of $B_{\bar{R}}$ needs to have an initial velocity $v$ which satisfies a behaviour well described in Lemma \ref{lem:var_ext}. Beside that, note that Theorem \ref{th:main1} ensures that the sets $\left(\mathcal{E}_{h,\bar{R},r}^+\right)^\pm$ are non-empty, since the theorem provides periodic solutions of the equation $\ddot{x}=\nabla V(x)$ that cross the circle $\partial B_{\bar{R}}$ an infinite number of times, exactly inside the neighbourhoods $\mathcal{U}_r$, in which the transversality condition $\langle\dot{x},x\rangle\gtrless 0$ is clearly satisfied. Moreover, the continuous dependence on initial data, together with the transversality of $\mathcal{E}_{h,\bar{R},r}^+$ with respect to the vector field $F$ guarantee that the set $\left(\mathcal{E}_{h,\bar{R},r}^+\right)^+$is open. At this point, for $(x,v)\in\left(\mathcal{E}_{h,\bar{R},r}^+\right)^+$ we define
\[
T_{\text{min}}^+(x,v)\uguale\inf\mathbb{T}^+(x,v),
\]
while for $x\in\left(\mathcal{E}_{h,\bar{R},r}^+\right)^-$ we set
\[
T_{\text{min}}^-(x,v)\uguale\inf\mathbb{T}^-(x,v).
\]
If we take $(x,v)\in\left(\mathcal{E}_{h,\bar{R},r}^+\right)^+\cap\left(\mathcal{E}_{h,\bar{R},r}^+\right)^-$ such that $T_{\text{min}}^-(x,v)<T_{\text{min}}^+(x,v)$, we can consider the piece of the orbit emanating from $(x,v)$ between the first two instants in which it crosses again $\partial B_{\bar{R}}$ in two of the admissible neighbourhoods, which is exactly the following restriction of the flow
\[
\{\Phi^t(x,v):\,t\in[T_{\text{min}}^-,T_{\text{min}}^+]\},
\]
where we have omitted the dependence on $(x,v)$ to ease the notations (see Figure \ref{fig:tmin}).

\begin{figure}
	\centering
	\begin{tikzpicture}[scale=1.2]
		\coordinate (0) at (0,0);
		\coordinate (p0) at (-0.26,1.477);
		\coordinate (p1) at (0.26,1.477);   		
		\coordinate (p2) at (-0.964,-1.149); 
		\coordinate (c1) at (-0.2,0.25);
		\coordinate (c2) at (-0.3,0.15);   		
		\coordinate (c3) at (-0.4,-0.1);    		
		\coordinate (c4) at (-0.1,-0.2); 
		\coordinate (c5) at (0.1,0);
		\coordinate (c6) at (0.15,0.2);
		
		
		\draw[-stealth,thick] (p0)--(-0.27,2.2);
		\draw[-stealth,thick] (p1)--(0.27,0.7);
		\draw[-stealth,thick] (p2)--(-1.5,-1.8);
		
		\draw (0,0) circle (1.5cm);
		
		
		\draw[dashed] (0)--(0,3);
		\draw[dashed] (0)--(-2.298,-1.928);
		
		\fill (p0)
		circle[radius=1.0pt]; 	
		\fill (p1)
		circle[radius=1.0pt];
		\fill (p2)
		circle[radius=1.0pt];
		\fill (c1)
		circle[radius=1.0pt];
		\fill (c2)
		circle[radius=1.0pt];
		\fill (c3)
		circle[radius=1.0pt];
		\fill (c4)
		circle[radius=1.0pt];
		\fill (c5)
		circle[radius=1.0pt];
		\fill (c6)
		circle[radius=1.0pt];

		\draw[-stealth] plot [smooth, tension=1.5] coordinates {(p0) (0,3) (p1)}; 
		
		\draw[-stealth,thick,blue] plot [smooth, tension=1] coordinates {(p1) (-0.05,0.2)  (p2)};   
		
		\node[left] at (-0.25,1.56) {\scriptsize $\boldsymbol{(x,v)}$};
		\node[right] at (0.27,1.6) {\scriptsize$\boldsymbol{\Phi^{T_\text{min}^-}(x,v)}$};
		\node[right] at (-0.92,-1.149) {\scriptsize $\boldsymbol{\Phi^{T_\text{min}^+}(x,v)}$};
		\node[above] at (0,3) {\scriptsize $\vt_r^*$};
		\node[left] at (-2.298,-1.928) {\scriptsize $\vt_s^*$};
		\node[left] at (c1) {\scriptsize $c_1$};
		\node[left] at (c2) {\scriptsize $c_2$};
		\node[left] at (c3) {\scriptsize $c_3$};
		\node[below] at (c4) {\scriptsize $c_4$};
		\node[below] at (c5) {\scriptsize $c_5$};
		\node[right] at (c6) {\scriptsize $c_6$};
		\node at (1.5,1) {\footnotesize $\partial B_{\bar{R}}$};
	\end{tikzpicture}
	\caption{This picture illustrates the notations introduced above. Indeed, we have drawn in bold the velocity vectors associated to every position in the configuration space, so that it is possible to visualize the flow associated to an initial data $\boldsymbol{(x,v)}\in\left(\mathcal{E}_{h,\bar{R},r}^+\right)^+\cap\left(\mathcal{E}_{h,\bar{R},r}^+\right)^-$. In particular, the blue arc represents the projection on the configuration space of the restriction of the flow between $T_\text{min}^-$ and $T_\text{min}^+$.}	
	\label{fig:tmin}
\end{figure}
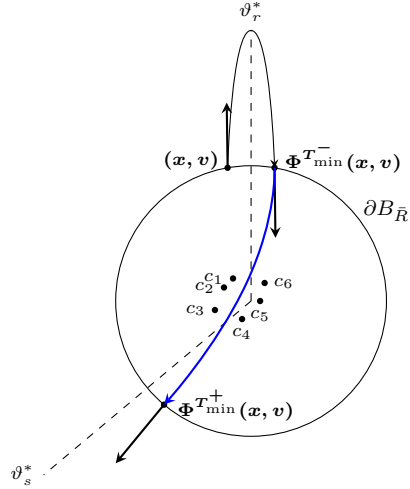

Recalling the set of symbols
\[
\mathcal{Q}=\{Q_j:\ j=0,\ldots,m(2^{N-1}-1)-1\},
\]
and recalling that $\pi_x\Phi^t(x,v)$ denotes the projection on the first component of the flow, let us now define the set
\[
\mathcal{E}_{h,\bar{R}}^{\mathcal{Q}}\uguale\left\lbrace (x,v)\in\left(\mathcal{E}_{h,\bar{R},r}^+\right)^+\cap\left(\mathcal{E}_{h,\bar{R},r}^+\right)^-\ \middle\lvert\
{
	\begin{aligned}
		& r\in\{0,\ldots,m-1\},\ T_{\text{min}}^-<T_{\text{min}}^+ \\
		& \{\pi_x\Phi^t(x,v)\}_{t\in[T_\text{min}^-,T_\text{min}^+]} \text{is the re-param.} \\
		&\text{of a minimizer of}\ \mathcal{M}_h\ \text{in the space} \\
		& \hat{\mathcal{K}}_{P_l}^{\pi_x\Phi^{T_\text{min}^-}(x,v),\pi_x\Phi^{T_\text{min}^+}(x,v)}([T_\text{min}^-,T_\text{min}^+]), \\ 
		&\text{for some}\ l\in\{0,\ldots,2^{N-1}-2\}, \\ 
		&\text{with}\ Q_j=Q_{lm+r}\in\mathcal{Q}
\end{aligned}}\right\rbrace.
\]
The above set is non-empty since Theorem \ref{th:main1} proves the existence of periodic solutions for the $N$-centre problem, which then identify an infinite number of points that belong to the set $\mathcal{E}_{h,\bar{R}}^\mathcal{Q}$. Indeed, the $x$-components of these points are nothing but the crosses that the periodic trajectories make on the circle $\partial B_{\bar{R}}$ when they start their motion outside the ball. We can then define a first return map on $\mathcal{E}_{h,\bar{R}}^\mathcal{Q}$ in this way
\[
\begin{aligned}
	\mathcal{R}\colon&\mathcal{E}_{h,\bar{R}}^\mathcal{Q}\to\mathcal{E}_{h,\bar{R}}^\mathcal{Q} \\
	&(x,v)\mapsto\mathcal{R}(x,v)\uguale\Phi^{T_\text{min}^+}(x,v),
\end{aligned}
\]
which is continuous as a consequence of our construction. Moreover, we can also define another map $\chi\colon\mathcal{E}_{h,\bar{R}}^\mathcal{Q}\to\mathcal{Q}$ such that
\[
\chi(x,v)=Q_j,\quad\mbox{if}\ \begin{cases} (x,v)\in\left(\mathcal{E}_{h,\bar{R},r}^+\right)^+ \\ \{\pi_x\Phi^t(x,v)\}_{t\in[T_\text{min}^-,T_\text{min}^+]}\in\hat{\mathcal{K}}_{P_l}^{\pi_x\Phi^{T_\text{min}^-}(x,v),\pi_x\Phi^{T_\text{min}^+}(x,v)}([T_\text{min}^-,T_\text{min}^+]) \\
	j=lm+r
\end{cases}.
\]
We can finally define the set $\Pi_h$ in this way
\[
\Pi_h\uguale\bigcap\limits_{j\in\Z}\mathcal{R}^j(\mathcal{E}_{h,\bar{R}}^\mathcal{Q}),
\]
which is exactly the set of all the possible initial data which generate solutions that cross the circle $\partial B_{\bar{R}}$ an infinite number of time having velocity directed toward the exterior of $B_{\bar{R}}$; moreover, each of these solutions, every time that travels inside $B_{\bar{R}}$, draws a partition $P_l$ of the centres for some $l$ and minimizing the Maupertuis' functional in the corresponding space $\hat{\mathcal{K}}_{P_l}$. To conclude this preliminary discussion we define also the application $\pi$ which maps every one of this initial data to its corresponding bi-infinite sequence of symbols, i.e., 
\[
\begin{aligned}
	\pi\colon&\Pi_h\to\mathcal{Q}^\Z \\
	&(x,v)\mapsto\pi(x,v)\uguale(Q_{j_k})_{k\in\Z},\ \mbox{with}\ Q_{j_k}\uguale\chi(\mathcal{R}^k(x,v));
\end{aligned}
\]
we also introduce the restriction of the first return map to the invariant submanifold $\Pi_h$ as $\mathfrak{R}\uguale\mathcal{R}|_{\Pi_h}$. At this point, we proceed with the proof of Theorem \ref{cor:symb}, i.e., we need to prove that the map $\pi$ that we have just defined is surjective and continuous. In order to do that, we need to prove some preliminary property on the pieces of solutions. The first one consists on showing that their definition intervals are uniformly bounded from above and from below.

\begin{lemma}\label{lem:bound}
	There exist two constants $c,C>0$ such that, for any $x_0,x_1\in\partial B_{\bar{R}}$ for which $x_{ext}(\cdot;x_0,x_1;h)$ exists, for any $x_2,x_3\in\partial B_{\bar{R}}$ and for any $P_j\in\mathcal{P}$ for which $x_{P_j}(\cdot;x_2,x_3;h)$ exists, we have
	\[
	\begin{aligned}
		&c\leq T_{ext}(x_0,x_1;h)\leq C \\
		&c\leq T_{P_j}(x_2,x_3;h)\leq C.
	\end{aligned}
	\]	
\end{lemma}
\begin{proof}
	Lemma \ref{lem:bound_ext} and Lemma \ref{lem:bound_int} provide such  uniform bounds for the $\ve$-problem; the conclusion is then a direct consequence of Proposition \ref{prop:scaling}.
\end{proof}

We also need a compactness lemma on sequences of minimizers of $\mathcal{M}_h$ which separate the centres with respect to the same partition. In particular, we want to prove that if the endpoints of the minimizers converge to a limit pair $(\bar{x}_2,\bar{x}_3)$ then the limit path is itself a minimizer of $\mathcal{M}_h$ in the space $\hat{\mathcal{K}}_{P_j}^{\bar{x}_2,\bar{x}_3}([0,1])$, for a fixed partition $P_j$.

\begin{lemma}\label{lem:compactness}
	Let $(x_2^n)\sset\mathcal{U}_2$ and $(x_3^n)\sset\mathcal{U}_3$ such that $(x_2^n,x_3^n)\to(\bar{x}_2,\bar{x}_3)$, with $\bar{x}_2\in\mathcal{U}_3$ and $\bar{x}_3\in\mathcal{U}_3$, where $\mathcal{U}_2$ and $\mathcal{U}_3$ are the neighbourhoods of two minimal non-degenerate central configurations of $\Wzero$ on $\partial B_{\bar{R}}$ which guarantee the existence of the internal arcs. Fix also a partition $P_j\in\mathcal{P}$ and let $v_n$ be a minimizer of $\mathcal{M}_h$ in the space $\hat{\mathcal{K}}_{P_j}^{x_2^n,x_3^n}([0,1])$. Then, there exists a subsequence $(v_{n_k})$ of $(v_n)$ and a minimizer $\bar{v}$ of $\mathcal{M}_h$ in the space $\hat{\mathcal{K}}_{P_j}^{\bar{x}_2,\bar{x_3}}([0,1])$ such that
	\[
	v_{n_k}\wconv\bar{v}\ \mbox{in}\ H^1.
	\]
\end{lemma}
\begin{proof}
	An analogous compactness property has been proved for sequences of minimizers of $\mathcal{M}_{-1}$ in the $\ve$-problem in Lemma \ref{lem:uniform} and Theorem \ref{thm:R}. Again, Proposition \ref{prop:scaling} gives the proof.
\end{proof}

We are now ready to give the proof of Theorem \ref{cor:symb} which follows the same technique employed in \cite{ST2012}.
\begin{proof}[Proof of Theorem \ref{cor:symb}]
	\textbf{Surjectivity of} $\boldsymbol{\pi}:$ Consider a bi-infinite sequence $(Q_{j_n})_{n\in\Z}\sset\mathcal{Q}^{\Z}$ and the sequence of finite sequences
	\[
	(Q_{j_0}),\ (Q_{j_{-1}},Q_{j_0},Q_{j_1}),\ \ldots\ (Q_{j_{-n}},Q_{j_{-n+1}},\ldots,Q_{j_{-1}},Q_{j_0},Q_{j_1},\ldots,Q_{j_{n-1}}Q_{j_n}),\ \ldots
	\] 
	If we fix $h\in(0,\bar{h})$, through Theorem \ref{th:main1} we can associate to each of these sequence a corresponding periodic solution; this will be made using this notation that takes into account the length of the finite sequence
	\[
	(Q_{j_{-n}},\ldots,Q_{j_{-1}},Q_{j_0},Q_{j_1},\ldots,Q_{j_n})\longleftrightarrow x^n(\cdot).
	\] 
	Without loss of generality, we can define $(x^n(0),\dot{x}^n(0))\in\Pi_h$ as the initial data such that the first symbol determined by $x^n$ is $Q_{j_0}$, for every $n\in\N$. In particular, we know that $j_0=l_0m+r_0$ and thus this symbol will refer to a first piece of solution, composed by an outer arc with endpoints in the neighbourhood $\mathcal{U}_{r_0}$ and inner arc that  agrees with the partition $P_{l_0}$ and that arrives in the neighbourhood $\mathcal{U}_{r_1}$. In this way, for every $n$ we can find a sequence of points $(x_k^n)_{k\in\Z}\sset\partial B_{\bar{R}}$ which correspond to the crosses of the periodic trajectory $x^n$ with the circle $\partial B_{\bar{R}}$. Note that since the trajectory is periodic, the sequence of points will be periodic too. We can now take into account the sequence of sequences
	\[
	\left\lbrace(x_k^n)_{k\in\Z}\right\rbrace_{n\in\N}
	\]
	in order to start a diagonal process that will imply a convergence on these sequences. If we fix $k=0$, since $\partial B_{\bar{R}}$ is compact we can extract a subsequence $(x_0^{n_0})_{n_0\in\N}$ such that $x_0^{n_0}\to\bar{x}_0$ when $n_0\to+\infty$. In the same way, we can fix $k=1$ and consider the subsequence $(x_1^{n_0})_{n_0\in\N}\sset\partial{B_{\bar{R}}}$ and extract a sub-subsequence $(x_1^{n_1})_{n_1\in\N}$ such that $x_1^{n_1}\to\bar{x}_1$ as $n_1\to+\infty$. This can be made for every $k\in\Z$, in order to find a subsequence $(x_k^{n_k})_{n_k\in\N}$ such that $x_k\to\bar{x}_k$ as $n_k\to+\infty$. At this point we can consider the diagonal sequence $(x_k^{n_n})_{n\in\N}$ and relabel it as $(x_k^n)_{n\in\N}$, so that
	\[
	\lim\limits_{n\to+\infty}x_k^n=\bar{x}_k,\ \mbox{for all}\ k\in\Z.
	\]	
	Note that all these limit points belong to $\partial B_{\bar{R}}$ so that we can connect two of them with an inner or outer arc; this would actually require that the points are inside the neighbourhoods $\mathcal{U}_{r_k}$ found in Theorem \ref{th:main1}, but up to restrict these neighbourhoods we can repeat the previous argument so that the limits would still be inside a neighbourhood in which the existence is guaranteed. Once this is clear, we can connect the points $\bar{x}_{2k},\bar{x}_{2k+1}\in\mathcal{U}_{r_k}$ with a unique outer arc using the technique illustrated in Theorem \ref{thm:outer_dyn}; we can also connect the point $\bar{x}_{2k+1}\in\mathcal{U}_{r_k}$ and the point $\bar{x}_{2k+2}\in\mathcal{U}_{r_{k+1}}$ following the procedure of Theorem \ref{thm:inner_dyn} so that the induced path would separate the centres according to the partition $P_{l_k}$. Repeating this procedure for every $k\in\Z$ we can glue together all these pieces to obtain a continuous function $\bar{x}\colon\R\to\colon\R^2$, using the same technique provided in Section \ref{sec:glueing}. We point out that $\bar{x}$ is not unique, since the inner pieces, coming from Maupertuis' minimizers, are not unique. In the following, we are going to show that $\bar{x}$ is a classical periodic solution of the equation $\ddot{x}=\nabla V(x)$ and verifies
	\[
	\begin{cases}
		(\bar{x}(0),\dot{\bar{x}}(0))\in\Pi_h \\
		\pi((\bar{x}(0),\dot{\bar{x}}(0)))=(Q_{j_k})_{k\in\Z}.
	\end{cases}
	\]
	If we introduce the set of collision instants of $\bar{x}$ as
	\[
	T_c(\bar{x})\uguale\left\lbrace t\in\R:\,\bar{x}(t)=c_j,\ \mbox{for some}\ j\in\{1,\ldots,N\}\right\rbrace,
	\]
	due to the nature of the sequence $(x^n)$ it is enough to show that $x^n\to\bar{x}$ in a $\mathcal{C}^2$ manner on every compact subset of $\R\setminus T_c(\bar{x})$. To start with, note that if we take $[a,b]\sset\R$ such that $\bar{x}(a)=\bar{x}_{2k}$ and $\bar{x}(b)=\bar{x}_{2k+1}$, then the outer arc connecting these two points depends on a continuous manner on the endpoints (see Theorem \ref{thm:outer_dyn}) and so $x^n\to\bar{x}$ uniformly on $[a,b]$. Moreover, if we take $[c,d]\sset\R$ such that $\bar{x}(c)=\bar{x}_{2k+1}$ and $\bar{x}(d)=\bar{x}_{2k+2}$, then the uniform convergence on $[c,d]$ is a straightforward consequence of Lemma \ref{lem:compactness}. This convergence also determines a unique choice for the inner solution that connects $\bar{x}_{2k+1}$ and $\bar{x}_{2k+2}$, so that now the function $\bar{x}$ is uniquely determined. Moreover, since the internal arcs provided in Theorem \ref{thm:inner_dyn} have a uniform distance $\delta$ from the centres, this actually proves that the uniform convergence of $x^n$ to $\bar{x}$ takes place in $\R\setminus T_c(\bar{x})$, i.e., $\bar{x}$ has no collision with the centres.
	
	At this point, the function $\nabla V(\bar{x}(\cdot))$ is continuous on the whole $\R$ and, since $x^n$ is a $\mathcal{C}^2$ solution of the equation $\ddot{x}=\nabla V(x)$ by the uniform convergence of $x^n$ to $\bar{x}$ on $[a,b]$ we have that
	\[
	\lim\limits_{n\to+\infty}\ddot{x}^n(t)=\lim\limits_{n\to+\infty}\nabla V(x^n(t))=\nabla V(\bar{x}(t)),
	\]
	for every $t\in[a,b]$. This means that the sequence $(\dot{x}^n(t))$ is equi-continuous in $[a,b]$; moreover, the energy equation implies that
	\[
	|\dot{x}^n(t)|=\sqrt{V(x^n(t))-h}\leq C
	\]
	for every $t\in[a,b]$ and for every $n\in\N$, i.e., the sequence $(\dot{x}^n(t))$ is also equi-bounded in $[a,b]$. This fact, together with the uniform convergence, finally shows that the sequence $(x^n(t))$ $\mathcal{C}^2$-converges to $\bar{x}(t)$ in $[a,b]$, for every compact set $[a,b]\sset\R$. As a consequence, $\bar{x}$ is a $\mathcal{C}^2$ solution of the equation $\ddot{x}=\nabla V(x)$ at energy $-h$ on every compact set of $\R$. As a final remark, note that the uniform convergence also implies the conservation of the topological constraint, i.e., the piece of  $\bar{x}$ between the points $x_{2k+1}$ and $x_{2k+2}$ will separate the centres with respect to $P_{l_k}$. This finally proves that $\pi((\bar{x}(0),\dot{\bar{x}}(0)))=(Q_{j_k})_{k\in\Z}$, where $j_k=l_km+r_k$.
	
	\textbf{Continuity of} $\boldsymbol{\pi}$: We recall that we can endow the set of bi-infinite sequences $\mathcal{Q}^\Z$ with the distance
	\[
	d((Q_m),(\tilde{Q}_m))\uguale\sum\limits_{m\in\Z}\frac{\rho(Q_m,\tilde{Q}_m)}{2^{|m|}},\ \forall\,(Q_m),(\tilde{Q}_m)\in\mathcal{Q}^\Z,
	\]
	where $\rho$ is the discrete metric defined through the Kronecker delta. Moreover, for every $m\in\Z$ we define the map 
	\[
	\begin{aligned}
		\pi_m\colon&\Pi_h\to\mathcal{Q} \\
		&(x,v)\mapsto \pi_m(x,v)\uguale\chi(\mathfrak{R}^m(x,v)),
	\end{aligned}
	\]
	i.e., it associates to $(x,v)$ the symbol corresponding to the $m$-th piece (composed by an outer arc and an inner arc) of the solution with initial data $(x,v)$. Given this, if we fix $(x_0,v_0)\in\Pi_h$, we need to show that for $\la>0$ there exists $\delta>0$ such that 
	\begin{equation}\label{eq:continuous}
		\forall\,(x,v)\in\Pi_h\ \mbox{s.t.}\ \|(x,v)-(x_0,v_0)\|<\delta\implies\sum\limits_{m\in\Z}\frac{\rho(\pi_m(x,v),\pi_m(x_0,v_0))}{2^{|m|}}<\la
	\end{equation}
	It is clear that we can find $m_0\in\N$ such that 
	\[
	\sum\limits_{|m|>m_0,m\in\Z}\frac{1}{2^{|m|}}<\la.
	\]
	For this reason and for the definition of the metric $d$ in the space $(\mathcal{Q},d)$, in order to prove \eqref{eq:continuous} it is enough to show that two initial data sufficiently close are mapped through $\pi_m$ to the same symbol $Q_m$, for any $m\in\{-m_0,\ldots,m_0\}$. Therefore to \eqref{eq:continuous} it is equivalent to prove that, for any $m_0\in\N$ there exists $\eta>0$ such that 
	\[
	\forall\,(x,v)\in\Pi_h\ \mbox{s.t.}\ \|(x,v)-(x_0,v_0)\|<\eta\implies\pi_m(x,v)=\pi_m(x_0,v_0),\ \forall\,|m|\leq m_0.
	\]
	If we take $m_0\in\N$, by means of Lemma \ref{lem:bound} there exists a time interval $[-a,a]$ such that every solution with initial data in $\Pi_h$ detects at least $2m_0+1$ symbols in $[-a,a]$, i.e., it determines at least $4m_0+2$ crosses on $\partial B_{\bar{R}}$. Moreover, the solution which emanates from the initial data $(x_0,v_0)$ is collision-free and its projection on the $x$-component has a uniform distance $\delta>0$ from the centres (see Theorem \ref{thm:inner_dyn}), i.e.,
	\[
	|\pi_x\Phi^t(x_0,v_0)-c_j|\geq\delta,\ \forall\,t\in[-a,a],\ \forall\,j\in\{1,\ldots,N\},
	\]
	recalling that $\pi_x$ denotes the projection on the $x$-component of the flow. At this point, if $(x,v)$ is sufficiently close to $(x_0,v_0)$, the continuous dependence on initial data implies that
	\[
	\left\lvert\pi_x\Phi^t(x,v)-\pi_x\Phi^t(x_0,v_0)\right\rvert<\frac{\delta}{2},
	\]
	for every $t\in[-a,a]$. This fact ensures that the flow associated to $(x,v)$ determines the same $2m_0+1$ symbols of the flow associated to $(x_0,v_0)$ so that, in particular
	\[
	\pi_m(x,v)=\pi_m(x_0,v_0),\ \forall\,m\in\{-m_0,\ldots,m_0\}.
	\]
	The proof is then concluded.
\end{proof}

\subsection{Collisional symbolic dynamics, proof of Theorem \ref{thm:symb_dyn}}\label{subsec:symb_dyn_coll}

We assume that $N\geq 2$ and $m\geq 2$ and we consider a potential $V\in\mathcal{C}^2(\R^2\setminus\{c_1,\ldots,c_N\})$ satisfying hypothesis \eqref{hyp:V}, that we recall here for the readers convenience
\[
	\tag{$V$} 
	\fontsize{9pt}{\baselineskip}
	\begin{cases}
		\al<2; \\
		\exists\,(\vt_l^*)_{l=0}^{m-1}\sset\cerchio\,:\,\ \forall\ l=0,\ldots,m-1,\ m>0,\ U''(\vt_l^*)>0,\  U(\vt)\geq U(\vt_l^*)>0,\ \forall\ \vt\in\cerchio; \\
	\end{cases}
\]
In order to treat a more general case, we have dropped hypothesis \eqref{hyp:V_al}, which, as it has been clear from Section \ref{sec:inner}, is crucial only in order to obtain collision-less periodic solutions. In this way, collisions with the centres may occur, but this does not prevent to build periodic solutions nor to show that a symbolic dynamics exists. Clearly, slight modifications of the previous proofs are in order and the most relevant change reveals to concern the number of symbols. We recall that collision-less periodic solutions for the anisotropic $N$-centre problem have been found in Section \ref{sec:glueing} as a juxtaposition of outer and inner solution arcs, provided respectively in Section \ref{sec:outer} and Section \ref{sec:inner}. No changes are involved with respect to the outer dynamics, since, when a solution arc travels outside a ball containing the centres, collisions can not exist. For this reason, the results on outer solution arcs provided in Section \ref{sec:outer} (Theorem \ref{thm:outer_dyn} and Lemma \ref{lem:bound_ext}) are valid also in this context. A different situation shows up concerning the inner dynamics since, due to the possible presence of collisions with the centres, Theorem \ref{thm:inner_dyn} or, equivalently, Theorem \ref{thm:inner_dyn_wind} do not hold. Beside that, we can provide a result which guarantees the existence of possible colliding solution arcs which connect two points which belongs to the neighbourhoods of two minimal non-degenerate central configurations of $\Wzero$ on $\partial B_R$ as a collisional counterpart of Theorem \ref{thm:inner_dyn}-\ref{thm:inner_dyn_wind}. Following the notations of Section \ref{sec:inner} (in particular, we refer to \eqref{def:K_P}-\eqref{def:k_l}), given $R>0$ as in Theorem \ref{thm:inner_dyn}-\ref{thm:inner_dyn_wind}, we define the set of all the $H^1$ possible colliding paths connecting two points $p_1,p_2\in\partial B_R$ as 
\[
\begin{aligned}
K^{p_1,p_2}([a,b])&\uguale\bigcup\limits_{P_j\in\mathcal{P}}K_{P_j}^{p_1,p_2}([a,b])=\bigcup\limits_{l\in\mathfrak{I}^N}K_l^{p_1,p_2}([a,b]) \\
&=\{u\in H^1([a,b];\R^2):\,u(a)=p_1,\ u(b)=p_2\}.
\end{aligned}
\]
With these notations, the following result on the existence of possibly colliding internal solution arcs holds (cf. Theorem \ref{thm:inner_dyn}-\ref{thm:inner_dyn_wind}).
\begin{lemma}\label{lem:inner_coll}
	Assume that $N\geq 2$, $m\geq 2$ and consider a function $V\in\mathcal{C}^2(\R^2\setminus\{c_1,\ldots,c_N\})$ defined as in \eqref{def:potential} and satisfying \eqref{hyp:V}. Fix $R>0$ as in \eqref{eq:R} at page \pageref{eq:R}. Then, there exist $\ve_{int}>0$ such that, for any $\vt^*,\vt^{**}\in\Xi$ minimal non degenerate central configurations for $\Wzero$, defining $\xi^*\uguale Re^{i\vt^*},\xi^{**}\uguale Re^{i\vt^{**}}\in\partial B_R$, there exist two neighbourhoods $\mathcal{U}_{\xi^*},\mathcal{U}_{\xi^{**}}$ on $\partial B_R$ with the following property: 
	
	for any $\ve\in(0,\ve_{int})$, for any pair of endpoints $p_1\in\mathcal{U}_{\xi^*}$, $p_2\in\mathcal{U}_{\xi^{**}}$, given the potential $\Veps$ defined in \eqref{def:scaling} at page \pageref{def:scaling}, there exist $T>0$, a possibly empty set $T_c(y)\sset[0,T]$  and a possibly colliding solution $y\in K^{p_1,p_2}([0,T])$ of the inner problem
	\[
	\begin{cases}
		\begin{aligned}
			&\ddot{y}(t)=\nabla \Veps(y(t))  &t\in[0,T]\setminus T_c(y) \\
			&\frac{1}{2}|\dot{y}(t)|^2-\Veps(y(t))=-1   &t\in[0,T]\setminus T_c(y) \\
			&|y(t)|<R   &t\in(0,T) \\
			&y(0)=p_1,\quad y(T)=p_2. &
		\end{aligned}
	\end{cases}
	\]
	In particular, the set $T_c(y)$ is the set of collision instants of $y$, it is discrete and finite. Moreover, the solution $y$ is a re-parametrization of a minimizer of the Maupertuis' functional in the space $K^{p_1,p_2}([0,1])$.
\end{lemma}
\begin{proof}
	The proof of the existence of such solution $y$, together with the conservation of the energy, is a consequence of Proposition \ref{prop:existence}, Theorem \ref{thm:maupertuis} and Lemma \ref{lem:conn_comp}. Theorem \ref{thm:R} applies also in this context and provides the confinement $|y|<R$. The characterization of the set $T_c(y)$ is given in Lemma \ref{lem:coll_time}.
\end{proof}
Proceeding exactly as in Section \ref{sec:glueing}, this time without taking care of collisions with the centres, we can now state the collisional counterpart of Theorem \ref{th:main1}.
\begin{theorem}\label{thm:per_coll}
    Assume that $N\geq 2$ and $m\geq 2$. Consider a function $V\in\mathcal{C}^2(\R^2\setminus\{c_1,\ldots,c_N\})$ satisfying \eqref{hyp:V}. There exists $h^*>0$ such that, for every $h\in(0,h^*)$, $n\in\N_{\geq 1}$ and $(\vt_{j_0}^*,\ldots,\vt_{j_{n-1}}^*)\in\Xi^n$, there exists a periodic possibly colliding solution $x=x(\vt_{j_0}^*,\ldots,\vt_{j_{n-1}}^*;h)$ of the equation $\ddot{x}=\nabla V(x)$ with energy $-h$ with the following properties:
    \begin{itemize}
    	\item $x$ satisfies the equation $\ddot{x}=\nabla V(x)$ and the energy equation $\frac12|\dot{x}|^2-V(x)=-h$ almost everywhere in one period;
    	\item there exists $R^*=R^*(h)>0$ such that the solution $x$ crosses $2n$ times the circle $\partial B_{R^*}$ in one period, at times $(t_s)_{s=0}^{2n-1}$ in such a way that, for any $k=0,\ldots,n-1$:
    	\begin{itemize}
    		\item in the interval $(t_{2k},t_{2k+1})$ the solution stays outside $B_{R^*}$ and there exists a neighbourhood $\mathcal{U}_{j_k}=\mathcal{U}(R^* e^{i\vt_{j_k}^*})$ on $\partial B_{R^*}$ such that
    		\[
    		x(t_{2k}),x(t_{2k+1})\in\mathcal{U}_{j_k};
    		\]
    		\item in the interval $(t_{2k+1},t_{2k+2})$ the solution stays inside $B_{R^*}$ and may collide with the centres $c_1,\ldots,c_N$.
    	\end{itemize}
    \end{itemize}
\end{theorem} 
\begin{proof}
	The proof goes exactly as the proof of Theorem \ref{th:main1} (see Section \ref{sec:glueing} and in particular the proof of Theorem \ref{th:main1} at page \pageref{proof:per}), this time using Lemma \ref{lem:inner_coll} in place of Theorem \ref{thm:inner_dyn}. Note indeed that the glueing technique of Section \ref{sec:glueing} becomes actually easier in this case since Lemma \ref{lem:inner_coll} does not involve the partitions of the centres.
\end{proof}
At this point, we are left to show that this dynamical system admits a symbolic dynamics and thus to prove Theorem \ref{thm:symb_dyn}. In order to do that and in light of Theorem \ref{thm:per_coll}, we introduce the set of symbols 
\[
\mathcal{S}\uguale\Xi=\{\vt_1^*,\ldots,\vt_m^*\},
\]
which is exactly the set of minimal non-degenerate central configurations for $\Wzero$. This is not surprising, since we have provided inner arcs with a minimization process in the space $K^{p_1,p_2}$ that has no topological constraint involving partitions nor winding numbers. Therefore, the only way to distinguish sequences of inner and outer arcs is to switch from a strictly minimal non-degenerate central configuration to another. At this point, a completely analogous proof to the one given in Subsection \ref{subsec:sym_dyn_nocoll} allows to prove Theorem \ref{thm:symb_dyn}, once the Bernoulli right shift is defined over $\mathcal{S}^\Z$ in this way
\[
\begin{aligned}
	T_r\colon &\mathcal{S}^\Z\to \mathcal{S}^\Z \\
	&(\vt_s^*)\mapsto T_r((\vt_s^*))\uguale (\vt_{s+1}^*),
\end{aligned}
\]
this time making use of Theorem \ref{thm:per_coll}.
\begin{remark}
	We conclude this treatment with a brief remark concerning the number of symbols. The set $K^{p_1,p_2}$, over which we found inner solution arcs of Lemma \ref{lem:inner_coll}, is exactly the weak $H^1$-closure of the set 
	\[
	\hat{K}^{p_1,p_2}\uguale\bigcup\limits_{P_j\in\mathcal{P}}\hat{K}_{P_j}^{p_1,p_2}=\bigcup\limits_{l\in\mathfrak{I}^N}\hat{K}_l^{p_1,p_2},
	\]
	which actually contains all the possible $H^1$ collision-less paths considered in Theorem \ref{thm:inner_dyn}-\ref{thm:inner_dyn_wind}. This means that, in this context, we are no more forcing a path to separate the centres with respect to a prescribed partition or, equivalently, to a vector of winding numbers. The reader may object that this selection of paths could be too much raw, because one for instance could choose a partition $P_j$ (or, equivalently, a winding vector $l$) and work in the closure $K_{P_j}^{p_1,p_2}$ (eq., $K_l^{p_1,p_2}$). Nonetheless, it could happen that the minimizer found in $K_{P_j}^{p_1,p_2}$ may collide with all the centres and this fact would destroy the topological constraint. As a result, the alphabet involved in Subsection \ref{subsec:sym_dyn_nocoll} is no more suitable in this context and the only actual distinction resides in the choice of at least two elements $\vt^*,\vt^{**}$ of $\Xi$, to which corresponds the juxtaposition of an inner (Lemma \ref{lem:inner_coll}) and an outer (Theorem \ref{thm:outer_dyn}) solution arc. 
\end{remark}

\subsection{Case \texorpdfstring{$N=2$, $m=1$}{}, proof of Theorem \ref{thm:symb_dyn2}}

In conclusion, we analyse the case of the $2$-centre problem, in which the potential with the smallest homogeneity degree $\al$ has a unique minimal non-degenerate central configuration $\vt^*\in\Xi$. Clearly, it could happen that the sum potential $V$ itself is $-\al$-homogeneous, but this case would not change the proof. Our final aim is to provide a topological semi-conjugation of this dynamical system with a suitable shift space. Note that in this case the partition set is a singleton and thus it does not fit as an alphabet. For this reason, we consider as symbols the two centres $c_1,c_2$ and therefore, we introduce the alphabet
\[
\mathcal{B}\uguale\{c_1,c_2\}.
\]
Since we are dealing with a unique minimal non-degenerate central configuration $\vt^*\in\Xi$, it makes sense to assume $\vt^*=0$. Given $R>0$ as in \eqref{eq:R} at page \pageref{eq:R}, we define an open neighbourhood $\mathcal{U}$ on $\partial B_R$ of $\vt^*$ as the biggest neighbourhood such that, given the extremal points $p_{min}$ and $p_{max}$ of $\mathcal{U}$, the area between the arc on $\partial B_R$ and the chord which connect $p_{min}$ and $p_{max}$ does not contain the centres $c_1,c_2$ (see Figure \ref{fig:2centre}).

\begin{figure}
	\centering
	\begin{tikzpicture}[scale=0.9]
		\coordinate (O) at (0,0);
		\coordinate (x0) at (6,0);
		\coordinate (pmin) at (1.27,-2.72);
		\coordinate (pmax) at (1.27,2.72);
		\coordinate (c1) at (0.8,1.4);
		\coordinate (c2) at (0.3,-1.2);
		
		\draw [dashed] (-1,0)--(4.5,0);
		\draw [dashed] (pmin)--(pmax);
		
		\fill (O) circle[radius=1.5pt];
		\fill (pmax) circle[radius=1.5pt];
		\fill (pmin) circle[radius=1.5pt];
		\fill (c1) circle[radius=1.5pt];
		\fill (c2) circle[radius=1.5pt];

		\draw [thick,domain=-65:65] plot ({3*cos(\x)}, {3*sin(\x)});
		\draw [dashed,domain=-90:90] plot ({3*cos(\x)}, {3*sin(\x)});
		
		\node[right] at (4.5,-0.3) {$\vt^*$};
		\node[right] at (0,-3.4) {$\partial B_R$};
		\node[left] at (0,-0.3) {$0$};
		\node[left] at (c1) {$c_1$};
		\node[left] at (c2) {$c_2$};
		\node[below] at (pmin) {$p_{min}$};
		\node[above] at (pmax) {$p_{max}$};
		\node[right] at (2.8,1) {$\mathcal{U}$};
	\end{tikzpicture}
	\caption{The neighbourhood $\mathcal{U}$ of the unique central configuration $\vt^*$.}
	\label{fig:2centre}
\end{figure}
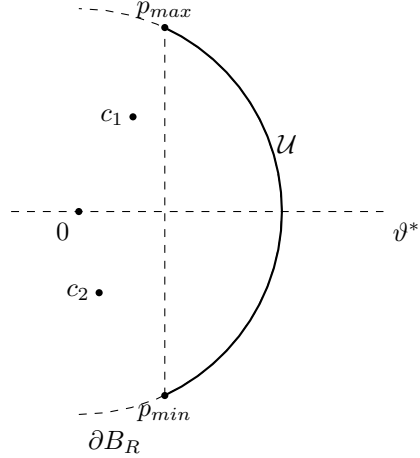

For two points $p_1=Re^{i\vt_1},p_2=Re^{i\vt_2}\in\mathcal{U}$, with $\vt_1,\vt_2\in(-\pi,\pi)$ and a path $u\in H^1([a,b];\R^2)$ which connects $p_1$ and $p_2$, we define the close path $\tilde{\Gamma}_u$ as the glueing between $u$ and the chord connecting $p_1$ and $p_2$. Define moreover $\delta\uguale|c_1-c_2|/2>0$. With these notations, we introduce the following classes of paths:
\[
	\hat{K}_{c_i}^{p_1,p_2}([a,b])\uguale\left\lbrace u\in H^1([a,b];\R^2)\,\middle\lvert\,\begin{aligned} 
		&u(a)=p_1,\ u(b)=p_2, \\ &u(t)\neq c_1,c_2,\ \mbox{for any}\ t\in(a,b), \\ 
		&\tilde{\Gamma}_u\ \mbox{is homotopic to}\ \partial B_\delta(c_i)
		\end{aligned}\right\rbrace
\]
for $i=1,2$ and $K_{c_i}^{p_1,p_2}$ as their $H^1$ weak-closure. 
Note that every path in the above spaces actually belongs to $\hat{K}_{(0,1)}^{p_1,p_2}\cup \hat{K}_{(1,0)}^{p_1,p_2}$ or, respectively, to their $H^1$ weak-closure (cf. \eqref{def:k_l} at page \pageref{def:k_l}). In particular, it is not difficult to see that, for any $p_1,p_2\in\mathcal{U}$
\begin{equation}\label{eq:K}
\begin{aligned}
	&\hat{K}_{c_1}^{p_1,p_2}\cup\hat{K}_{c_2}^{p_1,p_2}=\hat{K}_{(0,1)}^{p_1,p_2}\cup \hat{K}_{(1,0)}^{p_1,p_2} \\
	&K_{c_1}^{p_1,p_2}\cup K_{c_2}^{p_1,p_2}=K_{(0,1)}^{p_1,p_2}\cup K_{(1,0)}^{p_1,p_2}.
\end{aligned}
\end{equation}

Once this is clear, we can then prove a version of Theorem \ref{th:main1} in the presence of 2-centres and a unique minimal non-degenerate central configuration $\vt^*\in\Xi$.
\begin{theorem}\label{thm:per_2}
	Assume that $N=2$, $m=1$ and consider a function $V\in\mathcal{C}^2(\R^2\setminus\{c_1,c_2\})$ defined as in \eqref{def:potential} at page \pageref{def:potential} and satisfying \eqref{hyp:V}-\eqref{hyp:V_al}. There exists $\tilde{h}>0$ such that, for every $h\in(0,\tilde{h})$, $n\in\N_{\geq1}$ and $(c_{j_0},\ldots,c_{j_{n-1}})\in\mathcal{B}^n$, there exists a periodic collision-less and self-intersections-free solution $x=x(c_{j_0},\ldots,c_{j_{n-1}};h)$ of the equation $\ddot{x}=\nabla V(x)$ with energy $-h$, which depends on $(c_{j_0},\ldots,c_{j_{n-1}})$ in this way: there exists $\tilde{R}=\tilde{R}(h)>0$ such that the solution $x$ crosses $2n$ times the circle $\partial B_{\tilde R}$ in one period, at times $(t_s)_{s=0}^{2n-1}$, in such a way that, for any $k=0,\ldots,n-1$:
	\begin{itemize}
		\item in the interval $(t_{2k},t_{2k+1})$ the solution stays outside $B_{\tilde R}$ and there exists a neighbourhood $\mathcal{U}=\mathcal{U}(\tilde R e^{i\vt^*})$ on $\partial B_{\tilde{R}}$ such that
		\[
		x(t_{2k}),x(t_{2k+1})\in\mathcal{U};
		\]
		\item in the interval $(t_{2k+1},t_{2k+2})$ the solution stays inside $B_{\tilde R}$ and it is a reparametrization of a minimizer of the Maupertuis' functional in the space $\hat{K}_{c_{j_k}}^{x(t_{2k+1}),x(t_{2k+2})}$.
	\end{itemize}
\end{theorem}
\begin{proof}
	Once the notations in the statement are adopted, up to choose a smaller neighbourhood $\mathcal{U}$, the proof goes exactly as the proof of Theorem \ref{th:main1} (see page \pageref{proof:per}). Indeed, the glueing technique of Section \ref{sec:glueing} applies also when $N=2$ and $m=1$, since Theorem \ref{thm:outer_dyn} holds under such assumptions, while the proof of Theorem \ref{thm:inner_dyn_wind} does not change if one works with the spaces $\hat{K}_{c_1},\hat{K}_{c_2}$ instead of working with winding numbers, also in light of \eqref{eq:K}.
\end{proof}

It is left to prove the semi-conjugation of the 2-centre problem with a suitable shift space. As usual, we consider the set of bi-infinite sequences $\mathcal{B}^\Z$ of symbols $c_1,c_2$. The Bernoulli right shift on $\mathcal{B}^\Z$ is then defined as
\[
\begin{aligned}
	T_r\colon &\mathcal{B}^\Z\to \mathcal{B}^\Z \\
	&(c_m)\mapsto T_r((c_m))\uguale (c_{m+1}),
\end{aligned}
\]
and, also in this case, the proof of Theorem \ref{thm:symb_dyn2} goes as in Subsection \ref{subsec:sym_dyn_nocoll}, making use of the previous notations and of Theorem \ref{thm:per_2}.

\begin{remark}
	Let us consider for a while the 2-centre problem driven by isotropic (purely radial) potentials as in \cite{ST2012}. It is well-known that this problem is completely integrable and thus, as expected, the techniques of this section do not apply in this case. Indeed, due to the spherical symmetry of isotropic potentials, every direction is a central configuration and this allows to choose two points $p_1,p_2$ all around $\partial B_R$. This is the reason why it is not possible to distinguish between the result of a functional minimization in $K_{c_1}^{p_1,p_2}$ or in $K_{c_2}^{p_1,p_2}$ (as above), or a minimization in $K_{(0,1)}^{p_1,p_2}$ or $K_{(1,0)}^{p_1,p_2}$ (see Theorem \ref{thm:inner_dyn_wind} or \cite{ST2012}). In fact, moving $p_1$ and $p_2$ around $\partial B_R$, it is always possible to transform a path in $K_{(0,1)}^{p_1,p_2}$ into a path in $K_{(1,0)}^{p_1,p_2}$. This fact, as expected, destroys the symbolic dynamics.
\end{remark}

\appendix

\section{Variational principles} 
\label{app:var}

This short appendix collects some general variational results which involve the different functionals used along this work. In particular, we want to establish some important relations between the Maupertuis' functional, the Lagrange-action functional and the Jacobi-length functional.

Consider an open set $\Omega\sset\R^2$, a potential $V\in\mathcal{C}^2(\Omega)$ and introduce the second order system
\begin{equation}\label{app:var_eq}
	\ddot x=\nabla V(x).
\end{equation}
Note that \eqref{app:var_eq} has Hamiltonian structure and thus we can look for those solutions $x$ which preserve a fixed energy $h\in\R$ along their motion, i.e., such that
\begin{equation}\label{app:var_en}
	\frac12|\dot{x}(t)|^2-V(x(t))=h
\end{equation}
and, in particular, such solutions will be confined inside the open Hill's region 
\[
\mathring{\mathcal{R}}_h\uguale\{x\in\Omega:\,V(x)+h>0\}.
\]
For $T>0$ and $x\in H^1([0,T];\R^2)$ we define the Lagrange-action functional as
\[
\mathcal{A}_T(x)\uguale\int_0^T\left(\frac12|\dot{x}(t)|^2+V(x(t)\right)\,dt
\]
and it is well known that the \emph{Least Action Principle} affirms that a solution $x\colon[0,T]\to\Omega$ of \eqref{app:var_eq} corresponds to a critical point of $\mathcal{A}_T$.

In this work we have mainly used the Maupertuis' functional, which in this setting reads
\[
\mathcal{M}_h(u)\uguale\frac12\int_0^1|\dot{u}|^2\int_0^1(h+V(u))
\]
and it is differentiable in the space 
\[
H_h\uguale\{u\in H^1([0,1];\Omega):\,V(u)+h>0\}.
\]
An equivalent result of the Least Action Principle can be stated for $\mathcal{M}_h$, the so-called \emph{Maupertuis' Principle} (see \cite{AC-Z}, but also Theorem \ref{thm:maupertuis} for a version concerning fixed-ends problems), which also provides a first relation between the critical points of the two functionals. 
\begin{theorem}\label{thm:app_var_maupertuis}
	Let $u\in H_h$ be a critical point of $\mathcal{M}_h$ at a positive level. Define $\omega>0$ such that
	\[
	\omega^2\uguale\frac{\int_0^1(h+V(u))}{\frac12\int_0^1|\dot{u}|^2}.
	\]
	Then, the function $x(t)\uguale u(\omega t)$ solves \eqref{app:var_eq}-\eqref{app:var_en} in the interval $[0,T]$, with $T\uguale 1/\omega$. 
	
	As a consequence, the function $x$ is also a critical point of $\mathcal{A}_{1/\omega}$ in the space $H^1([0,1/\omega];\Omega)$.
\end{theorem}

The next result refines the correspondence between critical points of $\mathcal{M}_h$ and $\mathcal{A}_T$, showing that in particular a critical point of the Maupertuis' functional minimizes the action for every time $T>0$.

\begin{proposition}\label{prop:app}
	Let $u\in H_h$ be a critical point of $\mathcal{M}_h$ at a positive level. If, for every $T>0$, we define
	\[
	x_T(t)\uguale u\left(\frac{t}{T}\right),\quad\mbox{for}\ t\in[0,T],
	\]
	then
	\[
	2\sqrt{\mathcal{M}_h(u)}=\mathcal{A}_{1/\omega}(x_{1/\omega})+\frac{h}{\omega}=\min\limits_{T>0}\left(\mathcal{A}_T(x_T)+Th\right).
	\]
\end{proposition}
\begin{proof}
	For every $T>0$, we can compute
	\[
	\begin{aligned}
		\mathcal{A}_T(x_T)+Th&=\int_0^T\left(\frac12|\dot{x}_T(t)|^2+V(x_T(t))+Th\right)\,dt \\
		&=\int_0^T\left(\frac{1}{2T^2}|\dot{u}(t/T)|^2+V(u(t/T))+h\right)\,dt \\
		&=\int_0^1\left(\frac{1}{2T}|\dot{u}(s)|^2+TV(u(s))+Th\right)\,ds.
	\end{aligned}
	\]
	Since $u$ is fixed, the previous quantity depends only on $T$ and it is easy to check that it attains its minimum at 
	\[
	T=\left(\frac{\int_0^1|\dot{u}|^2}{2\int_0^1(h+V(u))}\right)^{1/2}=\frac{1}{\omega}.
	\]
\end{proof}

The previous result also shows a well-known property of the Maupertuis' functional, i.e., that this functional is invariant under time reparametrizations. However, $\mathcal{M}_h$ is not additive, and this suggests the introduction of another geometric functional. The Jacobi-length functional is defined as
\[
\mathcal{L}_h(u)=\int_0^1|\dot{u}(t)|\sqrt{h+V(u(t))}\,dt,
\]
for every $u\in H_h$. Note that Theorem \ref{thm:app_var_maupertuis} could be rephrased for $\mathcal{L}_h$ and thus classical solutions will be suitable reparametrizations of critical points of $\mathcal{L}_h$ (see for instance \cite{MoMoVe2012}). Moreover, the Jacobi-length functional is also a fundamental tool in differential geometry since $\mathcal{L}_h(u)$ is exactly the Riemannian length of the curve parametrized by $u$ with respect to the Jacobi metric
\[
g_{ij}(x)\uguale (-h+V(x))\delta_{ij},
\]
$\delta_{ij}$ being the Kronecker delta. As the Maupertuis' functional, $\mathcal{L}_h$ is invariant under time reparametrizations and, being a length, it is also additive. 

Note that, if $u\in H_h$, the Cauchy-Schwarz inequality easily gives
\[
\mathcal{L}_h(u)\leq\sqrt{2\mathcal{M}_h(u)},
\]
with the occurrence of the equality if and only if the quotient
\[
\frac{|\dot{u}(t)|^2}{V(u(t))+h}
\]
is constant for a.e. $t\in[0,1]$. This shows that $\mathcal{M}_h$ and $\mathcal{L}_h$ share the same critical points $u$ such that $\mathcal{M}_h(u)>0$. This is sufficient to give the (easy) proof of the next result.

\begin{proposition}\label{prop:app1}
	Let $u\in H_h$ be a critical point of $\mathcal{M}_h$ at a positive level, let $\omega$ be defined as in Theorem \ref{thm:app_var_maupertuis} and let $x_T$ be defined for every $T>0$ as in Proposition \ref{prop:app}. Then, the following chain of equalities hold
	\[
	\mathcal{A}_{1/\omega}(x_{1/\omega})+\frac{h}{\omega}=2\sqrt{\mathcal{M}_h(u)}=\frac{2}{\sqrt{2}}\mathcal{L}_h(u).
	\]
	We can say that, up to constant factors and time re-parametrizations, the three functionals coincide on the non-constant critical points of $\mathcal{M}_h$.
\end{proposition}

\bibliography{vivinabibliog}
\bibliographystyle{plain}

\end{document}